\documentclass{article}

\usepackage[final]{neurips_2023-hack}

\usepackage[utf8]{inputenc} % allow utf-8 input
\usepackage[T1]{fontenc}    % use 8-bit T1 fonts
\usepackage{hyperref}       % hyperlinks
\usepackage{url}            % simple URL typesetting
\usepackage{booktabs}       % professional-quality tables
\usepackage{amsfonts}       % blackboard math symbols
\usepackage{nicefrac}       % compact symbols for 1/2, etc.
\usepackage{microtype}      % microtypography
\usepackage[para,online,flushleft]{threeparttable}
\usepackage{colortbl}

\usepackage{amsthm, amsmath, amssymb, mathtools} % math package
\usepackage{cleveref} % clever reference

\usepackage{algorithm}
\usepackage{algorithmic}

\usepackage{float}
\usepackage{subcaption}

\usepackage{enumitem}
\setlist[enumerate]{leftmargin=5mm}

\usepackage{xspace}
\usepackage{color}
\usepackage[dvipsnames]{xcolor}
\usepackage{pifont}
\newcommand{\cmark}{{\color{PineGreen}\ding{51}}}%
\newcommand{\xmark}{{\color{BrickRed}\ding{55}}}%

\newcommand{\algname}[1]{{\color{ForestGreen}\small\sf#1}\xspace}
\newcommand{\algnamesmall}[1]{{\color{ForestGreen}\scriptsize\sf#1}\xspace}

\newcommand{\norm}[1]{\left\| #1 \right\|}
\newcommand{\ExpBr}[1]{\mathbb{E}\left[#1\right]}

\newcommand{\cZ}{\mathcal{Z}}

\newcommand{\cJ}{\mathcal{J}}

\newcommand{\sumin}{\sum_{i=1}^n}

\newcommand{\avein}{\frac{1}{n}\sum_{i=1}^n}

\newcommand{\RR}{\mathbb{R}}
\newcommand{\cC}{{\cal C}}
\newcommand{\cO}{{\cal O}}

\newcommand{\bA}{\mathbf{A}}

\newcommand{\bS}{\mathbf{S}}

\newcommand{\del}[1]{}

\let\la=\langle
\let\ra=\rangle

\newcommand{\eqdef}{\stackrel{\text{def}}{=}}

\newcommand{\R}{\mathbb{R}}

\newcommand{\squeeze}{}
\newtheorem{assumption}{Assumption}
\newtheorem{lemma}{Lemma}

\newtheorem{theorem}{Theorem}
\newtheorem{fact}{Fact}

\newtheorem{example}{Example}

\theoremstyle{plain}

\theoremstyle{definition}

\newtheorem{definition}[theorem]{Definition} % assumptions, lemmas, theorems, etc.
\usepackage[colorinlistoftodos,bordercolor=orange,backgroundcolor=orange!20,linecolor=orange,textsize=scriptsize]{todonotes}

\newcommand{\peter}[1]{\todo[inline]{\textbf{Peter: } #1}} 

\newcommand{\cI}{{\cal I}}

\newcommand{\supp}{{\rm supp}}
\newcommand{\topk}[1]{{\color{PineGreen}{\sf Top}{\sf#1}}}
\newcommand{\randk}[1]{{\color{PineGreen}{\sf Rand}{\sf#1}}}
\newcommand{\permk}[1]{{\color{PineGreen}{\sf Perm}{\sf#1}}}
\newcommand{\xct}{{\color{red}c^t}}
\newcommand{\xc}{{\color{red}c}}
\newcommand{\xr}{{\color{blue}r}}

\newcommand{\Lpl}{{L_+}}
\newcommand{\Lplsq}{{L_+^2}}
\newcommand{\Lplt}{{L_+^t}}

\title{Error Feedback Shines when Features are Rare}
\author{%
 Peter Richt\'{a}rik \\
KAUST${}^\dagger$\\
SDAIA-KAUST AI${}^\ddagger$\\
Saudi Arabia \\
  \And
	Elnur Gasanov \\
	KAUST\\
	Saudi Arabia \\
  \And
	Konstantin Burlachenko\\
	KAUST\\
	SDAIA-KAUST AI \\
	Saudi Arabia \\
}

\begin{document}

\maketitle

\begin{abstract}
We provide the first proof that gradient descent (\algname{GD}) with greedy sparsification ($\topk{K}$)  and error feedback (\algname{EF}) can obtain better communication complexity than vanilla \algname{GD} when solving the distributed  optimization problem $\min_{x\in \R^d} \{f(x)=\frac{1}{n}\sum_{i=1}^n f_i(x)\}$, where $n$ = \# of clients, $d$ = \# of features, and $f_1,\dots,f_n$ are smooth nonconvex functions.  
Despite intensive research since 2014  when \algname{EF} was first proposed by Seide et al., this problem remained open until now. %Surprisingly, this superior performance holds in and is facilitated by a heterogeneous data regime. 
Perhaps surprisingly, we show that \algname{EF} shines in the regime when features are rare, i.e., when each feature is present in the data owned by a small number of clients only. To illustrate our main result, we show that  in order to find a random vector $\hat{x}$ such that $\norm{\nabla f(\hat{x})}^2 \leq \varepsilon$ in expectation, \algname{GD} with the $\topk{1}$ sparsifier and \algname{EF} requires  
$\cO\left( \left(L +   \xr\sqrt{ \frac{\xc}{n} \min \left\{ \frac{\xc}{n}  \max_i L_i^2, \frac{1}{n}\sum_{i=1}^n L_i^2 \right\}}  \right) \frac{1}{\varepsilon} \right) $ bits to be communicated by each worker to the server only, where $L$ is the smoothness constant of $f$, 
$L_i$ is the smoothness constant of $f_i$, $\xc$ is the maximal number of clients owning any feature ($1\leq \xc\leq n$), and $\xr$ is the maximal number of features owned by any client ($1\leq \xr \leq d$). Clearly, the communication complexity improves as $\xc$  decreases (i.e., as features become more rare), and can be much better than the $\cO(\xr L \frac{1}{\varepsilon})$ communication complexity of \algname{GD} in the same regime.

%We recover the prior best result $\cO\left( \left(L +   \sqrt{   \frac{d}{n}\sum_{i=1}^n L_i^2 }  \right) \frac{1}{\varepsilon} \right) $ of Richt\'{a}rik et al (2021) as a special case for the worst case choices of the client-feature sparsity parameters: $\xc=n$ and $\xr=d$.

\end{abstract}

\section{Introduction}

In recent decades, the field of machine learning has undergone significant growth and development, presenting numerous opportunities for innovation and advancement. As a result, breakthroughs have been made in various areas such as computer vision~\citep{krizhevsky2012, resnet, diffusion}, natural language processing~\citep{transformer, gpt}, reinforcement learning~\citep{atari, SuttonBartoRL}, healthcare~\citep{healthcare}, and finance~\citep{dixon2020machine}. The surge in machine learning applications has also stimulated the development of associated optimization research. 

What unique characteristics has modern machine learning introduced to optimization research? i) Firstly, the design of modern models is inherently complex. For example, contemporary convolutional neural networks are constructed from multiple blocks consisting of diverse types of layers that are arranged hierarchically~\citep{deep_learning_review}. Consequently, these models are notably \textit{nonconvex} in nature~\citep{choromanska15}.  ii) Secondly, the quantity of data utilized during model training is so extensive that the use of a single computing device is no longer practical. Therefore, it is necessary to distribute the data across multiple computing resources, which raises the question of how to effectively coordinate \textit{distributed}~\citep{YANG2019278} model training. Another motivation for distributed training is derived from the federated learning framework~\citep{FL_overview}, where data is owned by users who are not willing to share it with each other in a cross-device federated learning scenario. In this case, the centralized algorithm must manage the training of several client devices. iii) Thirdly, modern machine learning models have a tendency to expand in size, often containing millions of parameters. For instance, in a distributed setting, during gradient descent, each device must transmit a dense gradient vector consisting of millions of parameters each round of communication~\citep{li2020acceleration}. This creates an unmanageable burden on the communication network. Hence, there is a need for techniques that can {\em reduce the number of bits transmitted} through communication channels while maintaining the convergence of the algorithm.

The first two points, namely the lack of convexity and the distributed setup, provide the rationale for focusing our investigation on the following optimization problem:

\begin{equation}\label{eq:main_problem}
	\squeeze \min \limits_{x \in \mathbb{R}^d} \left\{f(x) = \frac{1}{n}\sum \limits_{i=1}^n f_i(x)\right\},
\end{equation}

where $f_i: \mathbb{R}^d \rightarrow \mathbb{R}$ is a smooth nonconvex function representing the local loss function on client $i$, and $n \in \mathbb{N}$ is the number of participating clients in the training process, with $d \in \mathbb{N}$ representing the dimension of the model. The third point, which we address by  reducing the communication load via lossy compression of updates, is elaborated in the following section.

\section{Communication Compression and Error Feedback}

\subsection{Reducing communication in distributed learning}

A common approach to solving the optimization problem~\eqref{eq:main_problem} is to utilize Distributed Gradient Descent (\algname{DGD}), which takes the form:
\begin{equation}\label{eq:distributed_gd}
\squeeze	x^{t+1} = x^t -  \frac{\gamma^t}{n}\sum \limits_{i=1}^n\nabla f_i(x^t),
\end{equation}

where $\gamma^t > 0$ denotes the learning rate. The orchestration of this method occurs between a master node (a central node that holds no data and is responsible for aggregation actions) and $n$ clients. Prior to the start of the algorithm, an initial iterate $x^0 \in \RR^d$ and a learning rate $\gamma^0>0$ are selected. At iteration $t$, the master node broadcasts the current iterate $x^t \in \RR^d$ to the clients, each client then computes its gradient $\nabla f_i(x^t)$ and sends it back to the master. Finally, the master node aggregates all the gradients and uses them to perform the gradient descent step~\eqref{eq:distributed_gd}, updating the iterate to $x^{t+1}$. This process is repeated. 

While~\algname{DGD} is known to be an optimal algorithm for finding a stationary point with a minimum number of iterations on smooth nonconvex problems~\citep{nesterov2018lectures}, it poses a significant challenge to the communication network. At every communication round, \algname{DGD} transmits dense gradients of $d$ dimensions to the master. As previously mentioned, this communication load is unacceptable in many practical scenarios. One way to resolve this problem is to apply a contractive compressor to the communicated gradients.
\begin{definition}%[Contractive compressor]
A (possibly randomized) mapping $\cC: \RR^d \rightarrow \RR^d$ is called a contractive compressor if there exists a constant $\alpha \in (0, 1]$, known as the \textit{contraction parameter}, such that 
\begin{equation}
\ExpBr{\|\cC(x) - x\|^2} \leq (1 - \alpha) \|x\|^2, \quad \forall x\in \R^d.
\end{equation}
\end{definition}
An immensely popular contractive compressor is the $\topk{K}$ operator. Given a vector $x \in \RR^d$, this operator retains only the $k$ elements with the largest magnitudes and sets the remaining $d-k$ elements to zeroes. It is well-known that the $\topk{K}$ operator is contractive with $\alpha = \frac{k}{d}$. 

%In practical scenarios, applying this operator allows for sending only $k$ elements through the uplink network instead of all $d$ elements required in the case of \algname{DGD}, leading to $d:K$ speedup in uplink communication time per iteration. 

The next step in reducing the communication load within \algname{DGD},  is to implement workers-to-server  communication compression. This modification of \algname{DGD} is known as  Distributed Compressed Gradient Descent (\algname{DCGD}), performing iterations
\begin{equation}
\squeeze	x^{t+1} = x^t -  \frac{\gamma^t}{n}\sum \limits_{i=1}^n \cC_i^t(\nabla f_i(x^t)),	
\end{equation}
where $\cC_i^t$ is the compressor used by  client $i$ at iteration $t$. 

\algname{DCGD} has been shown to converge to stationary points using certain compressors~\citep{sigma_k, khirirat2018distributed}. However, greedy contractive compressors such as \topk{K}, which are often preferred in practical applications~\citep{szlendak2022permutation}, are not among them, and their practical superiority is not explained by any existing theoretical results. Furthermore, negative results suggest that \algname{DCGD} with greedy contractive compressors such as \topk{K} can experience exponential divergence in the distributed setting~\citep{beznosikov2020biased, Karimireddy_SignSGD}, even on simple convex quadratics. Hence, in order to safely apply contractive compressors in a multi-device communication network, a fix is needed. Fortunately, a fix exists: it is generically known as Error Feedback.

\subsection{Evolution of error feedback}

The concept of Error Feedback (EF), also known as Error Compensation, was initially proposed by~\citet{Seide2014} as a heuristic approach to fix the divergence issues of \algname{DCGD}. The first version of EF, which for the purposes of this paper we shall denote as \algname{EF14}, is structured as follows:
\[\squeeze v_i^t =  \cC_i^t(e_i^t + \gamma^t \nabla f_i(x^t)), \qquad x^{t+1} = x^t - \frac{1}{n}\sum \limits_{i=1}^n v_i^t, \qquad e_i^{t+1} = e_i^t + \gamma^t \nabla f_i(x^t) -  v_i^t.
\]

Here, $\cC_i^t$ represents a contractive compressor utilized by client $i$ at iteration $t$, and  $e_i^t$ is a memory vector that preserves all elements that were not transmitted by client~$i$ in the preceding iterations. In this approach, at iteration $t$, the method sends compressed gradient compensated by the memory vector. The memory vector is then updated and the process is repeated.

The first attempts to establish a theoretical basis for \algname{EF14} focused on the single-node scenario~\citep{Stich-EF-NIPS2018, Alistarh-EF-NIPS2018}, which is, however, detached from the practical utility of the mechanism that is meaningful only in the distributed setting. An analysis in the distributed setting was executed by~\citep{beznosikov2020biased} for strongly convex problems. However, a key limitation of their result was that the (expected) linear rate was only achievable in the homogeneous over-parameterized setting. Subsequent efforts aimed to relax these assumptions and allow for a fast rate even under data heterogeneity, but the \algname{EF14}  framework seem too elusive to  yield vastly improved results. For example, \citet{Koloskova2019DecentralizedDL}  established convergence of \algname{EF14} under the assumption of bounded gradients. The authors demonstrate that the method requires $T = \cO(\varepsilon^{-3/2})$ iterations to attain an  $\varepsilon$-accuracy in terms of the squared gradient norm. However, this rate is worse than the $T = \cO(\varepsilon^{-1})$ rate of vanilla \algname{DGD}, and under stronger assumptions.
\begin{algorithm}
	\begin{algorithmic}[1]
		\STATE {\bfseries Input:} $x^0 \in \RR^d$; $g_1^0\in \R^d_1,\dots,g_n^0 \in \R^d_n$ (as defined in equation~\eqref{eq:R_d_n_definiton}); stepsize $\gamma>0$; sparsification levels $K_1,\dots,K_n \in [d]$; number of iterations $T > 0$
		\STATE {\bfseries Initialize:} $g^0 = \frac{1}{n}\sum_{i=1}^n g_i^0$
		\FOR{$t = 0, 1, 2, \dots, T - 1 $}
		\STATE  $x^{t+1} = x^t - \gamma g^t$		
		\STATE Broadcast  $x^{t+1}$ to the clients
		\FOR{$i = 1, \dots, n$ {\bf on the clients in parallel}} 
		\STATE $g_i^{t+1} = g_i^t + \topk{K_i}(\nabla f_i(x^{t+1}) - g_i^t)$ \label{alg_line:g_update_step}
		\STATE Send $g_i^{t+1}$ to the server
		\ENDFOR 
		\STATE $g^{t+1} = \avein g_i^{t+1}$
		\ENDFOR
		\STATE {\bfseries Output:} $x^T$
	\end{algorithmic}
	\caption{\algname{EF21}: Error Feedback 2021 with the $\topk{K_i}$ compressor \citep{EF21}}
	\label{alg:ef21}
\end{algorithm}

Recently, \citet{EF21} re-engineered the EF technique and proposed their \algname{EF21} method. This algorithm, presented in a specific version with the $\topk{K_i}$ compressors in \Cref{alg:ef21}, achieves the optimal number of iterations of $O(\varepsilon^{-1})$ for smooth nonconvex  problems, without invoking any strong assumptions. This paper represents a breakthrough in realizing the complete potential of EF for distributed learning applications since \algname{EF21}  achieves a better convergence rate than \algname{EF14},  under weaker assumptions, and exhibits better performance in practice. While it may seem that \algname{EF21} constitutes the ultimate culmination of the Error Feedback story, we argue that this is far from true! And this is the starting point of our work. We elaborate on this in the next section.

\newpage
\section{Error Feedback: Discrepancy Between Theory and Practice}

Our work is motivated through {\em three observations} about the unreasonable effectiveness of  \algname{EF21}.
\subsection{Observation 1: In practice, \algname{EF21} can be extremely good!}

Let us start by looking at the practical performance of \algname{EF21} using the  \topk{K} compressor, contrasted with the performance of the current best method for solving \eqref{eq:main_problem} in the smooth nonconvex regime in terms of theoretical communication complexity: the \algname{MARINA} method of \citet{MARINA}.
In \Cref{fig:MARINAvsEF21}, which we adopted from from~\citep{szlendak2022permutation},  \algname{EF21} is compared to two variants of \algname{MARINA}: with the \randk{K} and \permk{K} compressors used by the clients. While \topk{K} is a {\em greedy} mechanism, performed by each client without any regard for what the other clients do, \permk{K} is the exact opposite: it is a {\em collaborative} compression mechanism. 

\begin{figure}[H]
\centering
\includegraphics[width=\textwidth]{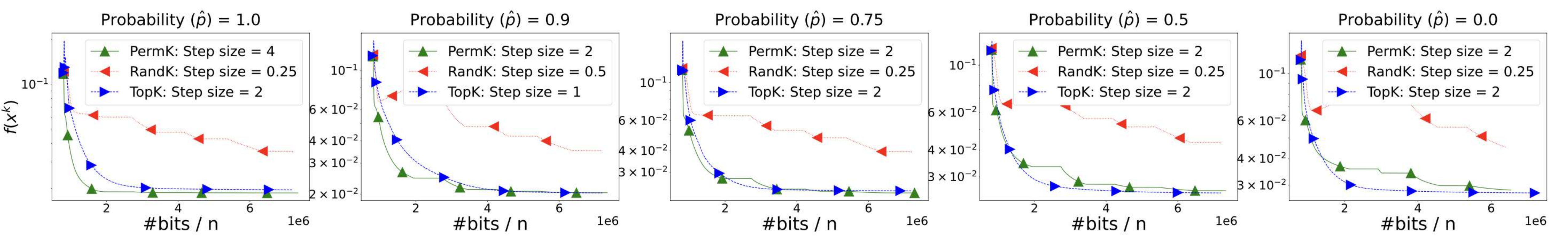}
\caption{Comparison of three algorithms on the encoding learning task for the MNIST dataset~\citep[Figure 2]{szlendak2022permutation}.}
\label{fig:MARINAvsEF21}
\end{figure}

It is clear from \Cref{fig:MARINAvsEF21} that \algname{EF21}+\;\topk{K}  has  similar performance to \algname{MARINA}+\;\permk{K}, and both are much faster than \algname{MARINA} with the ``naive'' \randk{K} compressor.

\subsection{Observation 2: In theory, ~\algname{EF21} is not better than~\algname{DGD}!}

However---and we believe that this is very important---while the theoretical communication complexity of \algname{MARINA} is significantly better than that of \algname{DGD} \citep{MARINA}, the theoretical communication complexity of \algname{EF21} at best matches the complexity of \algname{DGD}!  
%The current state-of-the-art EF algorithm, \algname{EF21}~\citep{EF21}, has a critical drawback in terms of communication complexity, which renders it no better than \algname{DGD}. 
To state this formally, let us revisit the assumptions \citet{EF21} used to perform their  analysis of \algname{EF21}.

\begin{assumption}\label{as:smooth}
	The function $f$ is $L$-smooth, i.e., there exists $L>0$ such that
	\begin{equation}\label{eq:smoothness_def}
		\norm{\nabla f(x) - \nabla f(y)} \leq L \norm{x - y}, \quad \forall x, y \in \R^d.
	\end{equation}
\end{assumption}
\begin{assumption}\label{as:L_i} The functions $f_i$ are $L_i$-smooth,  i.e., there exists $L_i>0$ such that
	\begin{equation}\label{eq:L_i}
		\norm{\nabla f_i(x) - \nabla f_i(y)} \leq L_i \norm{x - y}, \quad \forall x, y \in \R^d.
	\end{equation}
\end{assumption}
\begin{assumption}\label{as:lower_bounded}
	There exists $f^\ast \in \RR$ such that  that $f(x) \geq f^\ast$ for all $x\in\R^d$.
\end{assumption}
The following observation formalized the above claim.

\begin{lemma}\label{lm:alpha_1_the_best}
	Under Assumptions~\ref{as:smooth}, \ref{as:L_i}, and~\ref{as:lower_bounded}, the optimal theoretical communication complexity of \Cref{alg:ef21} obtained by \citet{EF21} is achieved when the parameter $K$ in the $\topk{K}$ compressor equals the dimension $d$. 
\end{lemma}

So, to summarize,  in theory, \algname{EF21} at its best reduces to vanilla \algname{DGD}, but in practice,  \algname{EF21}  can be orders of magnitude better than \algname{DGD}. Why is error feedback so unreasonably effective? Despite about a decade of research on the Error Feedback mechanism since the pioneering work of \citet{Seide2014}, and steady advances in the field and our theoretical understanding, there is clearly still much to be explored regarding the communication complexity advantages it offers.

\subsection{Observation 3: \algname{EF21} has absolutely no scaling in the homogeneous data regime!}

Our final observation is that  \algname{EF21} performs remarkably poorly in the fully data homogeneous regime, i.e., when $f_1 = f_2 = \dots = f_n =  f$. For simplicity, assume that all initial gradient estimators $g_1^0,\dots,g_n^0$ are identical. In this case, running \algname{EF21} with the $\topk{K}$ compressor in the distributed ($n>1$) setting is equivalent to running it in the single node ($n=1$) setting! That is, \algname{EF21}  has no parallel scaling at all!
Specifically, if at iteration $t$ we have $g_i^t\equiv g'$ for all $i$, then $$g_i^{t+1} = g^t_i + \topk{K}(\nabla f_i(x^t) - g_i^t) = g' + \topk{K}(\nabla f(x^t) - g')$$ for all $i$, meaning that all $g_i^{t+1}$ coincide, again. Note that the average of these estimators does not depend on $n$ and, as a result, \algname{EF21} generates the same chain of iterates $x^t$ for any value of $n$.

\iffalse
Error Feedback (\algname{EF}), first proposed by \citet{Seide2014}, is an immensely popular convergence stabilization technique for distributed gradient-type methods enhanced with communication compression via contractive operators such as random and greedy quantization and sparsification \citep{Alistarh-EF-NIPS2018,beznosikov2020biased}.

In practice, gradient methods enjoy a significantly smaller communication footprint when used together with communication compression and \algname{EF}. However,  there are no theoretical results explaining why this might be the case for greedy contractive compression operators such as $\topk{K}$. Indeed, prior to the recent work of \citet{EF21} who designed a new \algname{EF} mechanism called \algname{EF21}, all theoretical results for \algname{EF} with $\topk{K}$ predicted worse communication complexity than that of vanilla gradient descent. Although \algname{EF21} matches the communication complexity of gradient descent, its empirical communication complexity is vastly superior to that of gradient descent, and this phenomenon remains unexplained.

\subsection{The problem}

Consider solving the optimization problem 
\begin{equation}\label{eq:main} \min_{x\in \R^d} \left\{f(x)\eqdef \frac{1}{n}\sum_{i=1}^n f_i(x)\right\}, \end{equation}
where $n$ is the number of clients, vector $x\in \R^d$ represents the $d$ features/parameters/weights of a model that need to be learned, and $f_i:\R^d \to \R$ is the loss of model parameterized by $x$ on the data stored at client $i$.
\fi
\subsection{Summary of contributions}

The above three observations point to deep issues in our theoretical understanding of EF in general, and its most recent and best performing incarnation \algname{EF21} in particular. However, our last observation also offers a hint at a possible resolution. Indeed, since \algname{EF21} cannot possibly perform well on data homogeneous problems, we conjecture that its practical superiority is due to a certain ``favorable'' type of data heterogeneity.  This is perhaps counter-intuitive because heterogeneous problems are typically considered to be much harder than homogeneous problems~\citep{FL_overview}.  Specifically, our contributions are: 

\phantom{XX} {\bf a) Sparsity rules \& the first time error feedback beats gradient descent.}  We present a refined theoretical analysis of the \algname{EF21} algorithm (with the greedy $\topk{K}$ compressor) demonstrating that, under certain conditions on the {\em sparsity parameters} $\xc$ and $\xr$ associated with the problem (defined in~\eqref{eq:sigma}), \algname{EF21} can surpass \algname{DGD} in terms of \textit{theoretical communication complexity}.  This is the first result in the vast literature on error feedback of this type. 
In particular, in~\Cref{thm:convergence_separate}, we provide a stronger convergence rate for \algname{EF21} in terms of its communication complexity. For \algname{EF21} with the $\topk{1}$ compressor,  for example, we establish an $\cO((L + \xr \sqrt{\frac{\xc}{n}} \Lpl) \frac{1}{\varepsilon})$ communication complexity, where $\Lpl$ is defined in~\Cref{as:L+}. In the regime when $\sqrt{\frac{\xc}{n}} \Lpl \leq L$, which holds when either $\xc$ is small or $n$ is large, \algname{EF21} is   better than \algname{DGD}.

\phantom{XX}  {\bf b) Experimental validation.} We conduct rigorous toy experiments on linear regression functions to validate our theoretical findings. In one of our experiments, we redefine $\xc$ as an adaptive parameter $\xct$ satisfying the inequality $\norm{g^t -\nabla f(x^t)}^2 \leq \frac{\xct}{n} G^t$ presented in~\Cref{lem:98y89fhd_8fd}. We also introduce a heuristically defined adaptive stepsize $\gamma(\xct)$, which demonstrates good performance in a logistic regression experiment on non-sparse data.

We believe that this paper represents a milestone in advancing our understanding of Error Feedback algorithms, and lays the groundwork for further advancements in the field.

%Better than \algname{GD} if any of the following equivalent conditions holds: $$ \sqrt{\frac{\xc}{n}} \Lpl \leq L \quad \Leftrightarrow \quad \Lpl \leq \sqrt{\frac{n}{\xc}} L \quad \Leftrightarrow \quad \xc \Lplsq \leq n L^2 \quad \Leftrightarrow \quad \xc \leq n \frac{L^2}{ \Lplsq}.$$
%Since $\frac{\norm{\nabla f(x)-\nabla f(y)}^2}{\norm{x-y}^2} \leq L^2$ and $\frac{\norm{\nabla f(x)-\nabla f(y)}^2}{\norm{x-y}^2} \leq M^2$}
%
%\peter{$\xc \Lplsq \leq \xc L_{\rm qm}^2 =\xc n  \frac{ L_{\rm qm}^2}{n} =\xc n^2 \frac{1}{n}\sum_{i=1}^n \frac{L_i^2}{n^2} \leq \xc n^2 L^2  $ (BAD)}
%
%\peter{$\xc \Lplsq \leq\xc \frac{\xc}{n}L_{\max}^2 = \xc^2 n\frac{L_{\max}^2}{n^2} \leq  \xc^2 n L^2$}

\begin{table}[t]
\caption{Communication complexity of \algname{DGD} with the $\topk{1}$ compressor and Error Feedback (either \algname{EF14} due to \citet{Seide2014}, or \algname{EF21} due to \citet{EF21}). $L$ = smoothness constant of $f$;  $L_i$ = smoothness constant of $f_i$; $L_{\max} = \max_i L_i$; $\tilde{L}^2 = \frac{1}{n}\sum_{i=1}^n L_i^2$; $\Lpl$ is the smoothness constant defined in Assumption~\ref{as:L+}. Note that $\Lpl \leq \tilde{L}$ and $L \leq \tilde{L}$.}
\label{tbl:main}  		
\begin{center}	
\begin{threeparttable}		
 \scriptsize
\begin{tabular}{ccccccc}
			%    \hline 
			\begin{tabular}{c} \bf Base \\ \bf method \end{tabular} & \bf Compressor & \begin{tabular}{c} \bf Error feedback \\ \bf mechanism \end{tabular} & \bf \begin{tabular}{c} \bf Communication \\ \bf complexity \end{tabular} &  \bf \begin{tabular}{c} Better  \\ than \algnamesmall{DGD}?\end{tabular}  \\
%%%%%%%%%%%%%%%%%%%%%%%%%%%		
			\hline 
			\algnamesmall{DGD}& --- & --- &  $\cO(\xr L \varepsilon^{-1})$ & =  \\					
%%%%%%%%%%%%%%%%%%%%%%%%%%%		
			\hline 
			\algnamesmall{DGD}& $\topk{1}$ & --- &  diverges\tnote{(i)}  & \xmark \\			
%%%%%%%%%%%%%%%%%%%%%%%%%%%		
			\hline 
			\algnamesmall{DGD} & $\topk{1}$ &\algnamesmall{EF14} \citep{Seide2014}&  $\cO\left(\varepsilon^{-3/2}\right)$~\tnote{(iii)} & \xmark \\			
%%%%%%%%%%%%%%%%%%%%%%%%%%%		
			\hline 
			\algnamesmall{DGD} & $\topk{1}$ & \algnamesmall{EF21} \citep{EF21}&  $\cO\left( \left(L + \xr \tilde{L} \right) \varepsilon^{-1}\right)$~\tnote{(iv)}  & \xmark \\
%%%%%%%%%%%%%%%%%%%%%%%%%%%	
			\hline 
			\rowcolor{gray!20}
			\algnamesmall{DGD} & $\topk{1}$ & \algnamesmall{EF21} \citep{EF21}& $\cO\left( \left(L +   \xr \sqrt{\frac{\xc}{n}} \Lpl  \right) \varepsilon^{-1}\right)$~\tnote{(v)} & \cmark\tnote{(ii)} \\
%%%%%%%%%%%%%%%%%%%%%%%%%%%		
			\hline 
			\rowcolor{gray!20}
			\algnamesmall{DGD} & $\topk{1}$ & \algnamesmall{EF21} \citep{EF21} & $\cO\left( \left(L +   \xr \sqrt{\frac{\xc}{n} \min \left\{ \frac{\xc}{n} L_{\max}^2, \tilde{L}^2 \right\}}  \right) \varepsilon^{-1} \right)$~\tnote{(v)} & \cmark\tnote{(ii)} \\
%%%%%%%%%%%%%%%%%%%%%%%%%%%					
			\hline        
		\end{tabular}   

\begin{tablenotes}
\item[(i)] See \citet{beznosikov2020biased}; 
\item[(ii)] If $\sqrt{\frac{\xc}{n}} \Lpl \leq L$;  
\item[(iii)] Proved by \citet{Koloskova2019DecentralizedDL}under bounded gradient assumption; 
\item[(iv)] Proved by \citet{EF21}; 
\item[(v)] {\bf New results proved in this paper.}
\end{tablenotes}		
\end{threeparttable}	
\end{center}		   
\end{table}

\section{Rare Features} \label{sec:sparsity}
%%%%%%%%%%%%%%%%%%%%
%%%%%%%%%%%%%%%%%%%%
%%%%%%%%%%%%%%%%%%%%
We begin our narrative by delineating the key features of the sparsity pattern in~$f_i$. Let us define 
$$
\cZ \eqdef \left\{(i, j) \subseteq [n]\times [d] \ | \ [\nabla f_i(x)]_j = 0 \ \forall x \in \RR^d \right\},
$$or, equivalently, $\cZ \subseteq [n]\times [d]$ represents the set of pairs $(i,j)$, where $f_i(x)$ does not depend on $x_j$.  For instance, if all functions depend on all variables, then $\cZ = \emptyset$. Further, we define:
\begin{equation}\label{eq:cI_cJ_defs}
\cI_j\eqdef \{ i \in [n] \;:\; (i,j) \notin \cZ\}, \qquad \cJ_i\eqdef \{ j \in [d] \;:\; (i,j) \notin \cZ\}, 
\end{equation}
or, informally, $\cI_j$ denotes the set of local loss functions in which the variable $x_j$ is active, and $\cJ_i$ denotes the set of active variables in the function $f_i$. It is reasonable to assume that $1\leq |\cI_j|$ and $1\leq |\cJ_i|$ for all $j\in [d]$ and $i\in [n]$. Otherwise, if $\cI_{j'} = 0$, the specific variable $j'$ can be safely ignored in the equation~\eqref{eq:main_problem}. Similarly, if $\cI_{i'} = 0$, we can safely exclude the client $i'$ from consideration as they play no role in~\eqref{eq:main_problem}. Since the union of sets $\cI_j$ represents the set of all active variables, which we can express as $[n] \times [d] \setminus \cZ$, we write $\sum_{j=1}^d |\cI_j| = nd - |\cZ|$. Similarly, $\sum_{i=1}^n |\cJ_i| = nd - |\cZ|$. Hence, we have
\begin{equation}\label{eq:average_spartsity_sets_estimates}
	\squeeze
	\frac{1}{n}\frac{1}{d}\sum \limits_{j=1}^d |\cI_j| = \frac{1}{d} \frac{1}{n}\sum \limits_{i=1}^n | \cJ_i | = 1 - \frac{|\cZ|}{nd}.
\end{equation}
We now define two key sparsity parameters as follows:
\begin{equation}\label{eq:sigma}\xc \eqdef \max_{j\in [d]} |\cI_j|, \qquad \xr\eqdef \max_{i\in [n]} |\cJ_i|.\end{equation}
The parameter $\xc$ represents the maximum number of clients possessing any variable, while $\xr$ denotes the maximum number of variables possessed by any client. Therefore, $1\leq \xc \leq n$ and $1\leq \xr \leq d$. Since both $\xc$ and $\xr$ represent maximum values in their sets, they are both lower bounded by average values over those sets. The following inequalities result from both observations:
\[\squeeze 1 \geq \frac{\xc}{n} \geq \frac{1}{n}\frac{1}{d}\sum \limits_{j=1}^d | \cI_j | \overset{\eqref{eq:average_spartsity_sets_estimates}}{=} 1 - \frac{|\cZ|}{nd}, \qquad 1 \geq \frac{\xr}{d} \geq  \frac{1}{d} \frac{1}{n}\sum \limits_{i=1}^n | \cJ_i | \overset{\eqref{eq:average_spartsity_sets_estimates}}{=} 1 - \frac{|\cZ|}{nd}.\]

%Let $\cS_i\eqdef \{j\in [d] \;:\; (i,j) \notin \cZ\}$ be the support of function $f_i$.

An illustrative instance of loss functions having small $\xc$ and $\xr$ is represented by generalized linear models with sparse data.
\begin{example}[Linear models with sparse data]\label{ex:linear_models} Let $f_i(x) = \ell_i (a_i^\top x)$, $i=1,2,\dots,n$, where $\ell_i:\R\to \R$ are loss functions, and the vectors $a_1,\dots,a_n\in \R^d$ are sparse. Then \[\cZ = \{ (i,j) \in [n] \times [d] \;:\; \ell_i'(a_i^\top x) a_{ij} = 0, \; \forall x\in \R^d\} \supseteq  \{ (i,j) \in [n] \times [d] \;:\; a_{ij} = 0\} \eqdef  \cZ'.
\]
In this case, $\xc \leq \max_j |\{ i \;:\; (i,j) \neq \cZ'\}|  =  \max_j |\{ i \;:\; a_{ij} \neq 0\}|$.
\end{example}

We conclude this section by introducing notation that is essential for our subsequent results:\begin{equation}\label{eq:R_d_n_definiton}
\squeeze \R^d_i \eqdef \{u = (u_1,\dots,u_d)\in \R^d \;:\; u_j = 0 \text{ whenever } (i,j) \in \cZ \}.
\end{equation}
Note that for any $x \in \R^d$, the gradient $\nabla f_i(x)$ belongs to $\R^d_i$.

%%%%%%%%%%%%%%%%%%%%
%%%%%%%%%%%%%%%%%%%%
%%%%%%%%%%%%%%%%%%%%
\section{Theory} \label{sec:theory}
%%%%%%%%%%%%%%%%%%%%
%%%%%%%%%%%%%%%%%%%%
%%%%%%%%%%%%%%%%%%%%

In this section, we present fundamental insights into how the convergence of \algname{EF21} is affected by $\xc$~and~$\xr$. To accomplish this, we revisit all the crucial elements of the original analysis of \algname{EF21}~\citep{EF21} and enhance it.

\subsection{Average smoothness}

The convergence rate of~\algname{EF21} is directly affected by~\Cref{as:L_i} only when the analysis estimates the average smoothness. Specifically, the original analysis uses the inequality
$$
\squeeze \frac{1}{n} \sum \limits_{i=1}^n  \| \nabla f_i(x) - \nabla f_i(y)\|^2 \overset{\eqref{eq:L_i}}{\leq} \frac{1}{n} \sum \limits_{i=1}^n L_i^2 \| x - y\|^2 = \tilde{L}^2 \|x - y\|^2,
$$
where $\tilde{L}^2 \eqdef \avein L_i^2$~\citep{EF21}. Given that this is the only place in the analysis where the local smoothness constants $L_i$ are relevant, a more intelligent way to analyze \Cref{alg:ef21} is to replace this assumption with a less restrictive one presented as follows.

\begin{assumption}\label{as:L+}
There exists a constant $\Lpl\geq 0$ such that
\begin{equation}\label{eq:L+}
\squeeze \frac{1}{n}\sum \limits_{i=1}^n \norm{\nabla f_i(x) - \nabla f_i(y)}^2 \leq \Lplsq \norm{x - y}^2, \quad \forall x, y \in \R^d.
\end{equation}
\end{assumption}
Based on the aforementioned observation, it follows that $\Lplsq$ is bounded above by $\tilde{L}^2$. However, we can obtain a more refined upper bound on $\Lplsq$ by leveraging the concept of sparsity introduced earlier. Concretely, we present the following lemma which provides a tighter bound on $\Lpl$.

% Note that the above assumption holds if we are ready to assume that each $f_i$ is $L_i$-smooth. Since this is a standard assumption, \eqref{eq:L+} is best viewed as an inequality defining the quantity $\Lpl$ rather than as a an assumption restricting the class of functions we consider. 

\begin{lemma}\label{lem:M-bound-via-c}
If Assumption~\ref{as:L_i} holds, then Assumption~\ref{as:L+} holds with
\begin{equation}\label{eq:L_+_upper_bound}
\squeeze \Lpl \leq\sqrt{\frac{\max_j \left\{\sum_{i: (i,j)\notin \cZ} L_i^2 \right\}}{n} } \leq   \min \left\{ \sqrt{\frac{\xc  \max_i L_i^2}{n}}, \tilde{L} \right\}.
\end{equation}
\end{lemma}
Let us introduce $\Lpl(\cZ)\eqdef \sqrt{\frac{\max_j \left\{\sum_{i: (i,j)\notin \cZ} L_i^2 \right\}}{n} }$. According to~\Cref{lem:M-bound-via-c}, $\Lpl \leq \Lpl(\cZ)$. Moreover, if $\cZ' \supset \cZ$, all other factors being equal, then the resulting sparsity level is higher and thus $\Lpl(\cZ') \leq \Lpl(\cZ)$. Therefore, we can infer that a higher degree of sparsity leads to a smaller value of $\Lpl$.

\subsection{Contraction on $\R^d_i$}

%For $0\neq x\in \R^d$, let $s(x)\eqdef |\supp(x)|$ and $P(x)$ be the vector in $\R^{s(x)}$ defined as follows: $P(x)_j$ is equal to the $j$th nonzero entry of $x$. Further, let $Q_x$ be the mapping for which $Q_x(P(x))=x$.
%
%\begin{example} Let $d=5$ and $x=(7,-2,0,0,4)$. Then $\supp(x) = \{1,2,5\}$, $s(x)=3$ and $P(x) = (7,-2,4) \in \R^3$.
%\end{example}

Another important insight concerns the contraction parameter of the $\topk{K}$ compressor.
\begin{lemma}\label{lem:contraction_on_R^d_i} Consider problem~\eqref{eq:main_problem} and~\Cref{alg:ef21}. Then, for all $i\in [n]$ and $x\in \R^d_i$, we have
\begin{equation}\label{eq:08y09fdd}\squeeze \norm{\topk{K_i}(x)-x}^2 \leq \left(1-\frac{\min\{K_i,|\cJ_i|\}}{|\cJ_i|}\right)\norm{x}^2.\end{equation}
\end{lemma}
The intuition behind the lemma is that if the active dimension space $|\cJ_i|$ of the function $f_i$ is less than $d$, then the $\topk{K_i}$ compressor is more efficient. For example, if $K_i = |\cJ_i|$, then the contractive compressor returns all non-zero components, resulting in the corresponding contraction parameter $\alpha_i = 1$. The lemma implies that if $K_i \equiv K$ for all $i \in [n]$, then the worst contraction~parameter~is~$\nicefrac{K}{r}$.

\subsection{Interpolating between orthogonality and parallelism}
The sparsity parameter $\xc$ has an important role in interpolating between orthogonal and parallel vectors, as demonstrated in the following lemma.
\begin{lemma}\label{lem:c_bound} If $u_1\in \R^d_1,\dots,u_n\in \R^d_n$, then
$$ \norm{\sum \limits_{i=1}^n u_i}^2 \leq  \xc
 \sum \limits_{i=1}^n \norm{u_i}^2.$$
\end{lemma}
If $\xc=1$, then each client $i$ owns a unique set of variables, and hence the scalar products $\la u_i, u_j \ra$ are zero for $i \neq j$, since non-zero elements of one vector are multiplied by zeros of another in the product sum. Note that for orthogonal vectors, Lemma~\ref{lem:c_bound}  becomes an equality with $\xc=1$. When $\xc = n$, Lemma~\ref{lem:c_bound} reduces to Young's inequality, which states that for any $u_i \in \RR^d$, $\|\sumin u_i \|^2~\leq n\sum \|u_i\|^2$. This inequality is an equality when all vectors are parallel and of the same length.

In the context of \algname{EF21} theory, \Cref{lem:c_bound} provides another enhancement. To present it, we introduce the following quantity representing the error in our estimation of the gradients:
\begin{equation} \label{eq:grad_variation_def}
\squeeze G^t \eqdef \frac{1}{n}\sum \limits_{i=1}^n G_i^t,	\qquad G_i^t \eqdef \norm{g_i^t - \nabla f_i(x^t)}^2.
\end{equation}
The following lemma represents this enhancement.
\begin{lemma} \label{lem:98y89fhd_8fd} Assume that $g_i^0\in \R^d_i$ for all $i \in [n]$. Then, it holds for all $t\geq 0$ that
\begin{equation}\label{eq:c_approximates_orthogonality}
\squeeze \norm{g^t -\nabla f(x^t)}^2 \leq 
%\frac{\xc}{n^2} \sum_{i=1}^n \norm{g_i^t - \nabla f_i(x^t)}^2 \overset{\eqref{eq:grad_variation_def}}{=} 
\frac{\xc}{n} G^t.
\end{equation}
\end{lemma}
In contrast to the standard analysis that uses the inequaity $\|g^t - \nabla f(x^t)\|^2 \leq G^t$, \Cref{lem:98y89fhd_8fd} provides an enhanced result by incorporating the sparsity parameter $\xc$.

\subsection{Main theorem}
Drawing on the key observations made above, we can now present our main result.

\begin{theorem}\label{thm:convergence_separate}
Let Assumptions~\ref{as:smooth}, \ref{as:lower_bounded}, \ref{as:L+} hold.
Let $\alpha\eqdef\min_i \alpha_i$, where $\alpha_i \eqdef  \frac{\min\{K_i,|\cJ_i|\}}{|\cJ_i|}$, $\theta \eqdef 1 - \sqrt{1 - \alpha}$ and $\beta\eqdef \frac{1-\alpha}{1 - \sqrt{1-\alpha}}$. Choose  $\gamma \leq \frac{1}{L + \Lpl \sqrt{\frac{\beta \xc}{\theta n}}}$.
%\[\frac{1}{\gamma} - L - \frac{\gamma \beta \xc \Lplsq}{\theta n} \geq 0\]
%\[1 - \gamma L - \frac{\gamma^2 \beta \xc \Lplsq}{\theta n} \geq 0\]
%\[\gamma^2\frac{ \beta \xc \Lplsq}{\theta n} + \gamma L \leq 1 \]
%Furthermore,  $\sqrt{\frac{\beta}{\theta}} \leq \frac{2\sqrt{1-\alpha}}{\alpha} .$
Under these conditions, the iterates of Algorithm~\ref{alg:ef21} (\algname{EF21}) satisfy
\begin{equation}\squeeze
\frac{1}{T} \sum\limits_{t=0}^{T-1} \norm{\nabla f(x^t)}^2 \leq \frac{2 \left(f(x^0) - f^\ast\right)}{\gamma T} + \frac{\xc}{n} \frac{ G^0 }{\theta T}.
\end{equation}
\end{theorem}
We make the following immediate observations:
\begin{enumerate}
\item If we let $K_i=1$ for all $i\in [n]$, then $\alpha_i = \frac{1}{|\cJ_i|}$ and $\alpha = \min_i \alpha_i =  \frac{1}{\max_i |\cJ_i|} \overset{\eqref{eq:sigma}}{=} \frac{1}{\xr}$. Thus, the term appearing in the stepsize can be bounded as follows:
$\sqrt{\nicefrac{\beta}{\theta}} = \nicefrac{(\sqrt{1-\alpha} + 1-\alpha)}{\alpha}  \leq \nicefrac{2\sqrt{1-\alpha}}{\alpha}  = 2 \xr \sqrt{1-\nicefrac{1}{\xr}} = 2 \sqrt{\xr(\xr-1)} \leq 2\xr$, where the proof of the first equality is deferred to the appendix. By using~\Cref{thm:convergence_separate}, we find that~\algname{EF21} with $\topk{1}$ compressor requires $\cO((L + \xr \Lpl \sqrt{\frac{c}{n}})\frac{1}{\varepsilon})$ bits to converge to the $\varepsilon$-stationary point. Comparing this with the standard communication complexity of~\algname{DGD} $\cO(\xr L \frac{1}{\varepsilon})$, we can see that~\algname{EF21} gets asymptotically faster if $\Lpl \sqrt{\frac{\xc}{n}} \leq L$.

\item By using the largest stepsize allowed by the theory, we get the bound
\iffalse
\begin{equation}
\label{eq:conv_bound}
\begin{aligned}\squeeze \frac{1}{T} \sum\limits_{t=0}^{T-1} \norm{\nabla f(x^t)}^2  &\leq \squeeze  \frac{2\left(L +  \sqrt{\frac{\max_j \left\{\sum_{i: (i,j)\notin \cZ} L_i^2 \right\}}{n}}\sqrt{\frac{ \beta \xc }{\theta n}}   \right) \Psi^0}{T} \\
&\leq  \squeeze \frac{2\left(L +   \min \left\{\sqrt{\frac{\xc  \max_i L_i^2}{n}}, \sqrt{\frac{\sum_{i=1}^n L_i^2}{n} }\right\} \sqrt{\frac{ \beta \xc }{\theta n}}   \right) \Psi^0}{T}.
\end{aligned}
\end{equation}
\fi
\begin{equation}
	\label{eq:conv_bound}
	\begin{aligned}\squeeze \frac{1}{T} \sum\limits_{t=0}^{T-1} \norm{\nabla f(x^t)}^2 \leq  \squeeze 2\left(L +   \min \left\{\sqrt{\frac{\xc  \max_i L_i^2}{n}}, \sqrt{\frac{\sum_{i=1}^n L_i^2}{n} }\right\} \sqrt{\frac{ \beta \xc }{\theta n}}   \right)  \frac{\Psi^0}{T}.
	\end{aligned}
\end{equation}
\item In a pessimistic scenario where $\xc=n$, the bound~\eqref{eq:conv_bound} recovers the standard \algname{EF21} rate: %$\squeeze \frac{1}{T} \sum\limits_{t=0}^{T-1} \norm{\nabla f(x^t)}^2  \leq  \squeeze   \frac{2\left(L +   \sqrt{\frac{\sum_{i=1}^n L_i^2}{n}}  \sqrt{\frac{ \beta }{\theta }}   \right) \Psi^0}{T}.$
\begin{eqnarray*}\squeeze \frac{1}{T} \sum\limits_{t=0}^{T-1} \norm{\nabla f(x^t)}^2  &\leq & \squeeze   2\left(L +   \sqrt{\frac{\sum_{i=1}^n L_i^2}{n}}  \sqrt{\frac{ \beta }{\theta }}   \right) \frac{\Psi^0}{T}.\end{eqnarray*}
\item In an optimistic scenario where $\xc=1$, the bound~\eqref{eq:conv_bound} recovers the \algname{EF21} rate in the separable regime, discussed in the appendix:
\begin{eqnarray*}\squeeze \frac{1}{T} \sum\limits_{t=0}^{T-1} \norm{\nabla f(x^t)}^2  &\leq & \squeeze  2\left(L +   \sqrt{\frac{\max_i L_i^2}{n} } \sqrt{\frac{ \beta }{\theta n}}   \right) \frac{\Psi^0}{T}.\end{eqnarray*}
\end{enumerate}
In the next section, we perform computational experiments to validate our theoretical results.

\section{Experiments}

The practical superiority of \algname{EF21} has been demonstrated in various previous publications \citep{szlendak2022permutation, EF21}. These publications have explicitly indicated that the stepsize of \algname{EF21} utilized in the experiments is fine-tuned, and is typically substantially larger than the value suggested by theoretical analysis. Consequently, the objective of our experimental section is not to reassert the dominance of the Error Feedback algorithm, but rather to substantiate our theoretical contentions. In order to do so, it suffices to perform small scale but carefully executed experiments.

\subsection{Linear regression on sparse data}
\label{sec:exp_lin_reg_with sparse_data}
%\elnur{The comparison with standard \algname{GD} is missing.}

\begin{figure*}[t]
	\centering
	\captionsetup[sub]{font=scriptsize,labelfont={}}	
	\captionsetup[subfigure]{labelformat=empty}
	%\captionsetup{position=top}
	
	\begin{subfigure}[ht]{0.32\textwidth}
		\includegraphics[width=\textwidth]{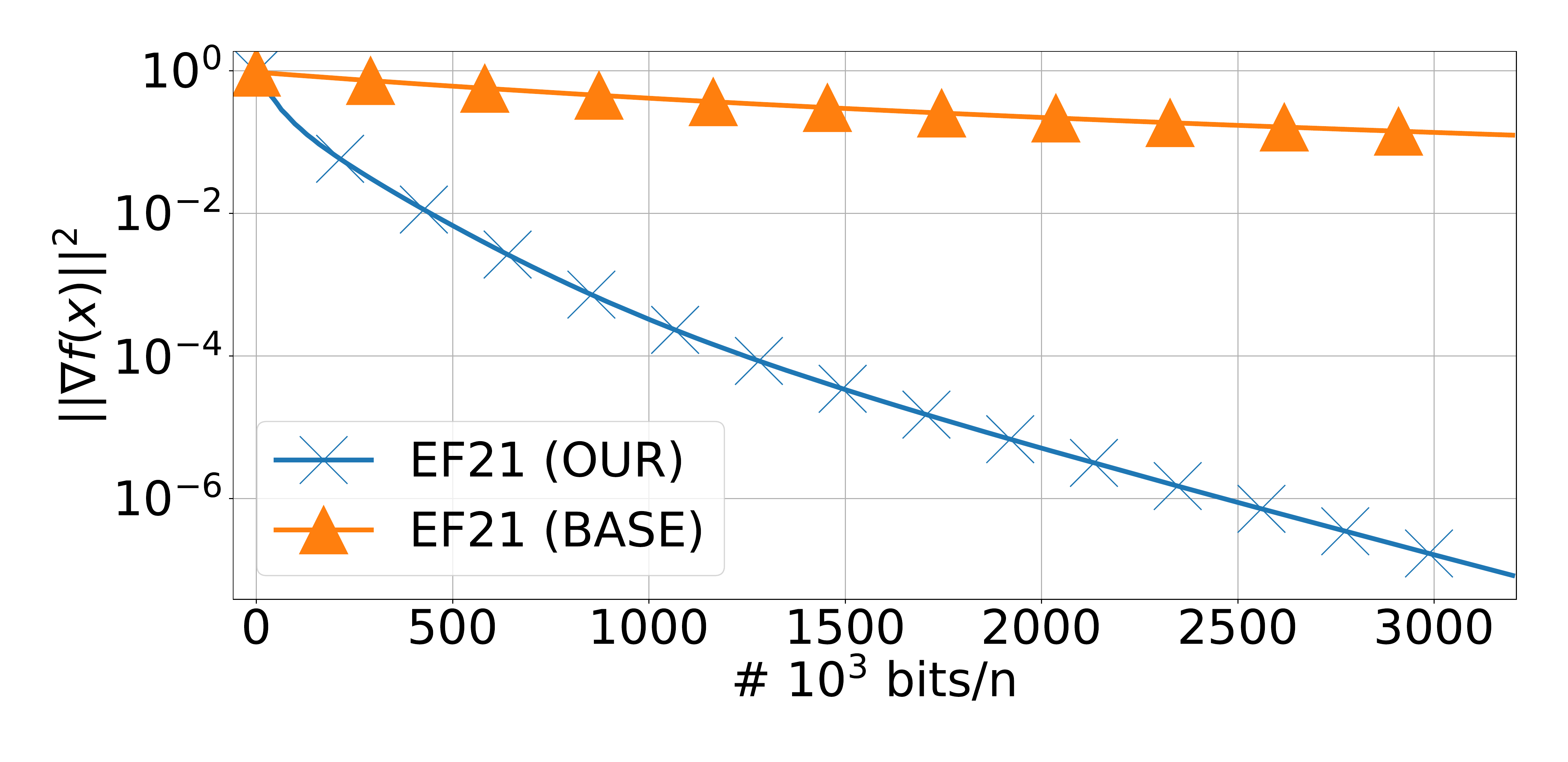} \caption{ $\xc/n=0.05$}
	\end{subfigure}
	\begin{subfigure}[ht]{0.32\textwidth}
		\includegraphics[width=\textwidth]{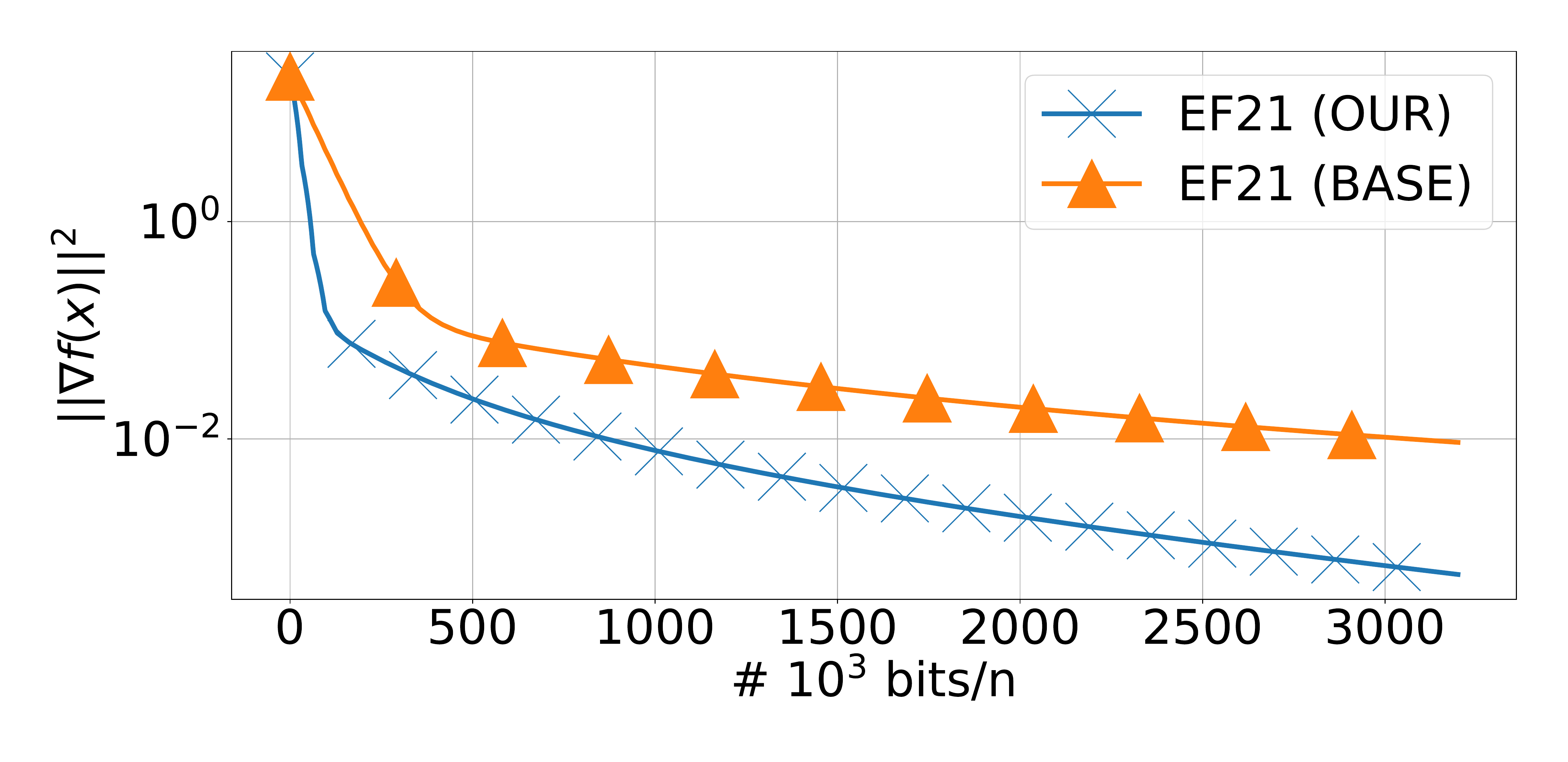} \caption{ $\xc/n=0.5$}
	\end{subfigure}
	\begin{subfigure}[ht]{0.32\textwidth}
		\includegraphics[width=\textwidth]{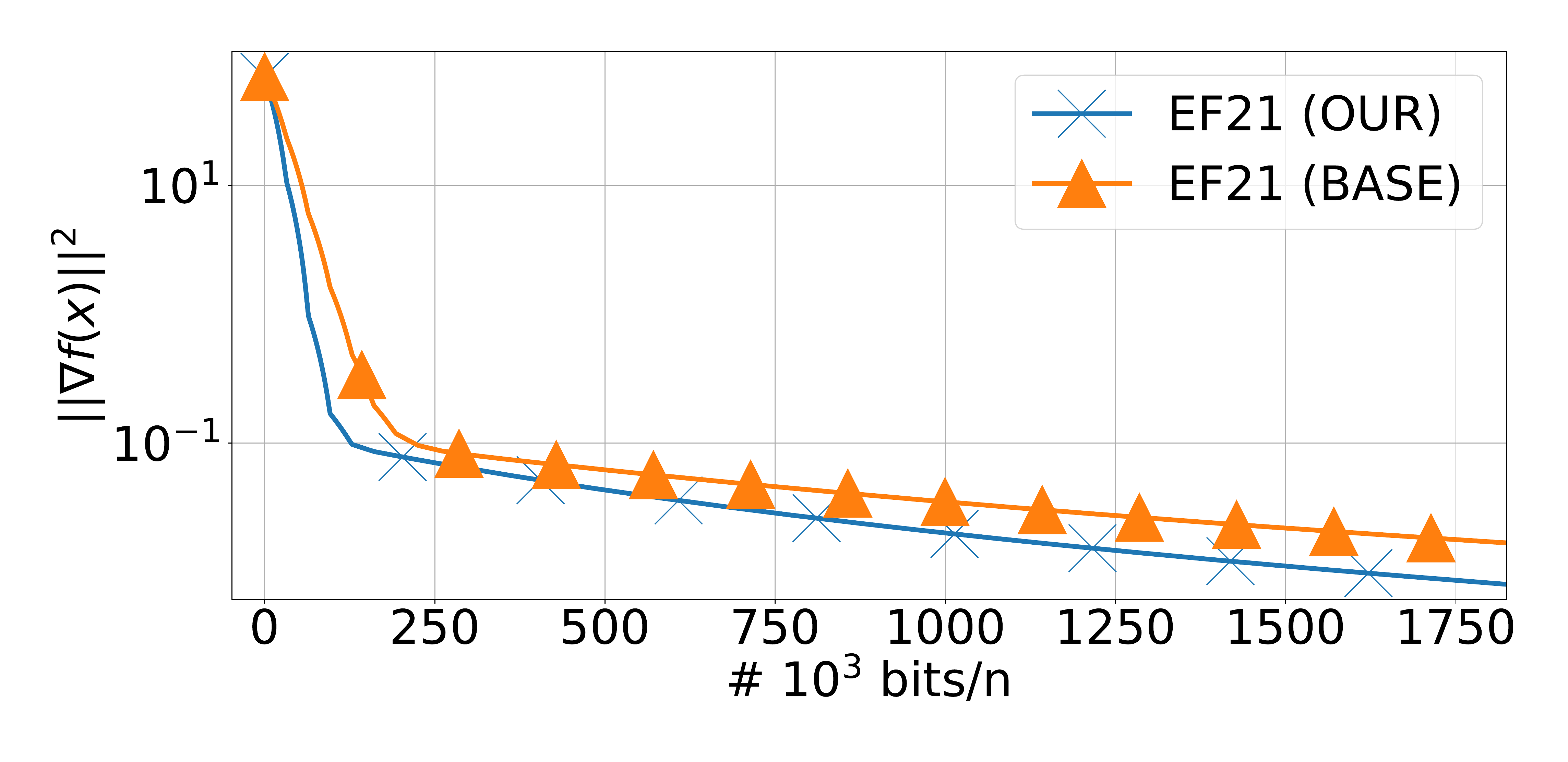} \caption{$\xc/n=0.9$}
	\end{subfigure}
	
	\iffalse	
	\begin{subfigure}[ht]{0.32\textwidth}
		\includegraphics[width=\textwidth]{./figs/expsyn/v-0-99/fig1-c-n-005.pdf} \caption{{ ($a_2$) $c/n=0.05, v=0.99$}}
	\end{subfigure}
	\begin{subfigure}[ht]{0.32\textwidth}
		\includegraphics[width=\textwidth]{./figs/expsyn/v-0-99/fig3-c-n-05.pdf} \caption{{ ($c_2$) $c/n=0.5, v=0.99$}}
	\end{subfigure}
	\begin{subfigure}[ht]{0.32\textwidth}
		\includegraphics[width=\textwidth]{./figs/expsyn/v-0-99/fig4-c-n-09.pdf} \caption{{($d_2$) $c/n=0.9, v=0.99$}}
	\end{subfigure}
	\fi
	
	%\vspave{-1pt}
	\caption{Comparison of the performance of \algname{EF21}+\;\topk{1} with the standard approach proposed by \citet{EF21} and a newly proposed stepsize (see~\Cref{thm:convergence_separate}) on the linear regression problem.  The sparsity pattern $\xc$ is controlled by manipulating the sparsity of the data matrix. }
	\label{fig:linear_regression_synthetic}
\end{figure*}

Consider the following optimization problem:

$$
 \squeeze \min\limits_{x \in \RR^d} \left\{f(x) = \frac{1}{n}\sum \limits_{i=1}^n f_i(x) =  \frac{1}{n}\sum \limits_{i=1}^n \frac{1}{m} \left\|\bA_i x - b_i\right\|^2\right\},
$$

where $\bA_i \in \RR^{m \times d}, b_i \in \RR^m$ are the sparse training data and labels. Each $f_i$ is a $L_i$-smooth function, where $L_i~=~ \sup_{x\in \RR^d} \|\nabla^2 f_i(x)\| ~=~ \left\|\frac{2}{m} \bA_i^\top \bA_i\right\|$~\citep{nesterov2018lectures}. We show in the appendix that $f$ satisfies~\Cref{as:L+} with $\Lplsq~=~\frac{4}{m^2 n} \lambda_{\max} \left(\sumin (\bA_i^\top \bA_i)^2 \right)$, where $\lambda_{\max}(\cdot)$ refers to the largest eigenvalue. In all experiments in this section, we fix $n$, $d$ and $m$ to 500, 100 and 12, respectively.

As the function $f_i$ is a generalized linear function, we can manipulate the sparsity of the data matrix $\bA_i$ to adjust the sparsity parameter $\xc$ (see~\Cref{ex:linear_models}). In our experiments, we varied the ratio $\xc/n$, a critical factor in the theoretical analysis, over the values in the list $[0.05, 0.5, 0.9]$. To strengthen our findings, we controlled individual smoothness constants $L_i$ to ensure that the constant $\Lplsq$ from~\Cref{as:L+}, which is presented in our theoretical results, was much smaller than $\tilde{L}^2 = \avein L_i^2$ used by~\citet{EF21}.  Further details can be found in the appendix.

 All experiments were implemented using {\tt FL\_PyTorch} \citep{burlachenko2021fl_pytorch} and were conducted on two Linux workstations with {\tt x86\_64} architecture and 48 CPUs each.

The results presented in Figure~\ref{fig:linear_regression_synthetic} demonstrate that the performance gap between \algname{EF21} with the standard and new stepsizes, as proposed in~\Cref{thm:convergence_separate}, is significant when the parameter $\xc$ is much smaller than its maximum value $n$. Conversely, as $\xc$ approaches $n$, the difference in performance becomes negligible. This is expected because firstly, the new stepsize scales directly with the ratio $\xc/n$, and secondly, the constant $\Lplsq$, as can be seen from \eqref{eq:L_+_upper_bound}, approaches $\tilde{L}^2 = \avein L_i^2$ which is utilized in the standard theory. The findings align with the paper's main claim that \algname{EF21} achieves faster convergence with smaller sparsity parameter $\xc$.

%From Fig. \ref{fig:linear_regression_synthetic} we can see as $c/n$ increases and come close to $1$ \algname{EF21} with our analysis is reducing to standard \algname{EF21}. From Fig. \ref{fig:exp_syn_v_0_0001} ($a_1$) we can observe the most tremendous improvement due to decreasing $\Lpl$ by factor $\sqrt{n}=10$ from the construction of $L_i$ distribution. The second contribution is due to that $1/(c/n)=1/0.05=20$. These two causes jointly lead to practical improvement in the convergence speed of $\|\nabla f(x^t)\|^2$ by a factor $200$ in terms of communication rounds. In another hand, we can generate situation Fig. \ref{fig:exp_syn_v_0_0001} ($d_3$) when both causes do not give sufficient convergence improvements.

%\kostya{TODO: describe the generation of dataset}
%\kostya{But in all our synthetic settings each xi is shared exactly across "c" clients.}

\subsection{Logistic regression with adaptive stepsize}
%\elnur{Kostya, please remove '$d = x, n = y$' on the plots, additional text is redundant. Also remove 'TopK[K=1]' -- we indicate this in the main body of the text.}
The inequalities in~\eqref{eq:L_+_upper_bound} and~\eqref{eq:c_approximates_orthogonality} rely heavily on the parameter $\xc$. To explore the possibility of eliminating the sparsity condition altogether, we consider replacing the initial definition of $\xc$ as a sparsity pattern with~\eqref{eq:c_approximates_orthogonality}. In this experimental section, we investigate this question.

We consider the following optimization problem:

$$
\squeeze \min\limits_{x \in \RR^d} \left\{f(x) = \frac{1}{N} \sum\limits_{i=1}^N \log(1 + e^{-y_i a_i^\top x})\right\},
$$
where $a_i \in \RR^d, y \in \{-1, 1\}$ represent the training data and labels, respectively. We utilize three LIBSVM~\citep{chang2011libsvm} datasets, namely \textit{phishing, mushrooms, w5a}, as the training data, dividing the data evenly between $n=300$ clients.

At each iteration of~\algname{EF21}, we compute the parameter $\xct$ adaptively as the smallest value satisfying Inequality \eqref{eq:c_approximates_orthogonality}, i.e., we set $$\xct \eqdef \frac{n \|g^t-\nabla f(x^t)\|^2}{G^t}.$$ We use this quantity to define $$L_{+}^t \eqdef \min \left\{ \sqrt{\frac{\xct  \max_i L_i^2}{n}} , \sqrt{\frac{\sum_{i=1}^n L_i^2} {n}} \right\},$$ as an estimation of $L_+$. Finally, we compute the stepsize using the formula provided in~\Cref{thm:convergence_separate},
$$
\squeeze \gamma^t \eqdef \left(L + \Lplt \sqrt{\frac{ \xct }{n}} \frac{\sqrt{1-\alpha} + 1-\alpha}{\alpha} \right)^{-1},
$$
but with the adaptive quantities $L_+^t$ and $\xct$ replacing the quantities $L_+$ and $\xc$ they estimate. Note  that evaluating $\xct$ is time consuming as it requires the computation of $\nabla f(x^t)$, which is usually unavailable at the master.

\begin{figure*}[t]
	\centering
	\captionsetup[sub]{font=scriptsize,labelfont={}}	
	\captionsetup[subfigure]{labelformat=empty}
	%\captionsetup{position=top}
	
	\begin{subfigure}[ht]{0.32\textwidth}
		\includegraphics[width=\textwidth]{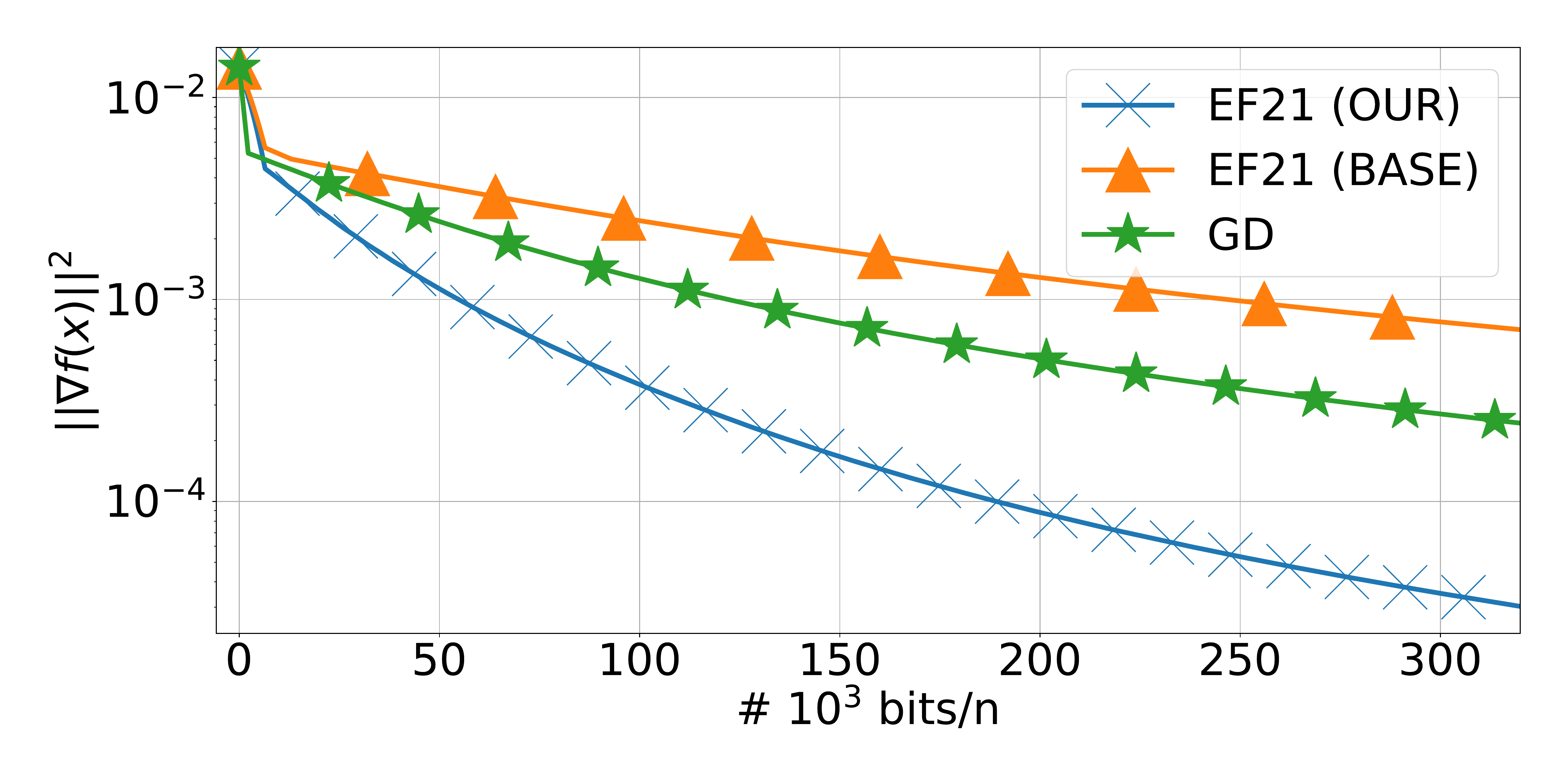}
		\caption{(a) phishing}
	\end{subfigure}		
	\begin{subfigure}[ht]{0.32\textwidth}
		\includegraphics[width=\textwidth]{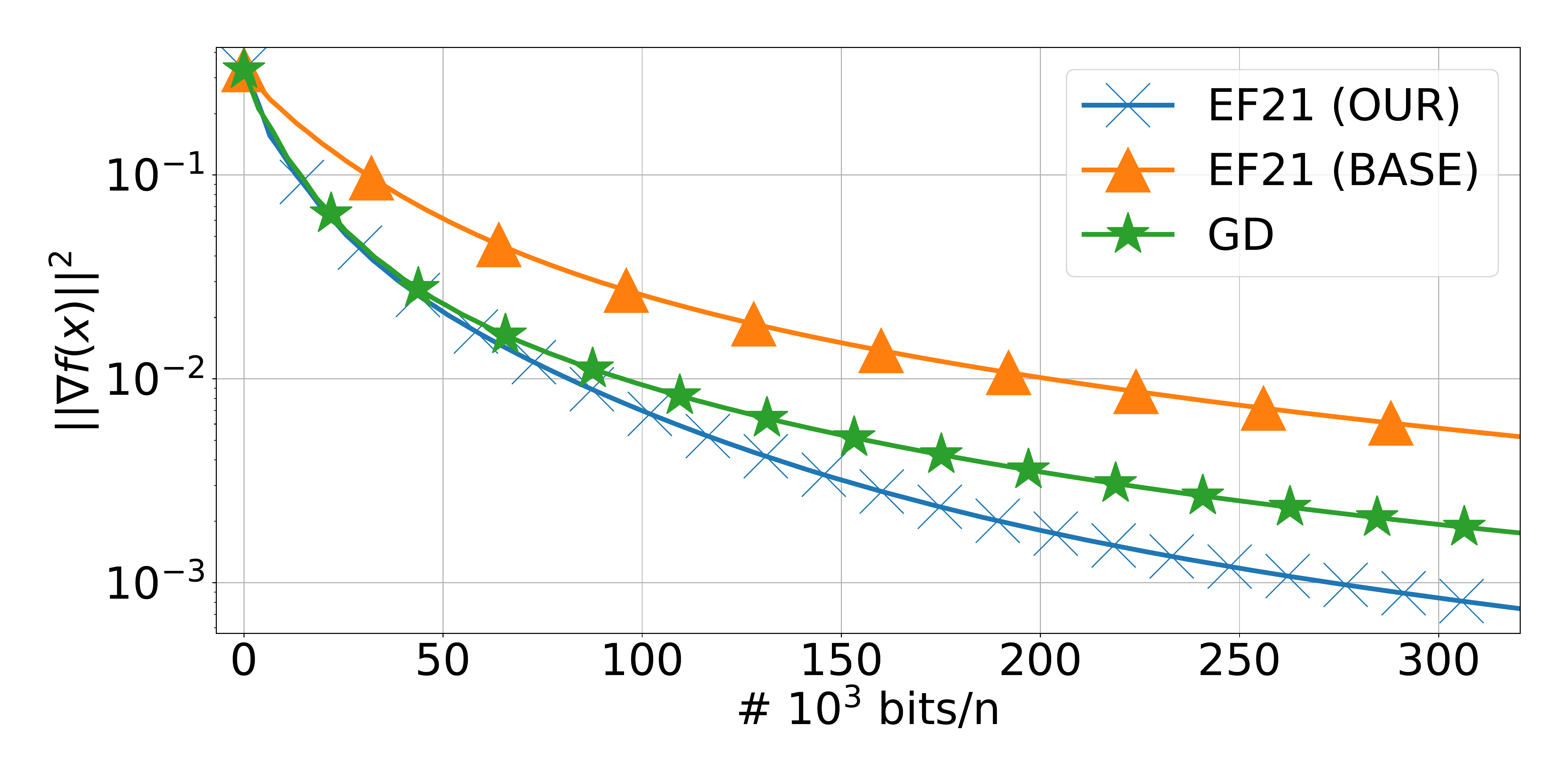} 
	\caption{(b) mushrooms}		
	\end{subfigure}	
	\begin{subfigure}[ht]{0.32\textwidth}
		\includegraphics[width=\textwidth]{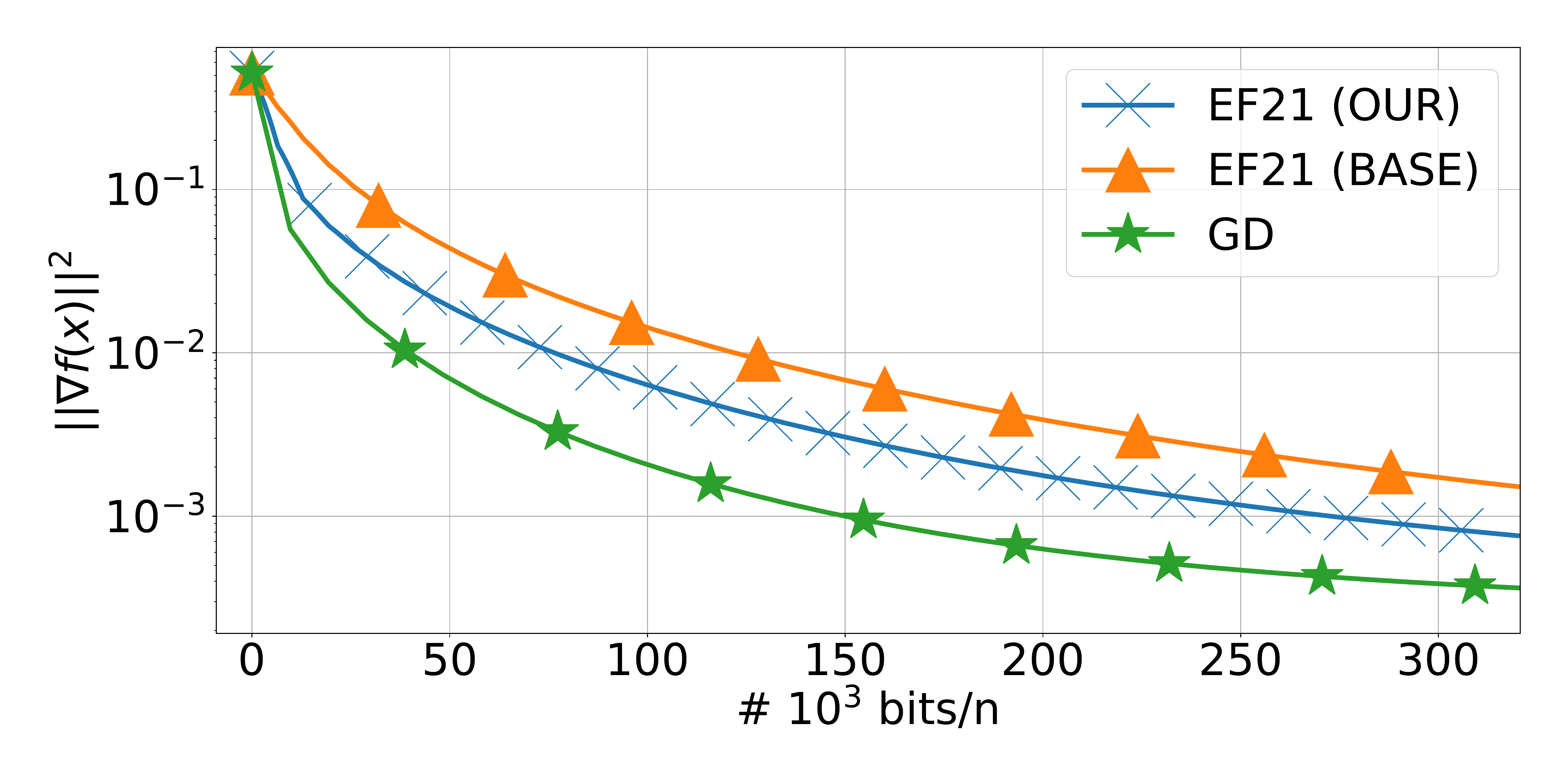} 
		\caption{(c) w5a}
	\end{subfigure}	
	%\vspave{-1pt}
	\caption{Comparison of the performance of the \algname{EF21}+\;\topk{1} algorithm, with both standard and adaptive stepsize, and the \algname{GD} method, on a logistic regression problem on the LIBSVM datasets. All stepsizes are chosen based on the theory.}
	\label{fig:log_reg_main}
\end{figure*}

As can be seen from \Cref{fig:log_reg_main}, our results demonstrate that adaptive computation of $\xct$ can be a promising direction for further investigation.

\section*{Acknowledgements}

Peter Richt\'{a}rik and Konstantin Burlachenko acknowledge support from the SDAIA-KAUST Center of Excellence in Data Science and Artificial Intelligence, and all authors acknowledge support from the KAUST Baseline Research Fund.

\newpage

\bibliographystyle{plainnat}
\bibliography{neurips_2023.bib}

\clearpage

\appendix
\part*{Appendix}

\tableofcontents

\newpage

\section{Auxiliary Results and Missing Proofs}

\subsection{Proof of Lemma~\ref{lm:alpha_1_the_best}}

\begin{proof}
	Let us now recall that \algname{EF21} needs at least\footnote{We suppose the initial gradient estimate $g^i_0$ equals $\nabla f_i(x^0)$ as a part of preprocessing step. This zeroes out the second term of Theorem 2 in~\citep{EF21}.}
	$
	T = \frac{2\Delta_f^0 \left(L + \tilde{L} \left(\frac{1 + \sqrt{1-\alpha}}{\alpha} - 1\right) \right)}{\varepsilon}
	$
	iterations to ensure 
	$
	\ExpBr{\norm{\nabla f(\hat{x}^T)}^2}\leq \varepsilon,
	$
	where $\Delta_f^0 = f(x^0) - f^*$, $\tilde{L} = \sqrt{\avein L_i^2}$, and $\alpha$ is the contraction parameter of the compressor. In the case of $\topk{K}$, as was noted earlier, $\alpha = \frac{k}{d}$. To find the optimal $\alpha$, which requires the minimum number of communication, we need to minimize $T \cdot k \cdot n$, since at each iteration of the algorithm each of $n$ clients sends $k$ float numbers. Thus,
	\begin{align*}
		&\min\limits_{k} \frac{2\Delta_f^0 \left(L + \tilde{L} \left(\frac{1 + \sqrt{1-\alpha}}{\alpha} - 1\right) \right)}{\varepsilon} \cdot k \cdot n \\
		&  \overset{\alpha = \frac{k}{d}}{\Longleftrightarrow} \ \min\limits_{\alpha} \left\{\xi(\alpha) \eqdef \left(L + \tilde{L}  \left(\frac{1 + \sqrt{1-\alpha}}{\alpha} - 1\right)\right) \cdot \alpha\right\}.
	\end{align*}
	
	Taking the derivative of the object over $\alpha$, we get
	\begin{eqnarray*}
		\xi'(\alpha) &= &  L + \tilde{L}  \left(\frac{1 + \sqrt{1-\alpha}}{\alpha} - 1\right) + \tilde{L}\alpha \left(-\frac{1}{(1 - \sqrt{1-\alpha})^2} \frac{1}{2\sqrt{1-\alpha}} \right)\\
		& =  &  L + \tilde{L}  \left(\frac{1 + \sqrt{1-\alpha}}{\alpha} - 1\right) - \tilde{L}\frac{1+\sqrt{1+\alpha}}{1 - \sqrt{1-\alpha}} \frac{1}{2\sqrt{1-\alpha}} \\
		& =  & L - \tilde{L} + \frac{\tilde{L}}{1 - \sqrt{1 - \alpha}} \left(1 - \frac{1 + \sqrt{1 + \alpha}}{2\sqrt{1 - \alpha}} \right) \\
		& = &  L - \tilde{L} + \frac{\tilde{L}}{(1 - \sqrt{1-\alpha}) 2 \sqrt{1 - \alpha}} \underbrace{(\sqrt{1 - \alpha} - 1)}_{\leq 0} \\
		& \leq & L - \tilde{L} \\
		& \leq & 0,
	\end{eqnarray*}
	where the last inequality holds since $L \leq \tilde{L}$. Indeed, $L \leq \avein L_i \leq \sqrt{\avein L_i^2} = \tilde{L}$. As a result, it can be concluded that the minimum of the objective function is achieved when the value of $\alpha$ is the largest, i.e., $\alpha = 1$). For the specific case of the $\topk{K}$ compressor, choosing $\alpha=1$ corresponds to selecting $k=d$ in the compressor definition.
\end{proof}

\subsection{Refining smoothness when sparsity is present}

\begin{lemma}\label{lem:L_i-implication}
	If Assumption~\ref{as:L_i} holds, then
	for every $i \in [n]$, we have  
	\begin{equation}\label{eq:b8d9yhd_098ufd} \sum \limits_{j: (i,j)\notin \cZ} ( (\nabla f_i(x))_j - (\nabla f_i(y))_j )^2 \leq L_i^2 \sum \limits_{j: (i,j)\notin \cZ} (x_j - y_j)^2, \qquad \forall x,y\in\R^d.\end{equation}
\end{lemma}

\begin{proof}
Given any $i\in [n]$ and any $x,y\in \R^d$, we know that
\begin{eqnarray}\sum_{j: (i,j)\notin \cZ} \left( (\nabla f_i(x) )_j - (\nabla f_i(y))_j \right)^2 &=&   \sum_{j=1}^d \left( (\nabla f_i(x) )_j - (\nabla f_i(y))_j \right)^2 \notag \\
&=& \norm{\nabla f_i(x) - \nabla f_i(y)}^2 \notag\\
&\leq & L_i^2 \norm{x-y}^2,  \label{eq:-080897df}
\end{eqnarray}
where the last inequality follows from Assumption~\ref{as:L_i}. Let $x'$ and $y'$ be formed from $x$ and $y$ by replacing the coordinates $j$ for which $(i,j)\in \cZ$ by zeros. That is, 
\begin{equation} \label{eq:h787td98hfd}x'_j = \begin{cases} x_j & (i,j) \notin \cZ \\
0 & (i,j) \in \cZ
\end{cases}, \qquad \qquad y'_j = \begin{cases} y_j & (i,j) \notin \cZ \\
0 & (i,j) \in \cZ
\end{cases}.\end{equation}

Applying inequality \eqref{eq:-080897df} with $x \leftarrow x'$ and $y \leftarrow y'$, we obtain
\begin{eqnarray}\sum_{j: (i,j)\notin \cZ} \left( (\nabla f_i(x') )_j - (\nabla f_i(y'))_j \right)^2 
&\leq & L_i^2 \norm{x'-y'}^2\notag\\
&=& L_i^2 \left(\sum_{j: (i,j)\notin \cZ} (x'_j - y'_j)^2 +  \sum_{j: (i,j)\in \cZ} (x'_j - y'_j)^2\right)\notag\\
&\overset{\eqref{eq:h787td98hfd}}{=}& L_i^2 \left(\sum_{j: (i,j)\notin \cZ} (x_j - y_j)^2 +  \sum_{j: (i,j)\in \cZ} (0 - 0)^2\right) \notag\\
&=& L_i^2 \sum_{j: (i,j)\notin \cZ} (x_j - y_j)^2. \label{eq:080fd--0))*f}
\end{eqnarray}

%\begin{eqnarray}\sum_{j: (i,j)\notin \cZ} \left( (\nabla f_i(x) )_j - (\nabla f_i(y))_j \right)^2 &=&   \sum_{j=1}^d \left( (\nabla f_i(x) )_j - (\nabla f_i(y))_j \right)^2 \notag \\
%&=& \norm{\nabla f_i(x) - \nabla f_i(y)}^2 \notag\\
%&\leq & L_i^2 \norm{x-y}^2,  \label{eq:-080897df}
%\end{eqnarray}

The result now follows by comparing the left-hand side of \eqref{eq:-080897df} and the right-hand side of \eqref{eq:080fd--0))*f} in view of the observation that $\nabla f_i(x) = \nabla f_i(x')$ and $\nabla f_i(y) = \nabla f_i(y')$.
\end{proof}

\subsection{Proof of Lemma~\ref{lem:M-bound-via-c}}

\begin{proof}

Using Lemma~\ref{lem:L_i-implication}, we can now write
\begin{equation}\label{eq:aux_smoothness1}
\begin{aligned}
 \frac{1}{n} \sum_{i=1}^n  \norm{\nabla f_i(x) - \nabla f_i(y)}^2 &=&  \frac{1}{n} \sum_{i=1}^n  \sum_{j=1}^d \left( (\nabla f_i(x) )_j - (\nabla f_i(y))_j \right)^2 \\
&=&  \frac{1}{n} \sum_{i=1}^n  \sum_{j: (i,j)\notin \cZ} \left( (\nabla f_i(x) )_j - (\nabla f_i(y))_j \right)^2 \\
&\overset{\eqref{eq:b8d9yhd_098ufd}}{\leq}&  \frac{1}{n} \sum_{i=1}^n  L_i^2 \sum_{j: (i,j)\notin \cZ} (x_j - y_j)^2 \\
&=&   \frac{1}{n} \sum_{i=1}^n  \sum_{j: (i,j)\notin \cZ} L_i^2 (x_j - y_j)^2  \\
&=&   \frac{1}{n} \sum_{j=1}^d \sum_{i : (i,j)\notin \cZ}  L_i^2 (x_j - y_j)^2  \\
&=& \frac{1}{n}\sum_{j=1}^d  \left[(x_j - y_j)^2 \sum_{i: (i,j)\notin \cZ} L_i^2\right] .
\end{aligned}
\end{equation}

To further advance our analysis, we proceed by determining the maximum value in each individual sum term.
\begin{eqnarray*}
 \frac{1}{n} \sum_{i=1}^n  \norm{\nabla f_i(x) - \nabla f_i(y)}^2 &\overset{\eqref{eq:aux_smoothness1}}{\leq}& \frac{1}{n}\sum_{j=1}^d  \left[(x_j - y_j)^2 \sum_{i: (i,j)\notin \cZ} L_i^2\right] \\
&\leq &\frac{1}{n}\sum_{j=1}^d  \left[(x_j - y_j)^2 \max_j \left\{ \sum_{i: (i,j)\notin \cZ} L_i^2 \right\}\right]  \\
&=& \frac{\max_j \left\{\sum_{i: (i,j)\notin \cZ} L_i^2\right\}}{n}\sum_{j=1}^d  (x_j - y_j)^2   \\
&=&\frac{\max_j \left\{\sum_{i: (i,j)\notin \cZ} L_i^2 \right\}}{n} \norm{x - y}^2.
%&\leq & \frac{\max_i L_i^2}{n} \sum_{i=1}^n \sum_{j: (i,j)\notin \cZ} (x_j - y_j)^2\\
%&=&\frac{\max_i L_i^2}{n} \sum_{j=1}^d \sum_{i: (i,j)\notin \cZ} (x_j - y_j)^2\\
%&=&\frac{\max_i L_i^2}{n} \sum_{j=1}^d \card\left\{i: (i,j)\notin \cZ\right\} (x_j - y_j)^2\\
%&\overset{\eqref{eq:sigma}}{\leq} & \frac{\xc  \max_i L_i^2}{n} \sum_{j=1}^d (x_j - y_j)^2\\
%&=&\frac{\xc  \max_i L_i^2}{n} \norm{x - y}^2.
\end{eqnarray*}
Based on the preceding inequality, it can be deduced that $\Lplsq \leq \frac{\max_j \left\{\sum_{i: (i,j)\notin \cZ} L_i^2 \right\}}{n} $. Furthermore,
\begin{eqnarray*}\frac{\max_j \left\{\sum_{i: (i,j)\notin \cZ} L_i^2 \right\}}{n} & \leq &\frac{\max_j \left\{\sum_{i: (i,j)\notin \cZ} \max_i L_i^2 \right\}}{n} \\
&= & \frac{\left( \max_i L_i^2 \right) \left\{\max_j \sum_{i: (i,j)\notin \cZ} 1 \right\}}{n}  \quad  \overset{\eqref{eq:sigma}}{=} \quad \frac{\left( \max_i L_i^2 \right) \xc}{n},
\end{eqnarray*}
and
\begin{eqnarray*}\frac{\max_j \left\{\sum_{i: (i,j)\notin \cZ} L_i^2 \right\}}{n} & \leq & \frac{\max_j \left\{\sum_{i=1}^n L_i^2 \right\}}{n} 
\quad = \quad \frac{\sum_{i=1}^n L_i^2 }{n}.
\end{eqnarray*}

This concludes the proof of the lemma.
\end{proof}

\subsection{Proof of Lemma~\ref{lem:contraction_on_R^d_i}}
 
\begin{proof}
Choose $i\in [n]$ and $x\in \R^d_i$. If $|\cJ_i|=d$, the statement turns into the standard contraction property of $\topk{K_i}$ on $\R^d$, and hence it holds\footnote{The standard contraction property says that $\norm{\topk{K_i}(x)-x}^2 \leq \left(1-\frac{K_i}{d}\right)\norm{x}^2$, for all $ x\in \R^d$.}. Assume therefore that $|\cJ_i|<d$. If $x=0$, inequality \eqref{eq:08y09fdd} clearly holds. Therefore, let us assume that $x\neq 0$. Notice that $\supp(x)\eqdef \{j \in [d] \;:\; x_j \neq 0\} \subseteq \cJ_i$. Hence, $$s\eqdef |\supp(x)| \leq |\cJ_i| <d.$$ 

Since neither the left nor the right hand side of \eqref{eq:08y09fdd} changes if we permute the coordinates of $x$, we can w.l.o.g.\ assume that $|x_1|\geq |x_2| \geq \cdots \geq |x_d|$. Notice that $|x_s|>0$ and that $|x_{s+1}|=\cdots=|x_{d}|=0$. Let $y=\topk{K_i}(x)$, and notice that $y_j=x_j$ for $1 \leq j\leq K_i$, $y_i=0$ for $j>K_i$, and $y_j=x_j=0$ for $ s+1 \leq j\leq d$. If $K_i\geq s$, then $y_j=x_j$ for all $j$, which means that the left-hand side in \eqref{eq:08y09fdd} is equal to zero. Therefore,  inequality \eqref{eq:08y09fdd} holds. If $K_i < s$, then 
\begin{eqnarray}\norm{y-x}^2 &=&  \sum_{j=1}^{K_i} \left( y_j-x_j\right)^2 +  \sum_{j=K_i+1}^{s} \left(y_j-x_j\right)^2 + \sum_{j=s+1}^{d} \left(y_j-x_j\right)^2 \notag \\ 
&=& \sum_{j=1}^{K_i} \left( x_j-x_j\right)^2 +  \sum_{j=K_i+1}^{s} \left(0-x_j\right)^2 + \sum_{j=s+1}^{d} \left(0-0\right)^2 \notag \\ 
&=& \sum_{j=K_i+1}^{s} x_j^2  .\label{eq:h8y0fd98y)_98yfd}
\end{eqnarray}
Because $x_1^2, x_2^2,\dots,x_d^2$ is non-increasing, we have
\[\frac{1}{s}\sum_{j=1}^{s} x_j^2 \geq \frac{1}{s-K_i}\sum_{j=K_i+1}^{s} x_j^2.\]
Plugging this estimate into \eqref{eq:h8y0fd98y)_98yfd}, we get
\[\norm{y-x}^2 \leq \frac{s-K_i}{s}\sum_{j=1}^{s} x_j^2  = \left(1-\frac{K_i}{s} \right)\norm{x}^2. \]
It remains to apply the bound $s \leq |\cJ_i|$ and use the identity $K_i = \min \{K_i,s\} =\min \{ K_i, |\cJ_i|\}$.
\end{proof}

\subsection{Helping lemma on orthogonality}

\begin{lemma}\label{lem:87t87fdu8df} Let $u_1,\dots,u_n \in \R^d$. Let us write $u_i = (u_{i1},\dots,u_{id}) \in \R^d$,  and define  sets $
	S_j \eqdef \{i \;:\; u_{ij} \neq 0\}$  for $j=1,2,\dots,d$.  
	Then
	\begin{equation}\label{eq:87t87fdu8df} \norm{\sum \limits_{i=1}^n u_i}^2 \leq  \left(\max_j |S_j| \right) \times
		\sum_{i=1}^n \norm{u_i}^2.\end{equation}
\end{lemma}

\iffalse
\peter{An idea for an alternative statement:
\begin{eqnarray*}\norm{\sum_{i=1}^n u_i}^2 = \sum_{j=1}^d \left(\sum_{i=1}^n u_{ij} \right)^2= \sum_{j=1}^d \left(\sum_{i\in S_j} u_{ij} \right)^2 \leq \sum_{j=1}^d  |S_j | \sum_{i\in S_j} u_{ij}^2 &=& \sum_{j=1}^d  |S_j | \sum_{i=1}^n u_{ij}^2 \\
&=& \sum_{j=1}^d \sum_{i=1}^n  |S_j |  u_{ij}^2 \\
&=&\sum_{i=1}^n \sum_{j=1}^d   |S_j |  u_{ij}^2 \\
&=& \sum_{i=1}^n \norm{u_i}_w^2,
\end{eqnarray*}
where $w=(|S_1 |,\dots,|S_d |)$  and $\norm{u}_w^2 \eqdef \sum_{j=1}^d w_j u_j^2.$
}
\fi

\begin{proof}
First, we rewrite the left-hand side of \eqref{eq:87t87fdu8df} into the form
\begin{equation}\label{eq:807tyhfdf}
\norm{\sum_{i=1}^n u_i}^2  = \sum_{j=1}^d \left(\sum_{i=1}^n u_{ij} \right)^2  = \sum_{j=1}^d \left(\sum_{i\in S_j} u_{ij} \right)^2 .\end{equation}
Let $\xi\eqdef \max_j |S_j|$. Notice that by Jensen's inequality, we have
\begin{equation}\label{eq:y8yuhfd8fu}\left(\sum_{i\in S_j} u_{ij} \right)^2 \leq |S_j | \sum_{i\in S_j} u_{ij}^2 \leq \xi \sum_{i\in S_j} u_{ij}^2\end{equation}
for all $j \in [d]$. By combining the above observations, we can continue as follows:
\begin{eqnarray*}
\norm{\sum_{i=1}^n u_i}^2 \overset{\eqref{eq:807tyhfdf} + \eqref{eq:y8yuhfd8fu}}{ \leq}   
  \xi  \sum_{j=1}^d  \sum_{i\in S_j} u_{ij}^2 
=   \xi  \sum_{j=1}^d  \sum_{i=1}^n u_{ij}^2 
=   \xi  \sum_{i=1}^n  \sum_{j=1}^d u_{ij}^2 
=   \xi  \sum_{i=1}^n  \norm{u_i}^2 .
\end{eqnarray*}
\end{proof}

\subsection{Proof of Lemma~\ref{lem:c_bound}}

\begin{proof} Let $S_j\eqdef \{i \;:\; u_{ij} \neq 0\}$ for $j\in [d]$. Since 
$$S_j =\{i \;:\; u_{ij} \neq 0\}=\{i \;:\; u_{ij} = 0\}^c \subseteq \{i \;:\; (i,j) \in \cZ \}^c = \{i \;:\; (i,j) \notin \cZ \},$$ we have
\[ \max_j |S_j| \leq \max_j  |\{i \;:\; (i,j) \notin \cZ \}| \overset{\eqref{eq:sigma}}{=} \xc.\]
It remains to apply Lemma~\ref{lem:87t87fdu8df}.
\end{proof}

\subsection{Gradient estimate $g_i$ stays within its active subspace}
\begin{lemma} \label{lem:98y08fd} Choose any $i\in [n]$, $K_i\in [d]$ and $x\in \R^d$. If $g_i\in \R^d_i$, then  the vector
	\begin{equation}\label{eq:g_+_def}
	g_i^+ \eqdef g_i + \topk{K_i} (\nabla f_i(x) - g_i) 
	\end{equation}
	also belongs to $\R^d_i$.  
\end{lemma}

\begin{proof} 
Let us note that the complement set of $\cJ_i$, as indicated by the definition of $\cJ_i$ in Equation~\eqref{eq:cI_cJ_defs}, is given by
$$
\cJ_i^\complement = \{j \in [d] \;:\; (i, j) \in \cZ\}.
$$
We define $l_i = |\cJ_i^\complement|$, which represents the cardinality of $\cJ_i^\complement$. It is now observed that the lemma's statement is equivalent to demonstrating that $[g_i^+]_{\cJ_i^\complement}= \bold{0}^{l_i}$. In order to prove this, we further note that
\begin{equation}\label{eq:g_+_aux1}
[g_i^+]_{\cJ_i^\complement} \overset{\eqref{eq:g_+_def}}{=}  [g_i + \topk{K_i} (\nabla f_i(x) - g_i)]_{\cJ_i^\complement} = [g_i]_{\cJ_i^\complement} + [\topk{K_i} (\nabla f_i(x) - g_i)]_{\cJ_i^\complement},
\end{equation}
where the last equality follows from basic arithmetic principles. Since $g_i \in \mathbb{R}^d_i$, it holds that $[g_i]_{\cJ_i^\complement} = \bold{0}^{l_i}$. Regarding the argument of $\topk{K_i}$, we can express it as
\begin{align*}
[\nabla f_i(x) - g_i]_{\cJ_i^\complement} = [\nabla f_i(x)]_{\cJ_i^\complement} -[g_i]_{\cJ_i^\complement} = \bold{0}^{l_i} - \bold{0}^{l_i} = \bold{0}^{l_i},
\end{align*}
since $\nabla f_i(x) \in \mathbb{R}^d_i $ for any $x \in \mathbb{R}^d$. It remains to recall that the $\topk{K_i}$ operator either retains the element of a vector or maps it to zero. Thus, the zero sub-vector $[\nabla f_i(x) - g_i]_{\cJ_i^\complement}$ is mapped to a zero sub-vector. Consequently,
\begin{equation*}
[g_i^+]_{\cJ_i^\complement} \overset{\eqref{eq:g_+_aux1}}{=} [g_i]_{\cJ_i^\complement} + [\topk{K_i} (\nabla f_i(x) - g_i)]_{\cJ_i^\complement} = \bold{0}^{l_i} + \bold{0}^{l_i} = \bold{0}^{l_i},
\end{equation*}
what concludes the proof.
\end{proof}

\subsection{Proof of Lemma~\ref{lem:98y89fhd_8fd}}
\begin{proof}
By assumption, $g_i^0\in \R^d_i$. By repeatedly applying Lemma~\ref{lem:98y08fd}, we conclude that $g_i^{t} \in \R^d_i$ for all $t\geq 0$. Since $ \nabla f_i(x^t)$ belongs to $\R^d_i$, so does the vector $u_i^t \eqdef g_i^t - \nabla f_i(x^t)$. We now have 
\[  \norm{g^t -\nabla f(x^t)}^2 =  \norm{\frac{1}{n}\sum_{i=1}^n \left(g_i^t -\nabla f_i(x^t)\right)}^2 =\frac{1}{n^2}\norm{\sum_{i=1}^n u_i^t}^2 \leq  \frac{\xc}{n^2}\sum_{i=1}^n \norm{u_i^t}^2,\]
where in the last step we applied Lemma~\ref{lem:c_bound}.
\end{proof}

\subsection{New descent lemma}

\begin{lemma}\label{lm:descent_lemma_new}
	Let Assumption~\ref{as:smooth} hold. Furthermore, let $g_i^0 \in \R^d_i$ for all $i=1,2,\dots,n$. Let $$x^{t+1} = x^t - \gamma g^t$$ be the \algname{EF21}  method, where $g^t = \frac{1}{n}\sum_{i=1}^n g_i^t$, and $\gamma>0$ is the stepsize. Then \begin{equation}\label{eq:descent_lemma_aux3}
		f(x^{t+1}) \leq f(x^t) - \frac{\gamma}{2} \norm{\nabla f(x^t)}^2 - \left(\frac{1}{2\gamma} - \frac{L}{2} \right)\norm{x^{t+1} - x^t}^2 + \frac{\gamma}{2} \frac{\xc}{n} G^t.
	\end{equation}
\end{lemma}

We start with a standard result~\citep{PAGE2021}.
\begin{fact}\label{lm:descent_lemma_standard}
Suppose Assumption~\ref{as:smooth} holds, and let $x^{t+1} = x^t - \gamma g^t$, where $g^t \in \R^d$ is any vector, and $\gamma>0$ any scalar. Then \begin{equation}\label{eq:descent_lemma_aux3}
f(x^{t+1}) \leq f(x^t) - \frac{\gamma}{2} \norm{\nabla f(x^t)}^2 - \left(\frac{1}{2\gamma} - \frac{L}{2} \right)\norm{x^{t+1} - x^t}^2 + \frac{\gamma}{2} \norm{g^t - \nabla f(x^t)}^2.
\end{equation}
\end{fact}

\begin{proof}Lemma~\ref{lm:descent_lemma_new} follows by plugging the inequality from Lemma~\ref{lem:98y89fhd_8fd} into the inequality described by Fact~\ref{lm:descent_lemma_standard}.\end{proof}

\subsection{Bounding the gradient estimate error}

\begin{lemma}\label{lm:3pc-standard} The iterates of the \algname{EF21} method satisfy
	\begin{equation}\label{eq:3pc_ef21-st}
		G_i^{t+1} \leq (1 - \theta) G_i^t + \beta \norm{\nabla f_i(x^{t+1}) - \nabla f_i(x^t)}^2,
	\end{equation}
	and \begin{equation}\label{eq:g_t_parallel}
		G^{t+1} \leq (1 - \theta) G^t + \beta \Lplsq \norm{x^{t+1}-x^t}^2,
	\end{equation}
	where $\theta \eqdef 1 - \sqrt{1 - \alpha}$,   $\beta\eqdef \frac{1-\alpha}{1 - \sqrt{1-\alpha}}$, $\alpha \eqdef \min_i \alpha_i $ and $\alpha_i\eqdef \frac{\min\{K_i,|\cJ_i|\}}{|\cJ_i|}$. \end{lemma}

\begin{proof}

\begin{eqnarray*}
G_i^{t+1} & \overset{\eqref{eq:grad_variation_def}}{=} &\norm{g_i^{t+1} - \nabla f_i(x^{t+1})}^2 \\
&\overset{\text{Step 7 of Alg}~\ref{alg:ef21}}{=} & \norm{g_i^t + \topk{K_i}(\nabla f_i(x^{t+1}) - g_i^t) - \nabla f_i(x^{t+1})}^2 \\
&= & \norm{\topk{K_i}(\nabla f_i(x^{t+1}) - g_i^t) - (\nabla f_i(x^{t+1})-g_i^t )}^2 \\
&\overset{\eqref{eq:08y09fdd}}{\leq} & (1-\alpha_i) \norm{\nabla f_i(x^{t+1}) - g_i^{t}  }^2 \\
&\leq & (1-\alpha) \norm{\nabla f_i(x^{t+1}) - g_i^{t}  }^2 \\
&=& (1-\alpha) \norm{\nabla f_i(x^{t}) - g_i^{t} + \nabla f_i(x^{t+1}) - \nabla f_i(x^{t}) }^2 \\
&\leq& (1-\alpha) (1+\zeta) \norm{\nabla f_i(x^{t}) - g_i^{t}}^2 +  (1-\alpha) (1+\zeta^{-1})\norm{\nabla f_i(x^{t+1}) - \nabla f_i(x^{t}) }^2,
\end{eqnarray*}
where $\zeta>0$ is arbitrary. By choosing $\zeta = \frac{1}{\sqrt{1-\alpha}}-1$, we obtain \eqref{eq:3pc_ef21-st}.
%\[(1-\alpha)(1+\zeta) = \sqrt{1-\alpha}, \qquad (1-\alpha)(1+\zeta^{-1}) = \frac{1-\alpha}{1 - \sqrt{1-\alpha}}.\]
%
%\[\zeta = \frac{1}{\sqrt{1-\alpha}}-1\]
%\[1+\zeta^{-1} = 1+ \frac{\sqrt{1-\alpha}}{1-\sqrt{1-\alpha}} = \frac{1}{1-\sqrt{1-\alpha}}\]
To establish \eqref{eq:g_t_parallel}, we only need to observe that
\begin{eqnarray*}
G^{t+1} &\overset{\eqref{eq:grad_variation_def}}{=}& \avein G_i^{t+1} \\
& \overset{\eqref{eq:3pc_ef21-st}}{\leq} &\avein \left((1-\theta)G_i^t + \beta \norm{\nabla f_i(x^{t+1}) - \nabla f_i(x^t)}^2\right)\\
& =& (1-\theta) \avein G_i^t + \beta  \avein\norm{\nabla f_i(x^{t+1}) - \nabla f_i(x^t)}^2\\
&\overset{\eqref{eq:L+}+\eqref{eq:grad_variation_def}}{\leq} & (1-\theta) G^t + \beta \Lplsq \norm{x^{t+1} - x^t}^2.
\end{eqnarray*}
\end{proof}

\subsection{Auxiliary result on connection between $\sqrt{\frac{\beta}{\theta}}$ and $\alpha$}

\begin{lemma}
Let $\theta \eqdef 1 - \sqrt{1 - \alpha}$ and $\beta\eqdef \frac{1-\alpha}{1 - \sqrt{1-\alpha}}$, where $\alpha \in (0, 1]$. Then,
\[\sqrt{\frac{\beta}{\theta}}  = \frac{\sqrt{1-\alpha} + 1-\alpha}{\alpha}.\]
\end{lemma}

\begin{proof}
It immediately holds from the following arithmetical operations:
	\[\sqrt{\frac{\beta}{\theta}} = \sqrt{\frac{1-\alpha}{(1 - \sqrt{1-\alpha})^2}} = \frac{\sqrt{1-\alpha}}{1 - \sqrt{1-\alpha}} = \frac{\sqrt{1-\alpha}(1 + \sqrt{1-\alpha})}{(1 - \sqrt{1-\alpha})(1 + \sqrt{1-\alpha})} = \frac{\sqrt{1-\alpha} + 1-\alpha}{\alpha}.\]
\end{proof}

\subsection{Convergence rate in the fully separable case}

If $\xc=1$, we can express the problem~\eqref{eq:main_problem} in a more concise form:
\begin{equation}\label{eq:fullsep_problem}
	\min_{x \in \mathbb{R}^d} \left[ f(x) = \frac{1}{n} \sum_{i=1}^n f_i(x_i)\right],
\end{equation}
where each function $f_i$ has a support of size $d_i$, and the sum of all $d_i$ values equals $d$.
We will now establish the following claim.

\begin{lemma}\label{lm:L_full_sep}
Let~\Cref{as:L_i} hold. Then~\Cref{as:smooth} holds for~\eqref{eq:fullsep_problem}  with $L = \frac{\max_i L_i}{n}$. 
\end{lemma}
\begin{proof}
For the sake of clarity and ease of presentation, we make the assumption that function $f$ is twice differentiable. We observe that the Hessian matrix of $f(x)$ is a block-separable matrix, given by $\nabla^2 f(x)~=~\avein \nabla^2 f_i(x_i)$ for any $x \in \RR^d$. To establish the main claim of the lemma, we aim to show that $\norm{\nabla^2 f(x)} \leq \frac{\max_i L_i}{n}$~\citep{nesterov2018lectures}.

Let $v \in \mathbb{R}^d$ such that $\norm{v}^2 = 1$. We represent the vector $v = [v_1^\top v_2^\top \dots v_n^\top]^\top$. It follows that
$$
v^\top \nabla^2 f(x) v = \avein v_i^\top \nabla^2 f_i(x_i) v_i,
$$
due to the block-separable structure of the Hessian matrix.  We proceed to
\begin{equation}\label{eq:block_sep_matrix}
v^\top \nabla^2 f(x) v = \avein v_i^\top \nabla^2 f_i(x_i) v_i \leq \avein \norm{\nabla^2 f_i(x_i)} \norm{v_i}^2  \leq \avein L_i \norm{v_i}^2  = \sumin \norm{v_i}^2 \cdot \frac{L_i}{n},
\end{equation}
where the first inequality follows from the definition of the operator norm, and the second inequality follows from the second-order definition of the function smoothness \citep{nesterov2018lectures}.

 Since $1 = \norm{v}^2 = \sumin \norm{v_i}^2$, we can interpret~\eqref{eq:block_sep_matrix} as a convex combination of $n$ positive numbers. As the convex combination never exceeds the value of its maximum term, we finally obtain
\begin{equation*}
v^\top \nabla^2 f(x) v \overset{\eqref{eq:block_sep_matrix}}{\leq} \sumin \norm{v_i}^2 \cdot \frac{L_i}{n} \leq \frac{\max_i L_i}{n},
\end{equation*}
which concludes the proof.
\end{proof}

We observe that applying Algorithm \ref{alg:ef21} to the problem \eqref{eq:fullsep_problem} is equivalent to performing $n$ independent runs of Algorithm \ref{alg:ef21} in a single-node scenario, where each run pertains to its own $i$-th block. Each individual run requires $\cO\left(\frac{L_i (1 + \sqrt{\frac{\beta}{\theta}})}{\delta}\right)$ iterations to achieve an accuracy of $\delta$ \citep{EF21}. Consequently, when considering all $n$ runs, the total number of iterations $T'$ required to attain $\delta$-accuracy by each client is given by $T' = \cO\left(\frac{\max_i L_i (1 + \sqrt{\frac{\beta}{\theta}})}{\delta}\right)$ in a parallel computing setting.

Since the functions $f_i$ are effectively independent of each other, after $T'$ iterations, the function $f(x)$ satisfies
$$
\norm{\nabla f(x^{T'})}^2 = \norm{\avein f_i(x_i^{T'})}^2 =   \frac{1}{n^2} \sumin \norm{\nabla f_i(x^{T'})}^2 \leq \frac{\delta}{n} = \varepsilon.
$$
Consequently, to achieve an accuracy of $\varepsilon$ for the function $f(x)$, each parallel run should be executed for $$\cO\left(\frac{\max_i L_i (1 + \sqrt{\frac{\beta}{\theta}})}{\varepsilon n}\right) \overset{\text{\Cref{lm:L_full_sep}}}{=} \cO\left(\frac{L + \frac{\max_i L_i}{n} \sqrt{\frac{\beta}{\theta}}}{\varepsilon}\right)$$ iterations, which aligns with the convergence result presented in Theorem~\ref{thm:convergence_separate}.

\newpage
\section{Proof of Theorem~\ref{thm:convergence_separate}}
\begin{proof}
Define the Lyapunov function\begin{equation}\label{eq:lyapunov_proof_separate}
\Psi^t \eqdef f(x^t) - f^\ast + \frac{\gamma \xc}{2 \theta n} G^t.
\end{equation}

By straightforward arguments, we get
\begin{eqnarray*}
\Psi^{t+1} &\overset{\eqref{eq:lyapunov_proof_separate}}{=} & f(x^{t+1}) - f^\ast + \frac{\gamma}{2\theta n} G^{t+1} \\
& \overset{\eqref{lm:descent_lemma_new}}{\leq}& f(x^t) - f^\ast -  \frac{\gamma}{2}\norm{\nabla f(x^t)}^2 - \left(\frac{1}{2\gamma} - \frac{L}{2}\right) \norm{x^{t+1}-x^t}^2 + \frac{\gamma \xc}{2n} G^t + \frac{\gamma \xc}{2 \theta n} G^{t+1} \\
& \overset{\eqref{eq:g_t_parallel}}{\leq} & f(x^t) - f^\ast -  \frac{\gamma}{2}\norm{\nabla f(x^t)}^2 - \left(\frac{1}{2\gamma} - \frac{L}{2}\right) \norm{x^{t+1}-x^t}^2 + \frac{\gamma \xc}{2n} G^t  \\
&& \qquad + \frac{\gamma \xc}{2\theta n} \left( (1 - \theta) G^t + \beta \Lplsq \norm{x^{t+1}-x^t}^2\right)\\
&=& f(x^t) - f^\ast +  \frac{\gamma}{2}\left(\frac{\xc}{n} + (1-\theta)\frac{\xc}{\theta n}\right)  G^t - \frac{\gamma}{2}\norm{\nabla f(x^t)}^2 \notag \\
&& \qquad - \underbrace{\left(\frac{1}{2\gamma} - \frac{L}{2} - \frac{\gamma \beta \xc \Lplsq}{2\theta n}\right)}_{\geq 0} \norm{x^{t+1}-x^t}^2 \\
& \leq &f(x^t) - f^\ast + \frac{\gamma}{2} \left(\frac{\xc}{n} + (1-\theta)\frac{\xc}{\theta n}\right) G^t - \frac{\gamma}{2}\norm{\nabla f(x^t)}^2 \\
&=& f(x^t) - f^\ast +  \frac{\gamma\xc}{2\theta n}  G^t - \frac{\gamma}{2}\norm{\nabla f(x^t)}^2 \\
& \overset{\eqref{eq:lyapunov_proof_separate}}{=} & \Psi^t - \frac{\gamma}{2}\norm{\nabla f(x^t)}^2.
\end{eqnarray*}

%\[\frac{\xc}{n} + (1-\theta)C = C \Leftrightarrow  \frac{\xc}{n} = \theta  C \Leftrightarrow  C = \frac{\xc}{ \theta n} . \]

Unrolling the inequality above, we get
\begin{equation}
0 \leq \Psi^T\leq  \Psi^{T-1} - \frac{\gamma}{2}\norm{\nabla f(x^{T-1})}^2 \leq \Psi^0 - \frac{\gamma}{2} \sum\limits_{t=0}^{T-1} \norm{\nabla f(x^t)}^2,
\end{equation}
and the result follows.
\end{proof}

\newpage
\section{Additional Details of Experiments}

Our code is available here \url{https://github.com/burlachenkok/ef21_with_rare_features}.

\subsection{Linear regression on sparse data}
\label{app:lin_reg_on_sparse}

In our synthetic experiments, we consider the minimization of the function $$f(x) = \frac{1}{n} \sum_{i=1}^{n} f_i(x),$$ where $f_i(x) = \dfrac{1}{m} \norm{\bA_i x - {b_i}}^2 + \phi(x)$ and and we choose $\phi\equiv 0$. Therefore, $$ \nabla f_i(x) = \dfrac{2}{m}\left( \bA_i^\top \bA_i x - \bA_i^\top b_i \right).$$
\paragraph{The function $f(x)$ satisfies~\Cref{as:L+}.} The statement follows from since for all $x,y \in \mathbb{R}^d$, we have
\begin{eqnarray*}
	\avein \norm{\nabla f_i(x) - \nabla f_i(y)}^2 &=&
	\avein \norm{\dfrac{2}{m} \bA_i^\top \bA_i x  - \dfrac{2}{m}\bA_i^\top \bA_i y }^2 \\
	& = &\avein \frac{4}{m^2} \norm{\bA_i^\top \bA_i (x - y)}\\
	& = &(x-y)^\top \left[ \frac{4}{m^2 n} \sumin (\bA_i^\top \bA_i)^2 \right] (x-y)\\
	& \leq & \Lplsq \norm{x - y}^2.
\end{eqnarray*}

From this we can conclude that for this problem functions  $f_1(x),\dots,f_n(x)$ satisfy Equation \eqref{eq:L+} with:
\[\Lplsq = \frac{4}{m^2 n} \lambda_{\max} \left(\sumin (\bA_i^\top \bA_i)^2 \right).\]
In our experiments, we fixed the dimension $d=500$ and the number of clients $n=100$. For the analysis, we designed a controlled way to generate instances of synthetic quadratic optimization problems with a desired sparsity pattern. The generation of an instance of optimization problems is driven by the main meta parameter $\xc/n$, and the auxiliary meta-parameter $v$ which affects the distribution of $L_i$.

\paragraph{Ensuring $\Lplsq \ll \tilde{L}^2$.}
To fully demonstrate the efficacy of the new theoretical results, it is desirable to enforce a significant difference between $\Lplsq$ and $\tilde{L}^2$. The standard theory assumes that $\Lplsq = \tilde{L}^2$~\citep{EF21}, but the refined \Cref{lem:M-bound-via-c} allows for the reduction of $\Lplsq$. Assuming that $\max_i L_i^2$ is attained at index $j'$, if we modify $L_i$ for $i\ne j'$ such that $L_i \le L_{j'}$, the left-hand side of the expression
$$ \min \left\{ \sqrt{\frac{\xc  \max_i L_i^2}{n}}, \sqrt{\frac{\sum_{i=1}^n L_i^2}{n} } \right\}.
$$ remains unaffected, while the second term, equal to $\tilde{L}$, is influenced. Clearly, for a fixed value of $\max_i L_i^2=L_{j'}^2$, in order to maximize $\tilde{L}$, we need to select $L_i = L_j, \forall i \in [n]$. The instances of a quadratic optimization problem with such a property demonstrate the greatest advantage of the new theory.

On the other hand, if we ask the question when our analysis does not bring a big improvement over the standard \algname{EF21} analysis, this is the case when $\max_i L_i^2 = \sum_{i=1}^{n}L_i^2$. This situation is attained when $L_{j'}$ is constant and $L_i=0$ for $i \ne j'$.

\paragraph{Main meta parameter $\xc/n$.}
From Lemma \ref{lem:M-bound-via-c}, we see that $\frac{\xc}{n}$ plays an important role in the multiplier in Lemma \ref{lem:M-bound-via-c} and comes into the denominator of the stepsize in  \Cref{thm:convergence_separate}. Firstly, we can observe from Equation \eqref{eq:sigma} that the minimum possible value of $\frac{\xc}{n}=\frac{1}{n}$, and the maximum possible value of $\frac{\xc}{n}=1$ is attained for $\xc=n$. So $\frac{\xc}{n} \in \left[\frac{1}{n}, 1\right]$. We provide a controllable way to specify this parameter. The main intuition behind this parameter is the following. As $\frac{\xc}{n}$ is smaller, it is more advantageous for our method compared to the standard \algname{EF21}. And in this case, there is a serious hope to observe in practice that our strategy of selecting the step size demonstrates better results

\paragraph{Auxiliary parameter $v$.} The meta-parameter $v\in[0,1]$ allows selecting between two extreme distributions of $L_i$ in context of Lemma \ref{lem:M-bound-via-c}. One extreme point is when all $L_i$ attains the same constant values $L_c$ (can be selected arbitrarily, but we have selected $L_c=20.0$). This distribution of $L_i$ corresponds to $v=0$ (\textit{and is preferable for our analysis}). 

Another extreme point is where $L_1=10.0$ and $L_i=0, \forall i \ne 1$. This distribution of $L_i$ corresponds to $v=1$ (\textit{and is not preferable for our analysis}). Finally, in the case of using values $v \in (0,1)$ the distribution of $L_i$ across clients will be linearly interpolated between these two distributions corresponding to the two cases described above.

In our experiments, we set $v = 0.1$.

\paragraph{Controlling the sparsity parameter $\xc$.} The process of dataset generation starts with constructing a matrix $\bS$ with $n$ rows and $d$ columns with $\bS_{i,j} \in \{0, 1\}$. We set $\bS_{i,j}=1$ when client $i \in [n]$ depends on the coordinate $j \in [d]$, otherwise we $\bS_{i,j}=0$. The filling of $\bS$ happens column-wise. The columns $s_j \in \mathbb{R}^n$ of the matrix $\bS$ are filled with values $1$ in positions corresponding of a random subset of $[d]$ of cardinality $\xc$ chosen uniformly at random. If after processing all columns there exists a client that depends on $0$ coordinates, the strategy of filling is restarted. In the logic of our generation algorithm, we use $5$ attempts to create a valid filling. If the dataset sparsity generation procedure fails after all attempts, we report the failure of the dataset generation process.

\paragraph{Generating datasets.}
Each client has a loss function $f_i(x):\mathbb{R}^d \to \mathbb{R}$. However, due to the previous construction of the sparsity pattern, the client $i$ depends on the coordinates $\{j: \bS_{i,j}=1\}$. After renaming variables and ignoring variables that $f_i(x)$ does not depend on, we define $f_i(z)$ as:
\[ f_i(z) \eqdef \frac{1}{n_i} \norm{{\bA}_i \cdot x(z) - b_i}^2.\]
Next, we generate a uniform spectrum $[1.0, 20.0]$ and fill ${\bA}_i$ in such a way that the spectrum $\lambda\left(\frac{2}{n_i} {\bA}_i^\top {\bA}_i\right)$ is represented by a linear interpolation controlled by the meta parameter $v$ from the uniform spectrum $[1.0,20.0]$ to $[L_c, 0.0, \dots, 0.0]$. After constructing ${\bA}_i$, we set $b_i~\eqdef~{\bA}_i \cdot x_{\mathrm{solution}} + {noise}_i$. The $x_{\mathrm{solution}}$ plays the role of a prior known solution, and ${noise}_i \sim \mathcal{U}_{[-1,1]}\cdot p$ is additive noise in the linear model, where $p\in \mathbb{R}$ is a fixed constant. In our experiments, $p=2.0$. It plays the role of a perturbation that scales the standard deviation of the zero mean r.v. ${noise}_i$. It helps to escape the interpolation regime, i.e., situation in which $\nabla f_i(x^*)=0$ for all $i \in [n]$.

\subsection{Logistic regression with adaptive stepsize}
\label{app:log_reg_adapt}

%\peter{To be done}

In this section, we provide additional numerical experiments in which we compare \algname{EF21} under the standard analysis and our analysis. We address the problem of training a binary classifier via a logistic model on several \texttt{LIBSVM} datasets \citep{chang2011libsvm}.

\paragraph{Computing and software environment.} We used the Python software suite \texttt{FL\_PyTorch} \citet{burlachenko2021fl_pytorch} to simulate the distributed environment for training. We trained logistic regression across $n=300$ clients in the experiments below. We ran the experiments on a compute node with Ubuntu 18.04 LTS, 251 GBytes of DRAM memory, and 48 cores of Intel(R) Xeon(R) Gold 6246 CPU @ 3.30GHz. We used single precision (FP32) arithmetic.

\paragraph{Experiment setup.}
We conducted distributed training of a logistic regression model on \texttt{A9A}, \texttt{MUSHROOMS}, \texttt{W5A}, \texttt{PHISHING} datasets. This setting is achieved by specifying for Equation \eqref{eq:main_problem} the functions $f_i(x)$ as:
\begin{eqnarray*}
f_i(x) \eqdef \dfrac{1}{n_i} \sum_{j=1}^{n_i} \log \left(1+\exp({-y_{ij} \cdot a_{ij}^{\top} x})\right), \qquad
(a_{ij},  y_{ij}) \in \mathbb{R}^{d} \times \{-1,1\}.
\end{eqnarray*}
The initial shifts for \algname{EF21} are $g_i^0=0, \forall i \in [n]$. All datasets are randomly reshuffled and spread across clients in such a way that each client stores the same amount of data points $n_i$; the residual is discarded. The initial iterate $x^0$ is initialized as $\mathcal{U}_{[-\sqrt{1/D}, \sqrt{1/D}]}$ according to the default initialization of a linear layer in PyTorch \footnote{\href{https://pytorch.org/docs/stable/generated/torch.nn.Linear.html}{Information about initialization  linear layer torch.nn.Linear}}. For standard \algname{EF21} we used the largest step size allowed by its theory.

\paragraph{Reasons of sparse features.}

In addition to what we have already mentioned in Section \ref{sec:sparsity} we would like to highlight reasons of appearing rare features in the training based on our experience:

First, when the input for ML models is a categorical value from a finite set $S$, not a real number, a special conversion is needed. If $S$ has no natural order, the conversion usually maps each $s\in S$ to a one-hot vector $\tilde{s} \in \mathbb{R}^k$, where $\tilde{s}$ has only one non-zero element. This conversion has drawbacks, such as introducing an artificial partial order in $\mathbb{R}^k$. It is used for models that cannot handle categorical inputs directly without conversion. Examples of models are Neural Nets and Linear Models.

Second, some features may be inherently sparse vectors in the application. For example, if $a$ encodes a voxel grid of solid geometrical physical objects, it will be a sparse vector in most applications.

Third, during the modeling stage, there may be a specific pattern called "\textit{feature template}", which defines how a family of close-by features is evaluated. This technique is often used in applications where ML is applied for complex tasks that require defining the input features as part of the problem, and they are not given in advance, e.g. because there is no established practice for specific tasks.

\paragraph{Practical applicability of our analysis for a case when $\xc=n$.} 

The \texttt{LIBVSM} datasets are mostly sparse datasets. As we explained above, this is not an unrealistic assumption. However, in practice, the Definition \ref{eq:sigma} may be too strict. According to this definition, adding a bias (or intercept) term to the Machine Learning model during training leads to $\xc=n$. We visualize the sparsity patterns in \texttt{A9A}, \texttt{MUSHROOMS}, \texttt{PHISHING},
\texttt{W5A} in \Cref{fig:c_patterns_libsvm_ds_fig}. This representation is obtained after uniformly shuffling the original train datasets and splitting them across $n=300$ clients.

In this experiment, we consider the setting where the Master executes Algorithm \ref{alg:ef21}, but with a varying step size $\gamma^t$ in Line 4.

In our modification, we use the maximum theoretical step size from Theorem \ref{thm:convergence_separate}, but we define the parameter $\xc$ based on Lemma \ref{lem:98y89fhd_8fd} because it reflects the original notion of $\xc$ when training Machine Learning models with intercept terms that lead to $\xc=n$. Moreover, we define the quantity $\Lpl$ based on Lemma \ref{lem:M-bound-via-c}. We summarize the rules for executing the adaptive version of \algname{EF21} as follows:
\begin{gather}
\xct \eqdef \dfrac{\norm{g^t-\nabla f(x^t)}^2}{G^t/n}, \nonumber \\
\Lpl \eqdef \min \left\{ \sqrt{\frac{\xct  \max_i L_i^2}{n}} , \sqrt{\frac{\sum_{i=1}^n L_i^2}{n} } \right\}, \nonumber \\
\gamma \eqdef \gamma^t \eqdef \frac{1}{L + \Lpl \sqrt{\frac{ \xct }{n}   } \frac{\sqrt{1-\alpha} + 1-\alpha}{\alpha} }.
\label{eq:practical_modif_ef21}
\end{gather}
The system of rules defined in system of equations \eqref{eq:practical_modif_ef21} raises several questions:

\begin{enumerate}
	\item How can we estimate the quantity $\xct$? To estimate $\xct$ from this definition, we need to be able to estimate $\norm{g^t - \nabla f(x^t)}^2$ in the master, but the vector $\nabla f(x^t)$ is not available in the master.
	\item How can we analyze the convergence of $\Psi^t$ when $\xct$ varies? If we allow $\xct$ to change during the optimization process, then the Lyapunov function from Equation \eqref{eq:lyapunov_proof_separate} also changes over time. And varying $\xc$ will make it very hard to analyze the behavior of $\Psi^t$.	 	
\end{enumerate}

We will not address these questions in our experiment below. We believe that a deep understanding of these questions is the subject of future research. The purpose of this experiment is to demonstrate that the notion of $\xct$ can open new opportunities for research in this direction.

\paragraph{Results and Conclusion.}
We present the results in Figure \ref{fig:exp_libsvm_ds}. In datasets \texttt{A9A} (a), \texttt{PHISHING} (b), \texttt{MUSHROOMS} (c), we can increase the step size by a factor of $10\mathrm{x}$ using our proposed scheme. In dataset \texttt{W5A} (d), we can only increase the step size by a small factor of $1.35\mathrm{x}$, which suggests the need for more refined analysis. In all experiments, the \algname{EF21} with our approximate scheme performs better than the standard \algname{EF21}. We do not observe any convergence or stability issues in any of the experiments. We also show the behavior of $\xct$, which is not exploited by the standard \algname{EF21}. We hope this experiment will inspire future research in the direction of adaptive $\xct$.

\begin{figure*}[t]
	\centering
	\captionsetup[sub]{font=tiny,labelfont={}}	
	\captionsetup[subfigure]{labelformat=empty}
	%\captionsetup{position=top}
	
	\begin{subfigure}[ht]{0.32\textwidth}
		\includegraphics[width=\textwidth]{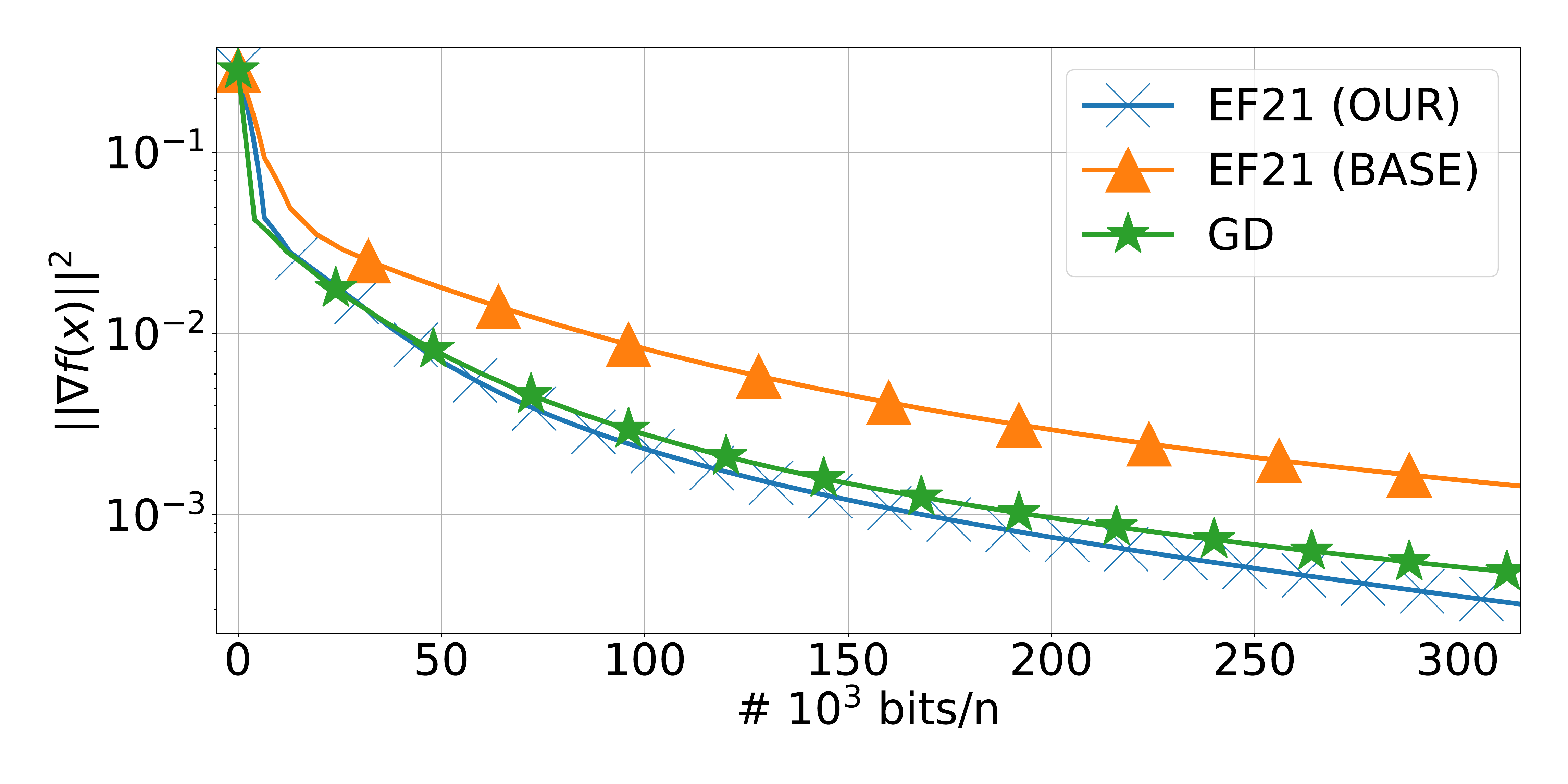} \caption{}
	\end{subfigure}
	\begin{subfigure}[ht]{0.32\textwidth}
		\includegraphics[width=\textwidth]{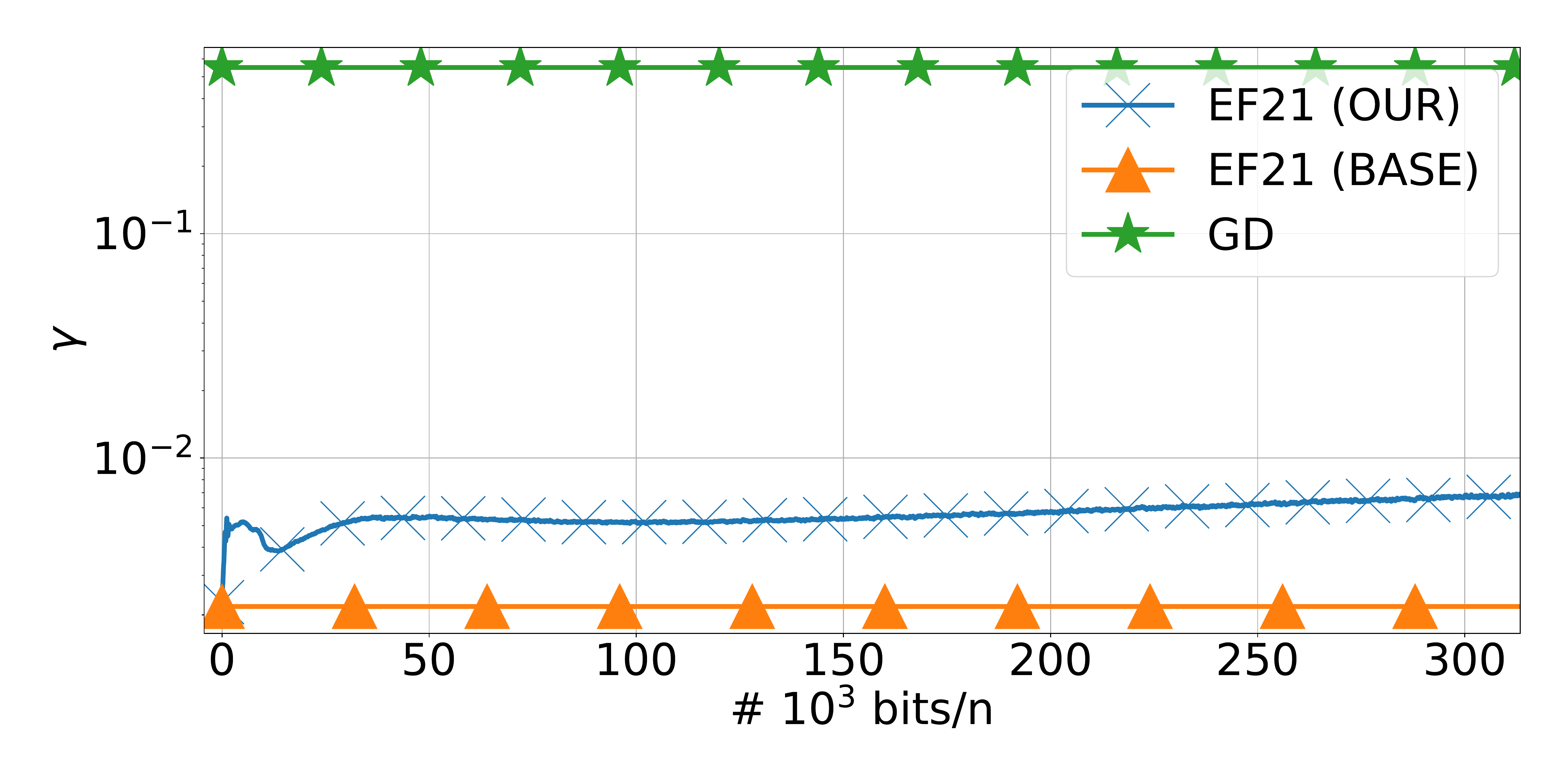} \caption{(a) A9A, $32561$ datapoints, $d=125$.}
	\end{subfigure}
	\begin{subfigure}[ht]{0.32\textwidth}
		\includegraphics[width=\textwidth]{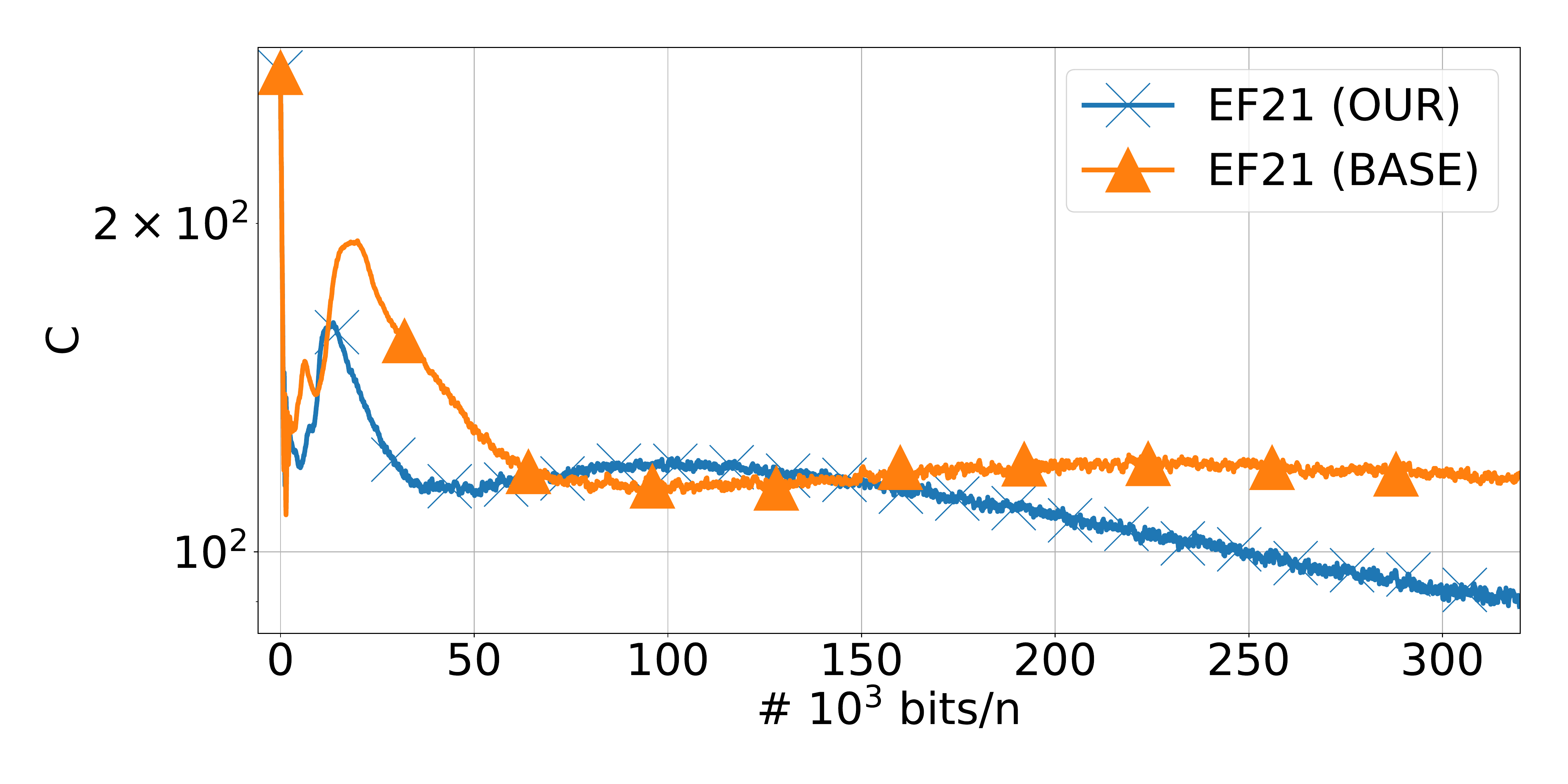} \caption{}
	\end{subfigure}
		
	\begin{subfigure}[ht]{0.32\textwidth}
		\includegraphics[width=\textwidth]{./figs/realexp/phishing/gradqr.pdf} \caption{}
	\end{subfigure}
	\begin{subfigure}[ht]{0.32\textwidth}
		\includegraphics[width=\textwidth]{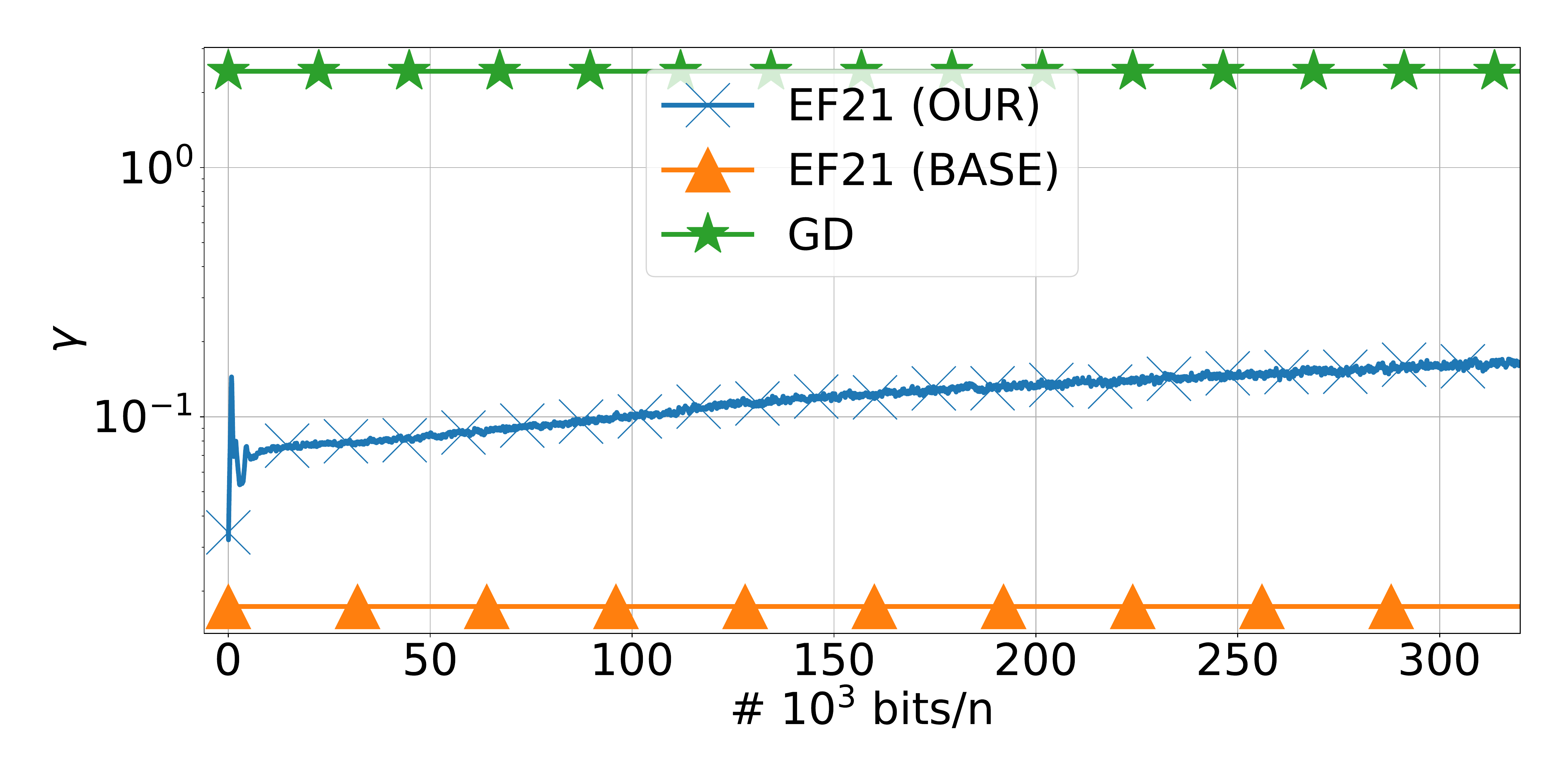} \caption{(b) PHISHING, $11055$ datapoints, $d=70$.}
	\end{subfigure}
	\begin{subfigure}[ht]{0.32\textwidth}
		\includegraphics[width=\textwidth]{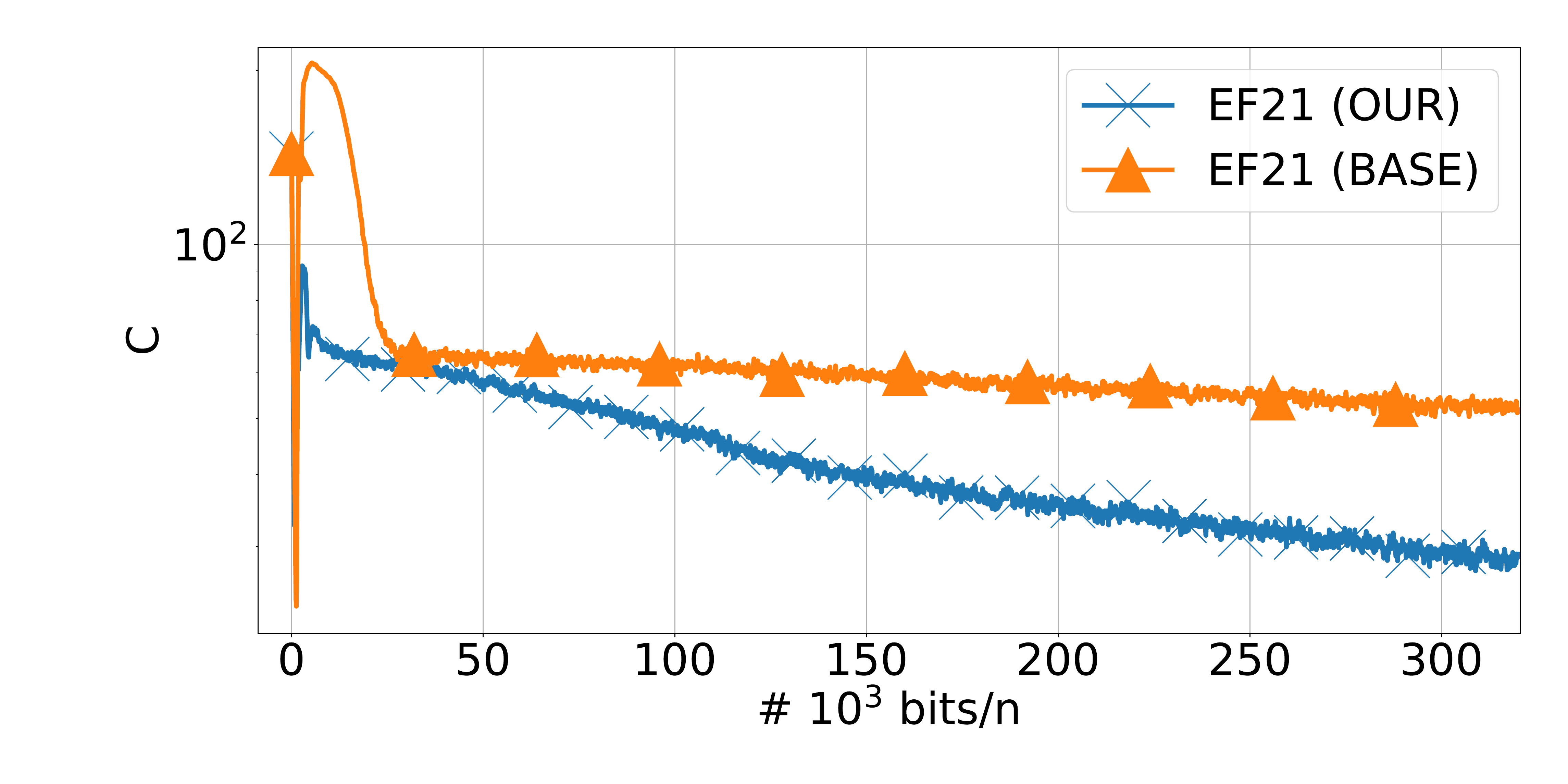} \caption{}
	\end{subfigure}

	\begin{subfigure}[ht]{0.32\textwidth}
		\includegraphics[width=\textwidth]{./figs/realexp/mushrooms/gradqr.pdf} \caption{}
	\end{subfigure}
	\begin{subfigure}[ht]{0.32\textwidth}
		\includegraphics[width=\textwidth]{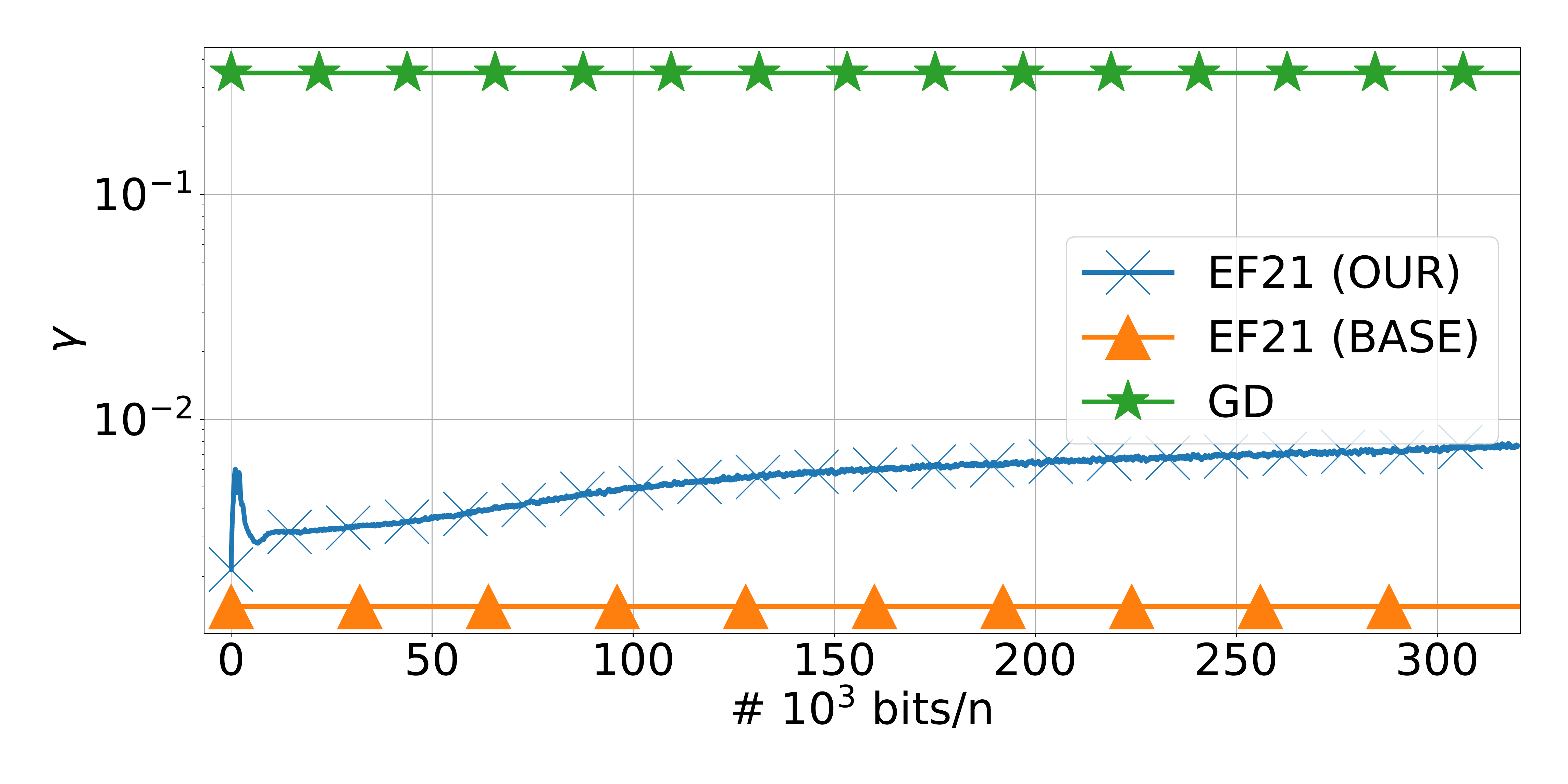} \caption{(c) MUSHROOMS, $8124$ datapoints, $d=114$}
	\end{subfigure}
	\begin{subfigure}[ht]{0.32\textwidth}
		\includegraphics[width=\textwidth]{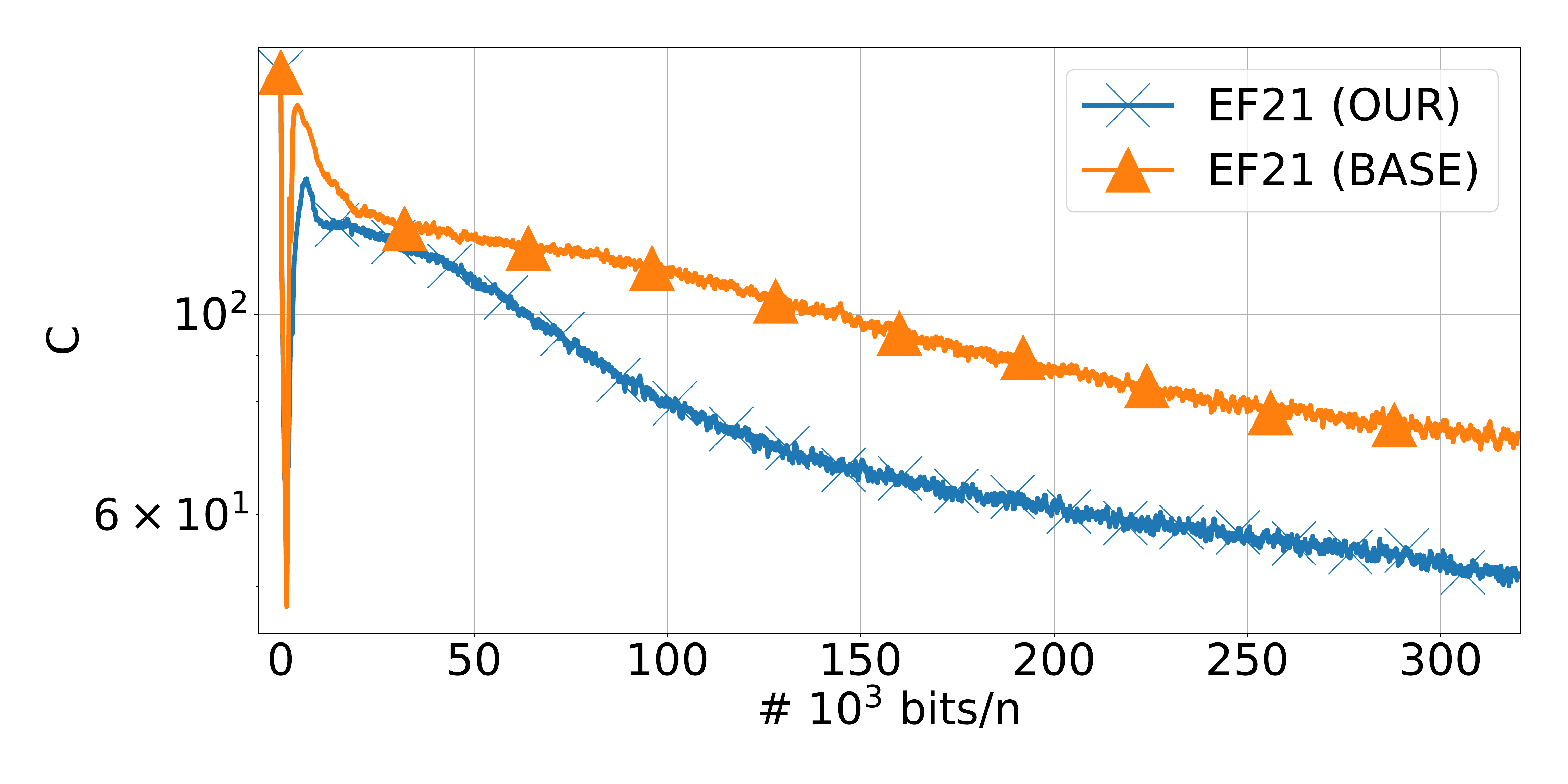} \caption{}
	\end{subfigure}
	
	\begin{subfigure}[ht]{0.32\textwidth}
		\includegraphics[width=\textwidth]{./figs/realexp/w5a/gradqr.pdf} \caption{}
	\end{subfigure}
	\begin{subfigure}[ht]{0.32\textwidth}
		\includegraphics[width=\textwidth]{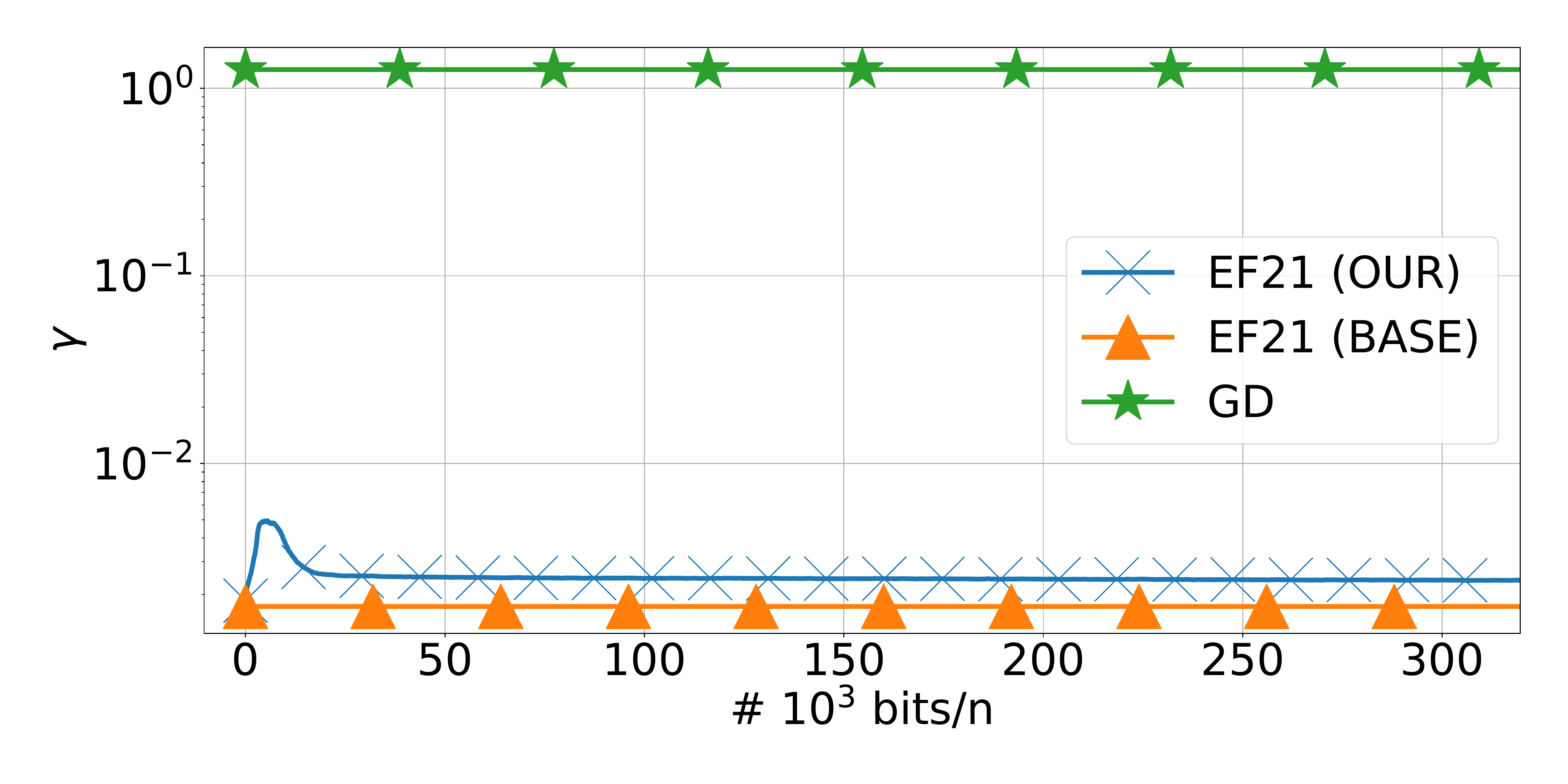} \caption{(d) W5A, $9888$ datapoints, $d=302$.}
	\end{subfigure}
	\begin{subfigure}[ht]{0.32\textwidth}
		\includegraphics[width=\textwidth]{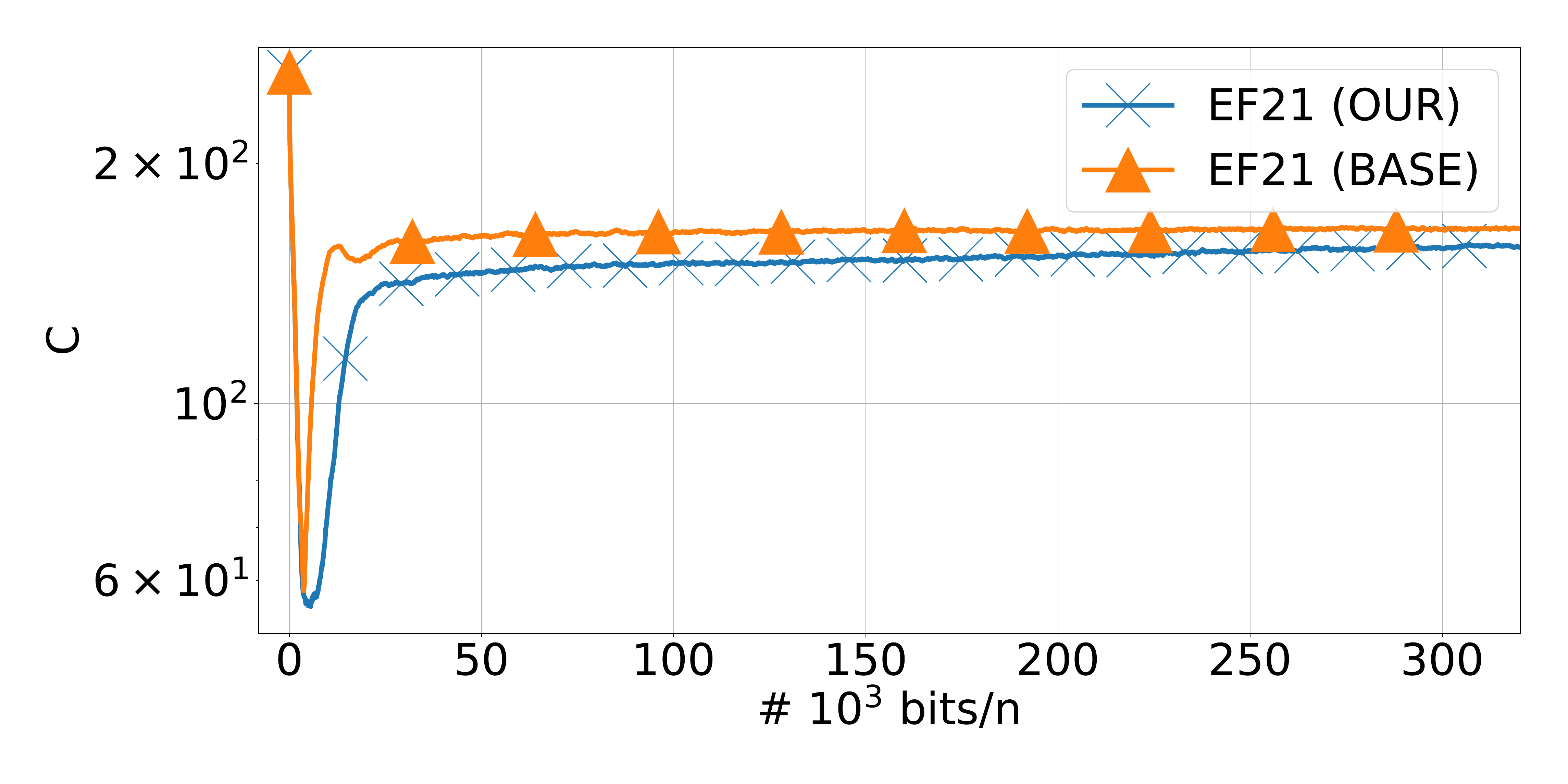} \caption{}
	\end{subfigure}
	
	%\vspave{-1pt}
	\caption{Training logistic regression model across $n=300$ client. Dimension of problem $d$ is presented in the plots, $\xct$ is the behavior of modified notion of variable $\xc$, $\gamma$ is the value of used step size. Computation is carried out in FP32. Full client participation. The step size used for standard \algname{EF21} and \algname{GD} are theoretical. Used client compressors for \algname{EF21} algorithms are \topk{1}.}
	\label{fig:exp_libsvm_ds}
\end{figure*}

%\begin{figure*}[t]
%	\centering
%	\captionsetup[sub]{font=scriptsize,labelfont={}}	
%	\captionsetup[subfigure]{labelformat=empty}
	%\captionsetup{position=top}
	
%	\begin{subfigure}[ht]{0.32\textwidth}
%		\includegraphics[width=\textwidth]{./figs/realexp/a9a/iterates.pdf} \caption{(a) A9A dataset.}
%	\end{subfigure}
%	\begin{subfigure}[ht]{0.32\textwidth}
%		\includegraphics[width=\textwidth]{./figs/realexp/phishing/iterates.pdf} \caption{(c) PHISHING dataset.}
%	\end{subfigure}	
%	\begin{subfigure}[ht]{0.32\textwidth}
%		\includegraphics[width=\textwidth]{./figs/realexp/mushrooms/iterates.pdf} \caption{(b) MUSHROOMS dataset.}
%	\end{subfigure}

%	\begin{subfigure}[ht]{0.32\textwidth}
%		\includegraphics[width=\textwidth]{./figs/realexp/w5a/iterates.pdf} \caption{(d) W5A dataset.}
%	\end{subfigure}
	
%	\caption{Training Logistic Regression model across $n=300$ client. Dimension of problem presented in the plots. Behavior of iterates during training process from Fig. \ref{fig:exp_libsvm_ds}.}	
%	\label{fig:exp_libsvm_ds_iterates}
%\end{figure*}

\begin{figure*}[t]
	\centering
	\captionsetup[sub]{font=scriptsize,labelfont={}}	
	\captionsetup[subfigure]{labelformat=empty}
	
	\begin{subfigure}[ht]{0.33\textwidth}
		\includegraphics[width=\textwidth]{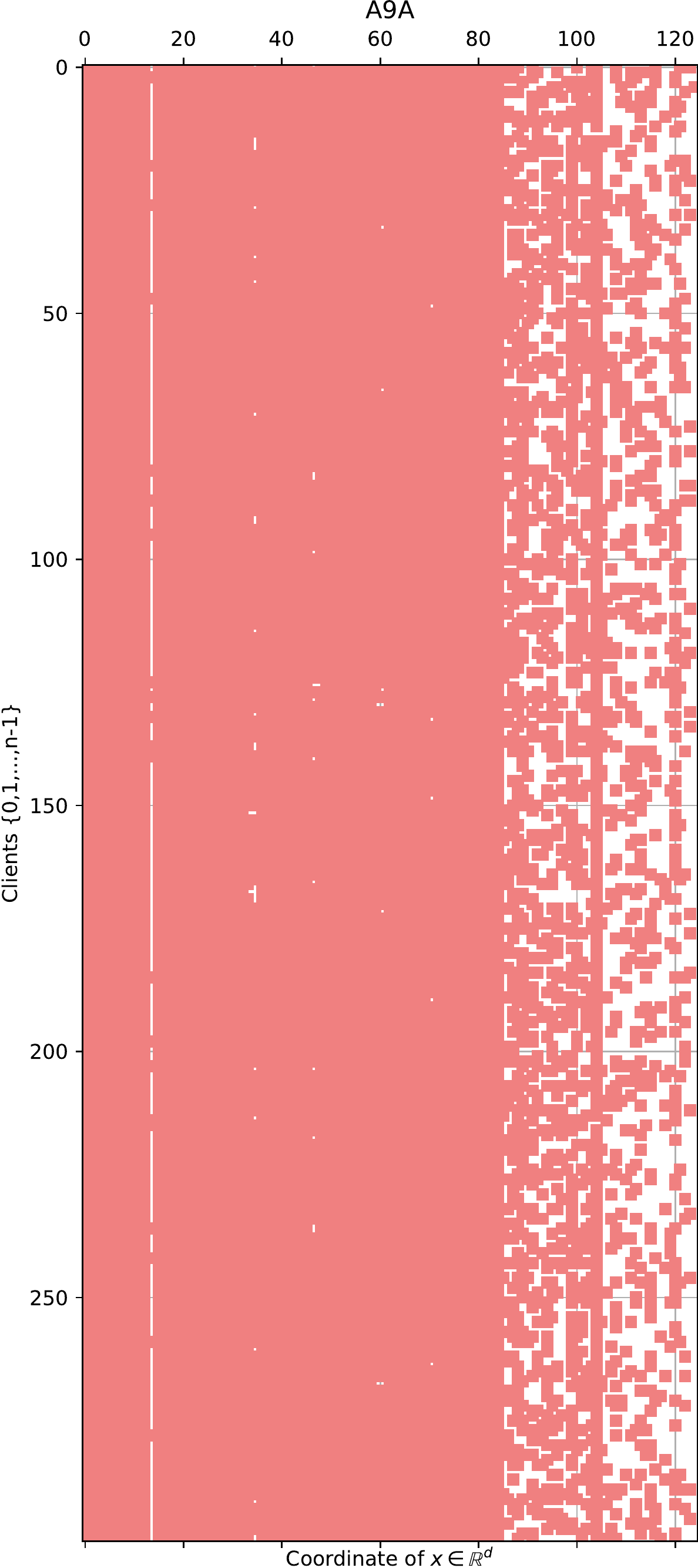} \caption{}
	\end{subfigure}
	\begin{subfigure}[ht]{0.31\textwidth}
		\includegraphics[width=\textwidth]{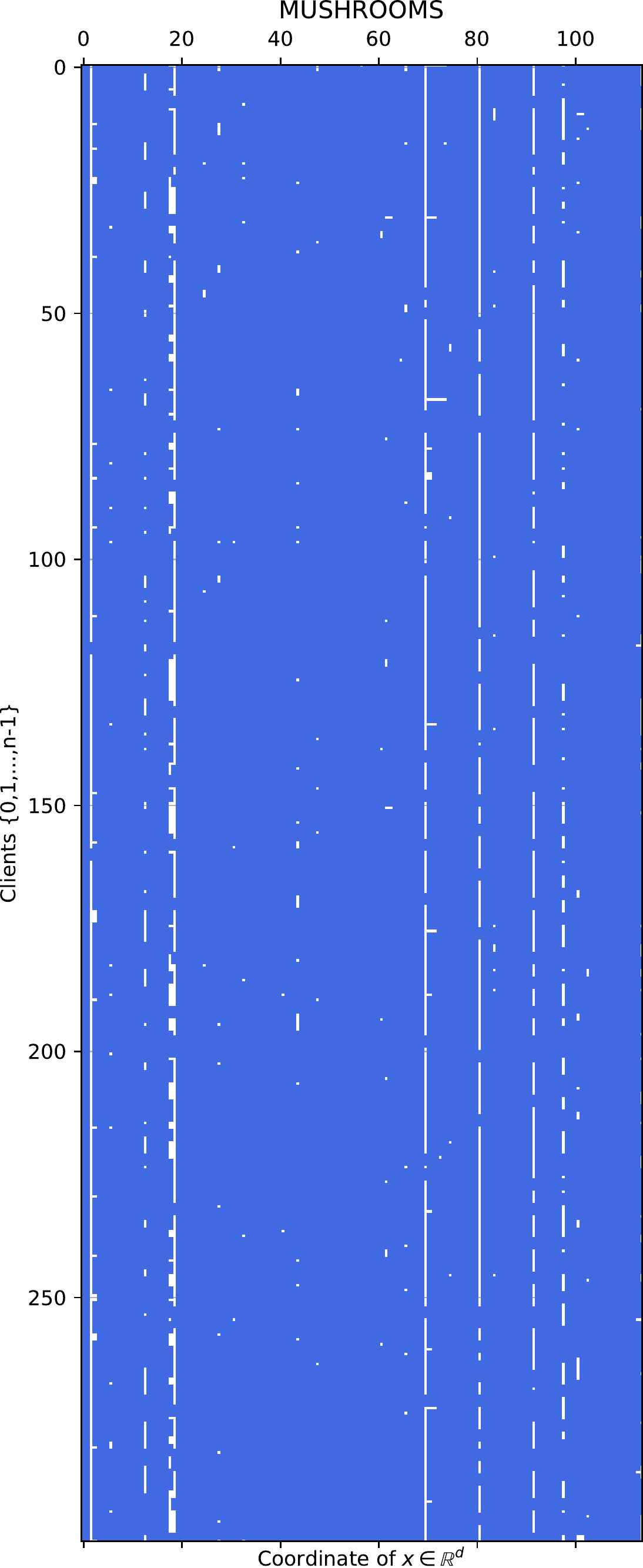} \caption{}
	\end{subfigure}
	\begin{subfigure}[ht]{0.19\textwidth}
		\includegraphics[width=\textwidth]{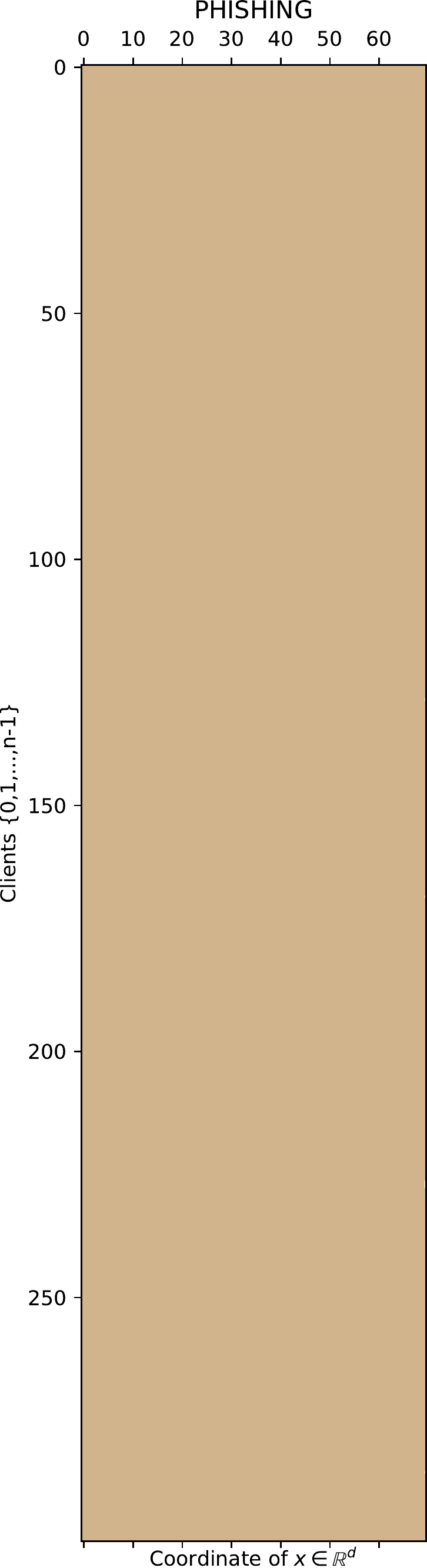} \caption{}
	\end{subfigure}

	\begin{subfigure}[ht]{0.70\textwidth}
		\includegraphics[width=\textwidth]{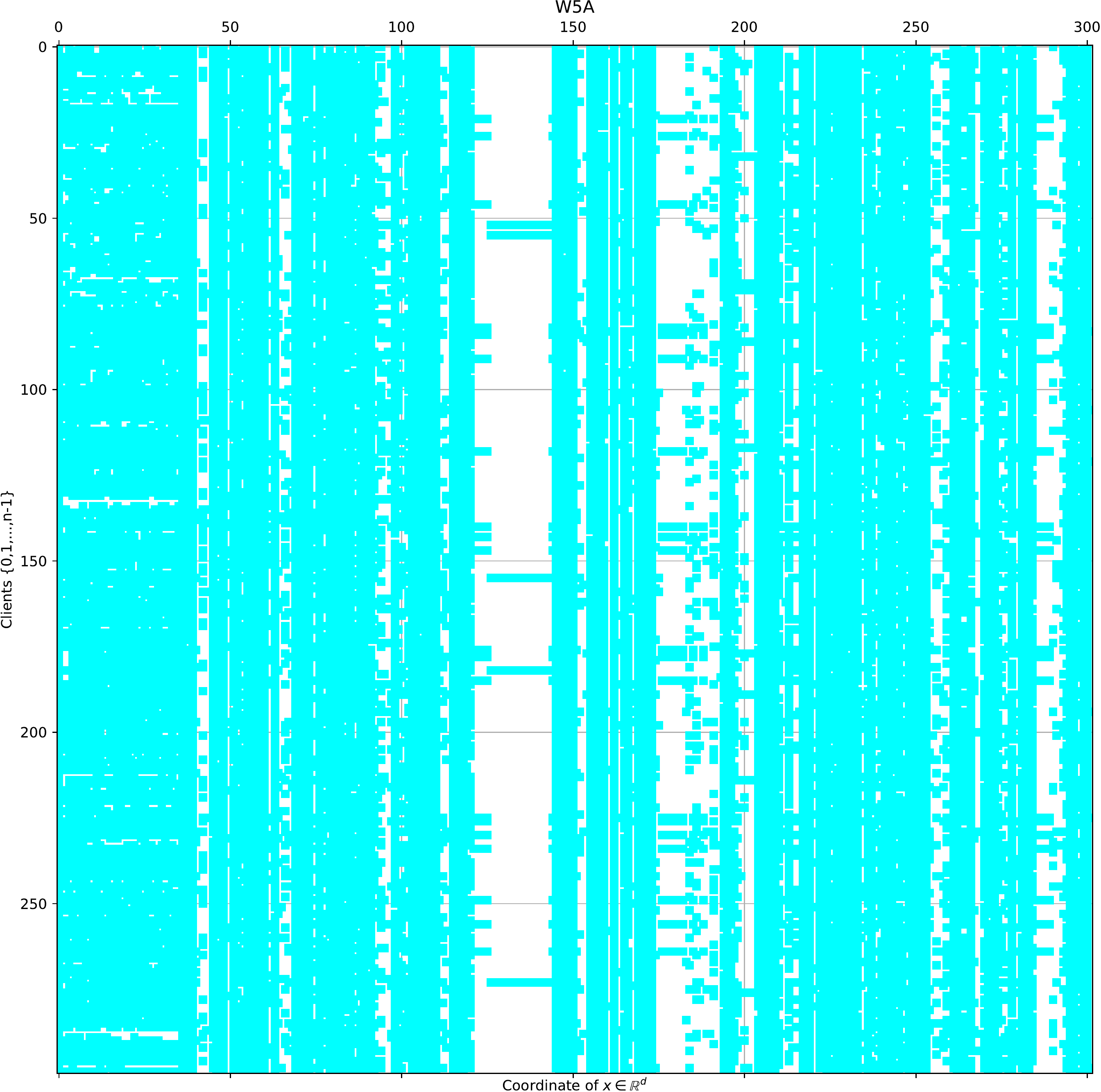} \caption{}
	\end{subfigure}

%	\begin{subfigure}[ht]{0.32\textwidth}
%		\includegraphics[width=\textwidth]{./figs/c-patterns-for-realexp/w3a.pdf} \caption{}
%	\end{subfigure}
%	\begin{subfigure}[ht]{0.45\textwidth}
%		\includegraphics[width=\textwidth]{./figs/c-patterns-for-realexp/w4a.pdf} \caption{}
%	\end{subfigure}

	\caption{Sparsity patterns of datasets for training logistic regression model across $n=300$ clients. An empty cell indicates that a specific client does not have any data for a specific trainable scalar variable $x_i$. The plots show the set $([n] \times [d]) \textbackslash \mathcal{Z}'$ from Section \ref{sec:sparsity}.}
	\label{fig:c_patterns_libsvm_ds_fig}
\end{figure*}

\clearpage
%\newpage
\section{Additional Experiments}
\subsection{Linear regression on sparse data with non-convex regularization}
\label{app:noncvx}

\begin{figure*}[t]
	\centering
	\captionsetup[sub]{font=scriptsize,labelfont={}}	
	\captionsetup[subfigure]{labelformat=empty}
	%\captionsetup{position=top}
	
	\begin{subfigure}[ht]{0.32\textwidth}
		\includegraphics[width=\textwidth]{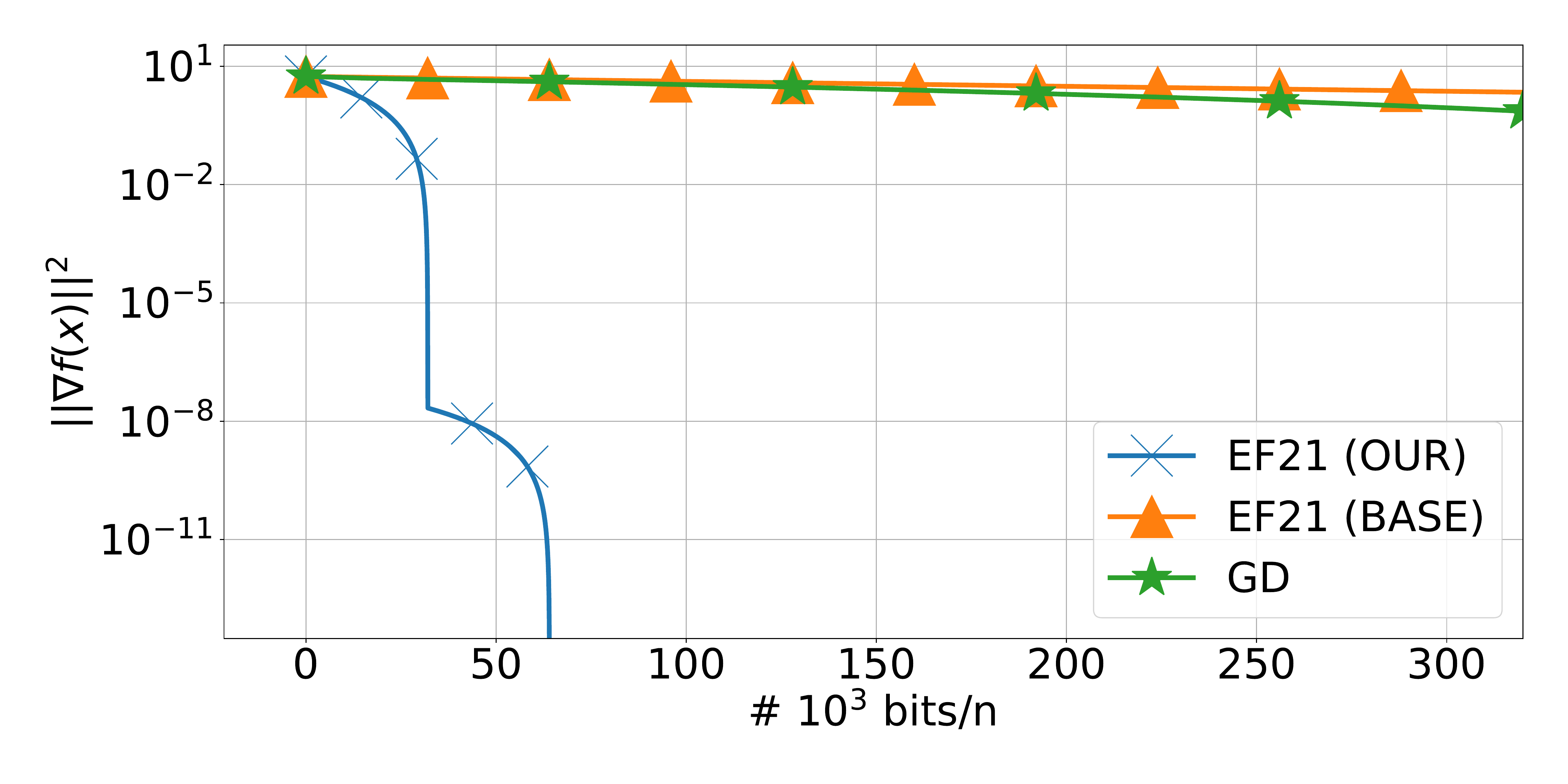} \caption{ $\xc/n=0.05$}
	\end{subfigure}
	\begin{subfigure}[ht]{0.32\textwidth}
		\includegraphics[width=\textwidth]{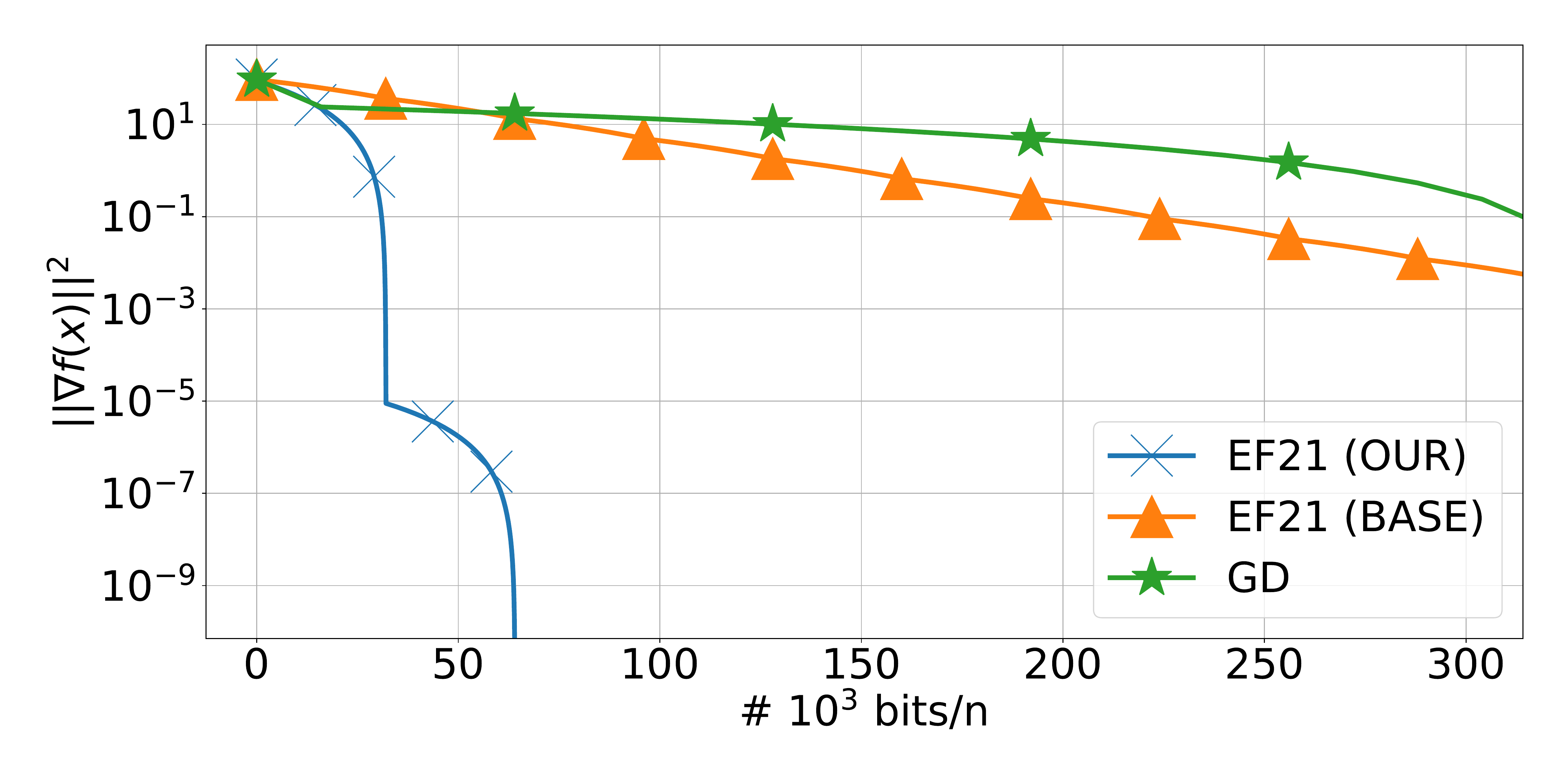} \caption{ $\xc/n=0.5$}
	\end{subfigure}
	\begin{subfigure}[ht]{0.32\textwidth}
		\includegraphics[width=\textwidth]{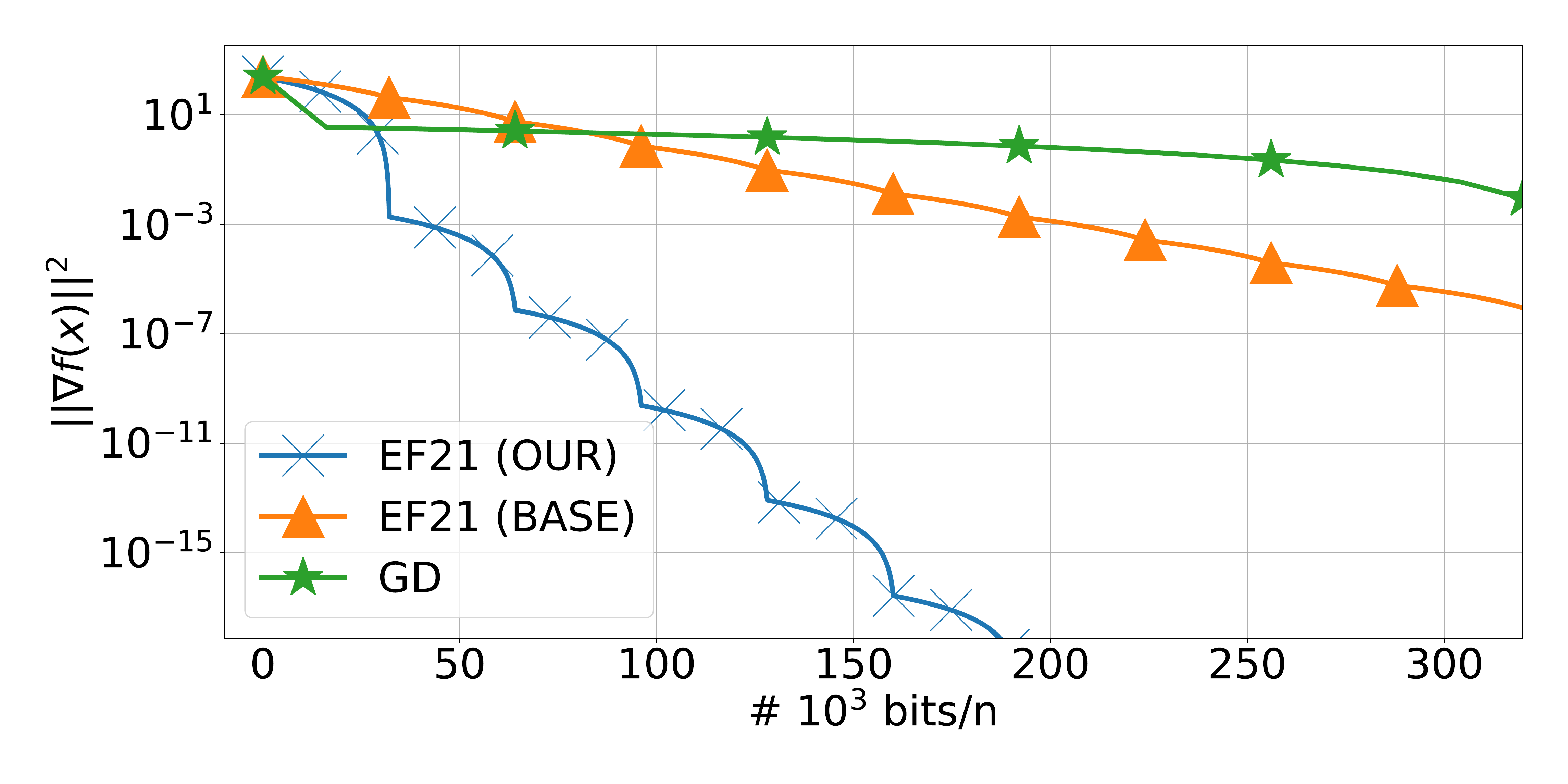} \caption{$\xc/n=0.9$}
	\end{subfigure}
	
	%\vspave{-1pt}
	\caption{Comparison of the performance of \algname{EF21}+\;\topk{1} with the standard approach proposed by \citet{EF21} and a newly proposed stepsize (see~\Cref{thm:convergence_separate}) on the linear regression problem and \algname{GD} for non convex case.  The sparsity pattern $\xc/n$ is changing in controlled way.}
	\label{fig:linear_regression_synthetic_non_cvx}
\end{figure*}

In this experiment, we consider the minimization of the function $f(x)$, which has the following form:
\begin{gather*}
	f(x) = \dfrac{1}{n} \sum_{i=1}^{n} (f_i(x) + \phi_i(x), \\
	f_i(x) = \dfrac{1}{n_i} \norm{\bA_i x - {b_i}}^2, \\
	 \phi_i(x) \eqdef \lambda \cdot \sum_{j=1}^{d} \dfrac{x_j^2}{x_j^2 + 1} \cdot I_{ij},\\
	 I_{ij} \eqdef \begin{cases}
	 	1, &\text{if } f_i(x) \text{ depends on } x_j,\\
	 	0, &\text{otherwise.}
					 \end{cases}
\end{gather*}
The term $\phi(x)$ is a non-convex regularizer, with eigenvalues in the set $\left[-\frac{1}{2}, 2\right]$. We use the dataset generation algorithm described in Section \ref{app:lin_reg_on_sparse}. It also explains the main meta-parameter that controls the generation of optimization problem instances such as $\xc/n$. After the generation process, the quadratic part depends only on a subset of variables $(x_1, \dots, x_d)$. Unlike in Section \ref{app:lin_reg_on_sparse}, we also restrict $\phi_i(x)$ to depend only on that subset of variables. This is why we include indicator variables $I_{ij}$ in the general formulation. We set the regularization coefficient $\alpha = 3 \cdot \max (L_{f_i})$ to make $f_i(x)$ non-convex smooth functions.

We present the results for varying $\xc/n$ in Figure \ref{fig:linear_regression_synthetic_non_cvx}. As we can see, $\xc/n$ plays a crucial role in the convergence of \algname{EF21}. Specifically, as $\xc/n$ goes to zero, our new analysis allows us to increase the step size of \algname{EF21}. As $\xc/n$ goes to one, the step size of \algname{EF21} becomes more similar to the standard one.
We use the meta-parameter $v=0.1$, which enforces the condition from Lemma \ref{lem:M-bound-via-c} and which is favorable for our algorithm and analysis in the context of Main Theorem \ref{thm:convergence_separate}.

\subsection{\algname{EF21}  versus \algname{GD} on convex objective}

\begin{figure*}[t]
	\centering
	\captionsetup[sub]{font=scriptsize,labelfont={}}	
	\captionsetup[subfigure]{labelformat=empty}
	%\captionsetup{position=top}
	
	\begin{subfigure}[ht]{0.32\textwidth}
		\includegraphics[width=\textwidth]{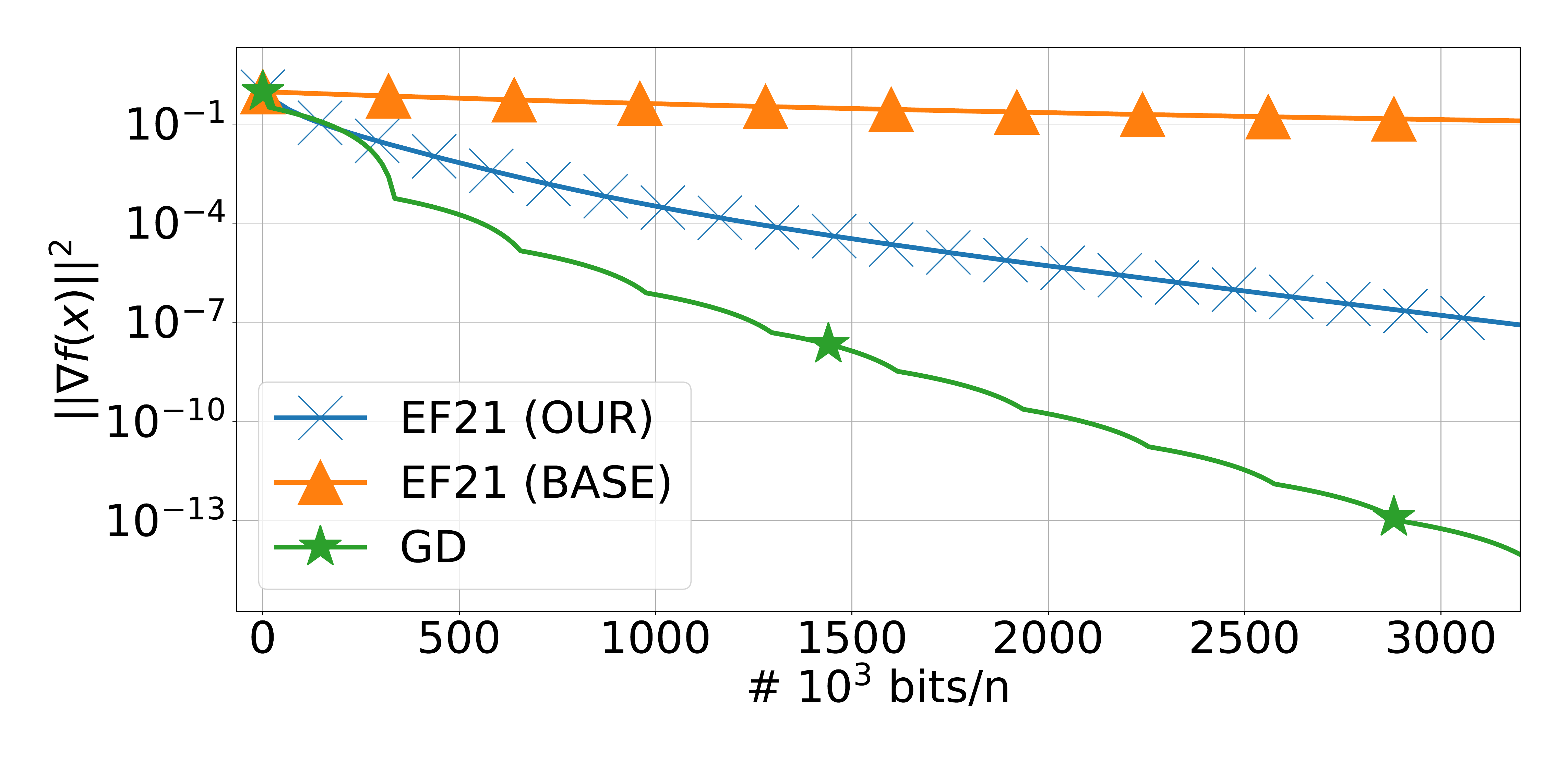} \caption{ $\xc/n=0.05$}
	\end{subfigure}
	\begin{subfigure}[ht]{0.32\textwidth}
		\includegraphics[width=\textwidth]{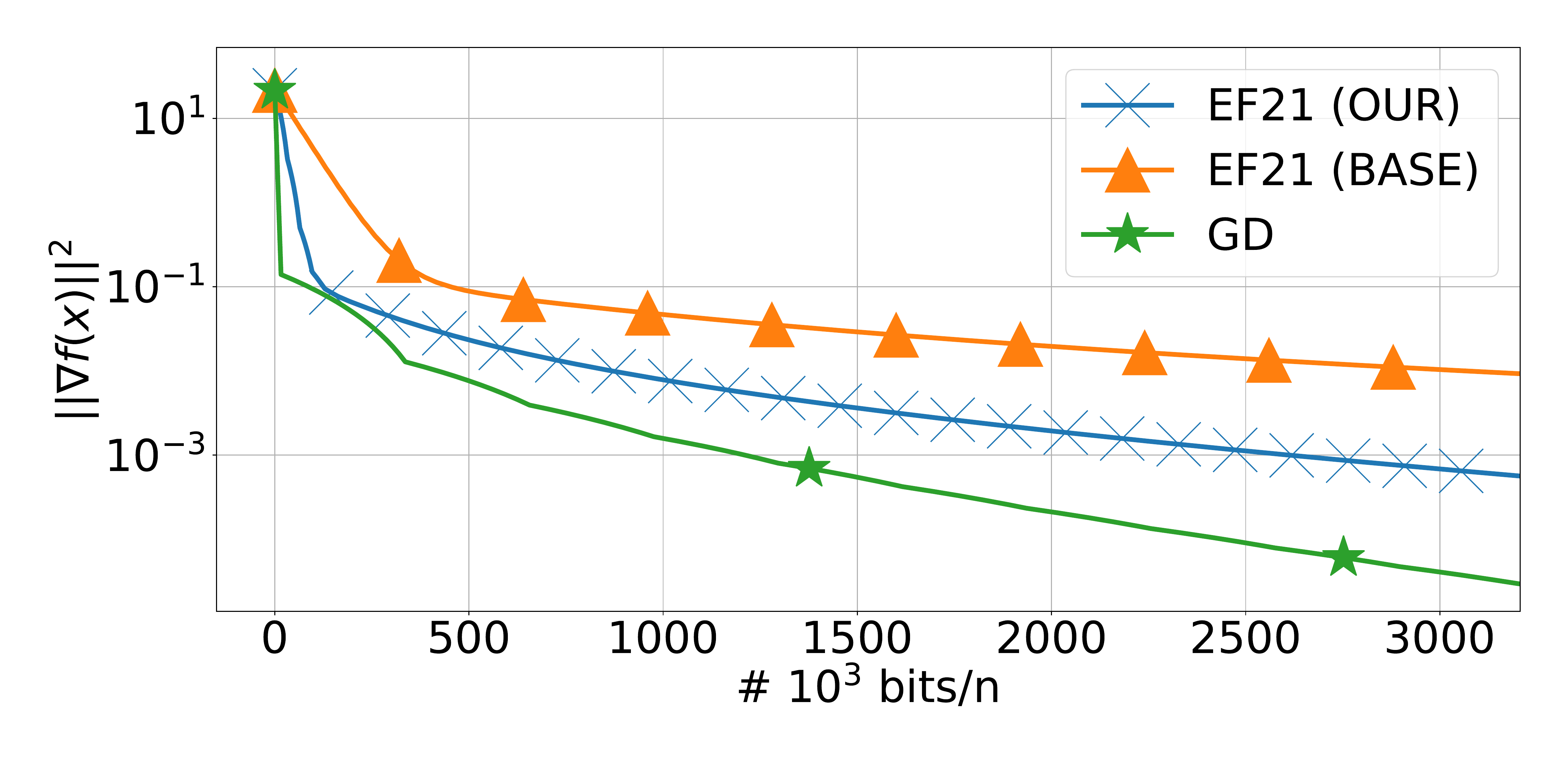} \caption{ $\xc/n=0.5$}
	\end{subfigure}
	\begin{subfigure}[ht]{0.32\textwidth}
		\includegraphics[width=\textwidth]{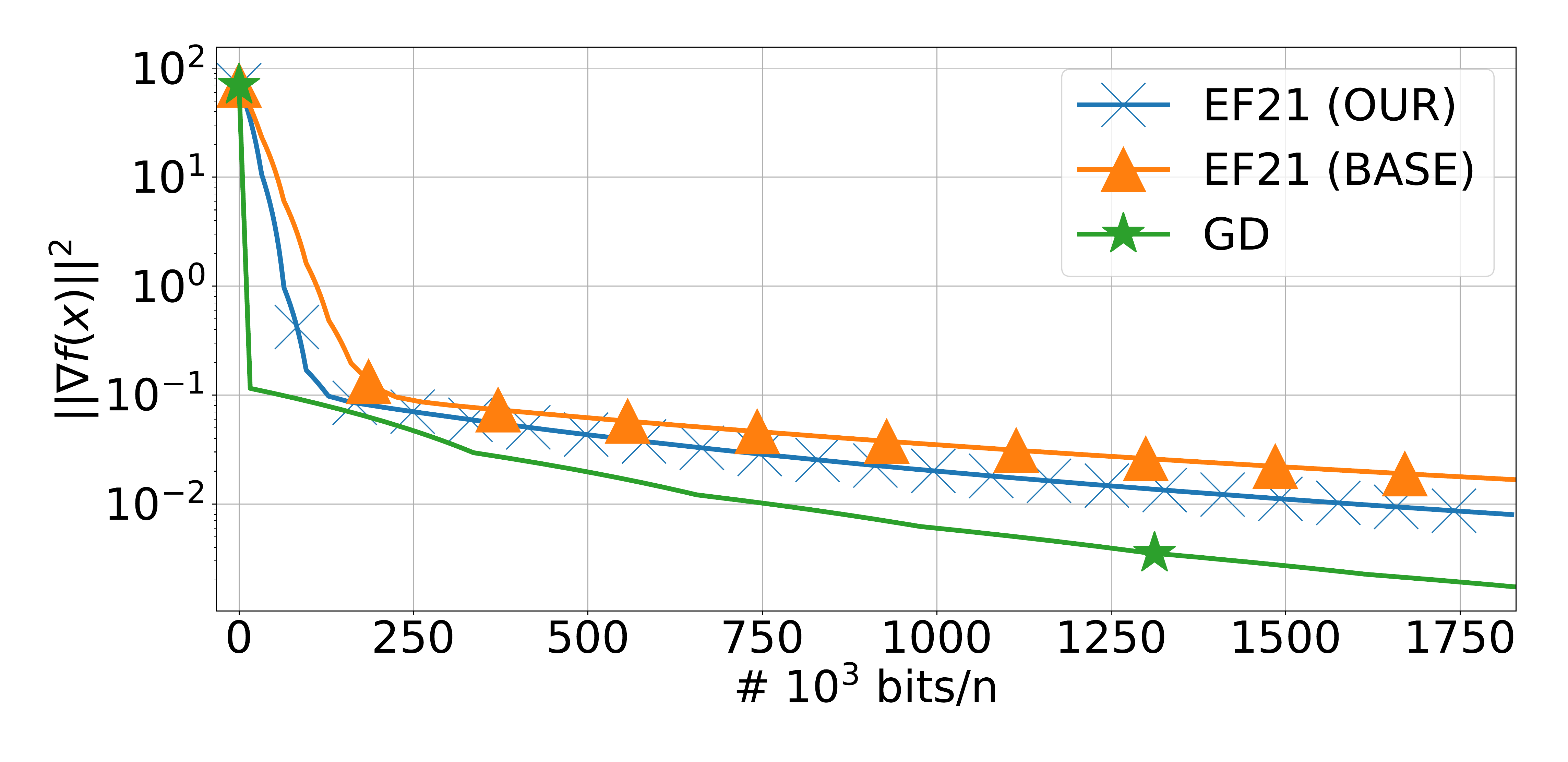} \caption{$\xc/n=0.9$}
	\end{subfigure}
		
	%\vspave{-1pt}
	\caption{Comparison of the performance of \algname{EF21}+\;\topk{1} with the standard approach proposed by \citet{EF21} and \algname{GD} on the linear regression problem.  The optimization objective  is convex. The sparsity pattern $\xc$ is controlled by manipulating the sparsity of the data. }
	\label{fig:linear_regression_synthetic_with_gd}
\end{figure*}

In Section \ref{sec:exp_lin_reg_with sparse_data} we have observed that we gain improvements for \algname{EF21} with our analysis compared to standard analysis. In this section, we add a comparison of these algorithms in addition to \algname{GD}. For all algorithms, we have used theoretical step size. Results presented in Figure \ref{fig:linear_regression_synthetic_with_gd}.

As we see \algname{GD} behaves better compared to the classical analysis of \algname{EF21} and our proposed analysis if do not tune step sizes. An important caveat is that \algname{EF21} is an algorithm for non-convex optimization and this setting is convex. In conclusion, we see to apply \algname{EF21} in a convex setting with theoretical step sizes there is room for future research.

% k=\mathcal{O}(1/\varepsilon^2)$ attains $\norm{\nabla f(x^k)}^2 \le \varepsilon^2$, but for convex smooth case we have that after $k\ge \frac{1}{2 \gamma \varepsilon}$ the \algname{GD} with fixed step size attains $\dfrac{1}{2L} \|\nabla f(x^k)\|^2 \le f(x^k) - f^* \le \varepsilon \|x^0 - x^*\|^2$. So convergence in $\|\nabla f(x^k)\|^2 = \mathcal{O}(1/k)$. 

%%%%%%%%%%%%%%%%%%%%
%%%%%%%%%%%%%%%%%%%%
%%%%%%%%%%%%%%%%%%%%
\section{Limitations} 
%%%%%%%%%%%%%%%%%%%%
%%%%%%%%%%%%%%%%%%%%
%%%%%%%%%%%%%%%%%%%%

We acknowledge that  the practical applicability of our results is limited because, firstly, many real-world datasets are not sparse enough to enjoy a significantly small $\xc/n$ ratio, and secondly, the complex architecture of deep neural networks already has non-zero weights on the second layer, independent of the initial dataset's sparsity, further increasing the value of $\xc$.

\end{document}